\documentclass{ws-m3as}
\usepackage{graphicx,color,comment}
\definecolor{darkgreen}{rgb}{0.,0.5,0.0}
\definecolor{darkblue}{rgb}{0.,0.0,0.5}
\newcommand{\ul}[1]{\underline{#1}}
\newcommand{\uld}[1]{\underline{d#1}}

\newcommand{\skipsimple}[1]{}
\newcommand{\kspace}[1]{{#1}}
\newcommand{\allximgDzero}[1]{}
\newcommand\xred{\check{x}}
\newcommand\Fred{\check{F}}
\renewcommand\Re{\mathrm{Re}}
\renewcommand\Im{\mathrm{Im}}
\def\Mpp{M_{\perp\!\!\!\!\!\bigcirc}}
\def\Mpptil{\tilde{M}_{\perp\!\!\!\!\!\bigcirc}}
\def\Mphp{\widehat{\Mpp}}
\def\pp{p_{\perp\!\!\!\!\!\bigcirc}}
\def\pptil{\tilde{p}_{\perp\!\!\!\!\!\bigcirc}}
\def\uldMpp{\uld{M}_{\perp\!\!\!\!\!\bigcirc}}
\def\ulddMpp{\uld{dM}_{\perp\!\!\!\!\!\bigcirc}}
\def\yp{y_{\perp\!\!\!\!\!\bigcirc}}
\def\fpp{f_{\perp\!\!\!\!\!\bigcirc}}

\def\bb{b}
\newcommand\Xpar{\mathbb{X}_{\text{par}}}
\newcommand\taup{\varepsilon}
\newcommand\tauone{\tau^{(II)}}
\newcommand\tautwo{\tau^{(III)}}
\newcommand\tII{t_{II}}
\newcommand\tIII{t_{III}}
\newcommand{\Dterm}[1]{\|D_0^{1/2}{#1}\|^2}
\definecolor{darkgreen}{rgb}{0.0,0.7,0.0}
\newcommand{\eexp}[1]{\text{exp}\left(#1\right)}

\newcommand{\revision}[1]{{ #1}}
\newcommand{\revisionown}[1]{{ #1}}
\newcommand{\Margin}[1]{}

\begin{document}

\markboth{Barbara Kaltenbacher}{Uniqueness of coefficient identification in the Bloch-Torrey equation}

%%%%%%%%%%%%%%%%%%% Publisher's Area please ignore %%%%%%%%%%%%%%%%%%%%%%%
%
\catchline{}{}{}{}{}
%
%%%%%%%%%%%%%%%%%%%%%%%%%%%%%%%%%%%%%%%%%%%%%%%%%%%%%%%%%%%%%%%%%%%%%%%%%%
\title{On uniqueness of coefficient identification in the Bloch-Torrey equation for magnetic resonance imaging}
\author{Barbara Kaltenbacher}
\address{Department of Mathematics, University of Klagenfurt, Universit\"atsstra\ss e 65-67, 9020 Klagenfurt, Austria\\
barbara.kaltenbacher@aau.at}

\maketitle

\begin{history}
\received{(Day Month Year)}
\revised{(Day Month Year)}
%\accepted{(Day Month Year)}
\comby{(xxxxxxxxxx)}
\end{history}

\begin{abstract}
In this paper we provide some uniqueness results for the (multi-)coefficient identification problem of reconstructing the spatially varying spin density as well as the spin-lattice and spin-spin relaxation times and the local field inhomogeneity in the Bloch-Torrey equation, 
as relevant in magnetic resonance imaging MRI. To this end, we follow two approaches: (a) Relying on sampling of the k-space and (approximately) explicit reconstruction formulas in the simplified (Bloch) ODE setting, along with perturbation estimates; (b) Relying on infinite speed of propagation due to diffusion. The results on well-posedness and Lipschitz continuous differentiability of the coefficient-to-state map derived for this purpose, are expected to be useful also in the convergence analysis of reconstruction schemes as well in mathematical optimization of the experimental design in MRI.   
\end{abstract}



\section{Introduction}

Whereas MRI is conventionally viewed as reconstruction from a sampled version of the Fourier transform, based on model simplifications, we here view it as an inverse problem of coefficient identification in a partial differential equation PDE -- the Bloch-Torrey equation, whose  ODE version for vanishing diffusion is the Bloch equation see, e.g., \cite{bloch46,torrey56,Nishimura2010}.
This approach is known as model based reconstruction in the physics and engineering literature, see \cite{Uecker_Review:2021} and the references therein.

The fundamental quantity to be identified in MRI is the spin density, considered as a spatially varying function on the space domain $\Omega$ where the Bloch-Torrey PDE holds. Quite often, also the spin-lattice and spin-spin relaxation times are reconstructed, in view of their additional diagnostic value.  
These (and additionally the field inhomogeneity, which can be viewed as the imaginary part of an effective  reciprocal spin-spin relaxation time) are the unknowns in the inverse problem studied also here, which, in the above mentioned PDE context, amounts to multi-coefficient identification.
On top of these imaging quantities, 
our long term goal is to also study
reconstruction of the diffusion tensor and the velocity in a convective term, that is additionally incorporated in the Bloch-Torrey equation to model motion or flow.
%, as part of forthcoming work.

Even in the idealized setting of reconstruction from a sampled Fourier transform, due to the fact that sampling takes place along a smooth curve rather than all of the three dimensional frequency space, uniqueness of the spin density (and other coefficients) cannot be expected to hold in all of $L^2(\Omega)$ but in a (possibly still infinite dimensional) subspace $\Xpar$.
Consistently with intuition, the more exhaustive the sampling is done, the larger will be the subspace $\Xpar$ on which uniqueness holds.

With this in mind, we derive two types of uniqueness results here
\begin{itemize}
\item Relying on smallness of diffusion, we consider the Bloch-Torrey PDE model as a perturbation of the Bloch ODE, for which (approximate) solution formulas result in a well-posed inverse problem. Substantiating this perturbation argument in appropriate function spaces, we prove well-posedness of the nonlinear PDE inverse problem of reconstructing spin density 
\allximgDzero{and relaxation times} 
and spin-lattice relaxation time 
for sufficiently close to constant 
(with respect to the $L^\infty(\Omega)$ topology) diffusion.
\revision{The challenge here lies in the transition from a pointwise ODE to a parabolic PDE; we achieve this by considering the problem in Fourier domain.}
As a by-product we obtain local quadratic convergence of Newton's method for the multicoefficient identification problem.
\item In case of possibly larger diffusion, relying on infinite speed of propagation, rather than maximizing coverage of the sampling procedure, we extend \revision{ recently developed } methodology for uniqueness in backwards/sideways parabolic problems, cf.~\cite{JiangLiPauronYamamoto2023}
%, e.g., \cite{BukhgeimKlibanov:1981,ImanuvilovYamamoto:1998,Isakov:2006} 
to the Bloch-Torrey equation setting, to prove linearized uniqueness of spin density and relaxation times. 
The uniqueness results obtained there come without stability estimates, as to be expected in view of the severe ill-posedness of backwards and sideways diffusion.   
\end{itemize}
\Margin{R1 0.}

In Remark~\ref{Kaczmarz90deg} we also indicate a possibility of combining both approaches, whose detailed investigation is subject of future research.

Clearly, the quality of reconstruction always depends on the choice of an appropriate excitation; in our present context, this is the radio frequency field which is subject to pulsed changes.
There is a small caveat on what the right function space setting is for describing such pulses: While considering delta distributions in the space $\mathcal{M}(0,T)$ of Radon measures appears to be both intuitive and advantageous due to the favorable properties of this space as the dual of the separable space $C(0,T)$, it turns out that the PDE estimates on which our perturbation results rely, can only handle $L^1(0,T)$ excitations. 
Thus we stay with rectangular pulses that are in fact quite common in the MR literature.
In fact, the uniqueness results we obtain here rely on elementary ``textbook'' pulse sequences, that is, they do not require a sophisticated choice of the excitation.
Thus we expect that -- starting from these fundamental uniqueness results -- there is ample room for improvement of reconstruction efficiency and quality by a combination of mathematical optimization with expert MRI knowledge.
\revision{This corresponds to an optimal control task to be addressed in future work.}
\Margin{R2 (2)}
\footnote{We plan to work on this in the collaborative research center MR-DYNAMO; 
%(Mathematics of Reconstruction in Dynamical and Active Models); 
see the funding acknowledgment.}

As an ingredient for the uniqueness proofs \revisionown{that form the core of this paper}, but also as a preparation for various  optimization tasks, as well as for convergence proofs of regularization methods for the inverse problem under consideration,
this paper also contains some well-posedness and regularity results for the forward problem of solving the convective Bloch-Torrey equation with initial and boundary conditions. 

\medskip

The remainder of this paper is organized as follows.
In Subsection~\ref{sec:setting} we describe the underlying model and introduce the (many) physical quantities involved.
Section~\ref{sec:forward} is devoted to results on the forward problem. 
The outcome of Section~\ref{sec:wellposed} allows us to establish well-definedness and continuous differentiability of the parameter-to-state map $\mathcal{S}$ in Section~\ref{subsec:S}.
In Section~\ref{sec:explsol} we derive explicit solution formulas for the Bloch ODE case with elementary rectangular radio frequency pulses and equip them with estimates of the error due to (a) finiteness of the pulse duration (b) neglection of diffusion.
This preparatory work allows us to prove the 
\revision{uniqueness}
\Margin{R2 1)} 
results under the two paradigms alluded to above in Subsections~\ref{sec:uniqueness_perturbation} and \ref{sec:uniqueness_diffusion} of Section~\ref{sec:uniqueness}. 
While these constitute the main outcome of this paper, as already mentioned, we also expect the other analytical results obtained in Section~\ref{sec:forward} to be useful for several purposes in the context of the underlying inverse problem. Vice versa, linearized uniqueness, as established in Subsection~\ref{sec:uniqueness_diffusion} will be a crucial ingredient for proving convergence of certain Newton type reconstruction methods.  

\medskip

\subsubsection*{Notation}
\begin{itemize}
\item (space-)constant versions of otherwise spatially variable quantities are often indicated by a subscript ${}_0$ here.
\item The usual Sobolev (Hilbert) spaces over a spatial domain $\Omega$ are denoted by $H^s(\Omega)$, $s\in(0,\infty)$; 
For describing time and space dependent functions, we will employ Bochner spaces such as 
%and make use of the abbreviations
%$L^p(X):=$
$L^p(0,T;H^s(\Omega))$, 
%$H^r(X):=$
$H^r(0,T;H^s(\Omega))$.
\item In most of this paper we will explicitly write out space and time dependence of functions. We are fully aware of the fact that this is not very common in PDE analysis papers; however, in the present context such an explicit distinction of dependencies appears to be essential for capturing the structure of the problem.
\item Even more unusual for a mathematical paper, we are using $\vec r$ for the space variable.
This is not only an adoption of engineering/physics notation, but also allows us to use the letter $x$ for unmistakably referring to the inverse problems solution.  
%\footnote{In case the reviewers find this too exotic/confusing I am of course happy to change the notation to something more common in mathematics.}
\end{itemize}

\subsection{Problem setting}\label{sec:setting}
The fundamental \textbf{PDE model} describing the physics underlying magnetic resonance imaging is the 
Bloch-Torrey equation for the magnetization $\vec M$,
%cf. \cite[pages 10-11]{Uecker2009NonlinearRM} 
see, e.g., \cite{bloch46,torrey56,Nishimura2010}.
We here consider a convective version 
\begin{equation}\label{eqn:bloch-torrey}
\begin{aligned}
    \frac{d}{dt} \vec M(t, \vec r) = \gamma \vec M(t, \vec r) \times  
\begin{pmatrix} B^1_{\mathsf{x}}(t, \vec r) \\ B^1_{\mathsf{y}}(t, \vec r) \\ B^1_{\mathsf{z}}(t, \vec r)  \end{pmatrix} 
    +  \begin{pmatrix} \frac{-M_{\mathsf{x}}(t, \vec r)}{T_2(\vec r)} \\ \frac{-M_{\mathsf{y}}(t, \vec r)}{T_2(\vec r)} \\ \frac{M^{eq}(\vec r) - M_{\mathsf{z}}(t, \vec r)}{T_1(\vec r)} \end{pmatrix}\\
    +\nabla\cdot\Bigl(D(\vec r) \, \nabla \vec M(t, \vec r)\Bigr) 
    -(\vec v(t,\vec r)\cdot\nabla)\vec M(t,\vec r)
\end{aligned}
\end{equation}
containing a velocity term $(\vec v(t,\vec r)\cdot\nabla)\vec M(t,\vec r)$ that allows to describe mechanical displacements (e.g., due to patient motion or blood flow) inside the imaged object.
In \eqref{eqn:bloch-torrey}, the first term on the right hand side describes precession driven by an imposed magnetic field with flux density $B^1(t, \vec r)$, and the second one relaxation towards the equilibrium magnetization, that is conventionally assumed to be aligned in $z$-direction. 
Diffusion -- which distinguishes the Bloch-Torrey PDE from the Bloch ODE model -- is possibly anisotropic but typically acts the same way on all components of the magnetization 
\[
\begin{aligned}
\nabla\cdot\Bigl(D \, \nabla \vec M)&:=\Bigl(\nabla\cdot\Bigl(D \, \nabla M_j)\Bigr)_{j\in\{1,2,3\}}
%\text{ i.e., } \\
%\nabla\cdot\Bigl(D \, \nabla \vec M)&=
%\Bigl(\sum_{k=1}^3 \frac{\partial}{\partial x_k}\bigl(\sum_{n=1}^3 D_{kn} \frac{\partial}{\partial x_n} M_j)\bigr)\Bigr)_{j\in\{1,2,3\}}\\
%&=\Bigl(\sum_{k=1}^3 \frac{\partial}{\partial x_k}\bigl(\sum_{m=1}^3\sum_{n=1}^3\tilde{D}_{k j m n} \frac{\partial}{\partial x_n} M_m)\bigr)\Bigr)_{j\in\{1,2,3\}} \text{ with } \tilde{D}_{k j m n} = \delta_{jm} D_{kn}
\end{aligned}
\]
with a symmetric positive definite diffusion tensor $D:\mathbb{R}^3\mapsto \mathbb{R}^{3\times 3}$.
Further quantities appearing in \eqref{eqn:bloch-torrey} are
\Margin{R2 2)}
\begin{equation}\label{Meq-rho}
\begin{aligned}
B^1_{\mathsf{z}}(t, \vec r) =&B^0 + \delta B^0(\vec r) +   \vec r \cdot \vec G(t)\hspace*{-4cm}&& \\
\text{\revision{where }}&B^0 \quad&& \text{\ldots static magnetic field}\\
&\delta B^0(\vec r)\quad&& \text{\ldots local field inhomogeneity}\\
&\vec G(t)\quad&& \text{\ldots field gradient}\\
M^{eq}(\vec r)=& \rho(\vec r)\frac{\gamma^2 \hbar^2}{4k_B\vartheta}B^0  && \text{\ldots equilibrium magnetization}\\
\text{\revision{where }}&\rho(\vec r) \quad&& \text{\ldots spin density}\\
&\gamma  \quad&& \text{\ldots gyromagnetic ratio}\\
&\hbar  \quad&& \text{\ldots reduced Planck constant}\\
&k_B  \quad&& \text{\ldots Boltzmann constant}\\
&\vartheta  \quad&& \text{\ldots temperature}\\
T_1(\vec r)\quad& \quad &&\text{\ldots spin-lattice relaxation time } \\
T_2(\vec r)\quad& \quad &&\text{\ldots spin-spin relaxation time } 
\end{aligned}
\end{equation}
The latter usually satisfy $T_2(\vec r)\leq T_1(\vec r)$, % page 11 in [ThesisUecker] 
cf., e.g., \cite[pages 10-11]{Uecker2009NonlinearRM}. 

Partition of the magnetization vector $\vec M$ into a longitudinal $M_z$ and a transversal $M_{\perp}$ component is notationally conveniently captured by identification of the $x-y$ plane with the complex plane, considering $M_\perp$ and the transversal part of
$B^1$ to be complex valued, 
\[
M_{\perp}(t, \vec r) := M_{\mathsf{x}}(t, \vec r) + i M_{\mathsf{y}}(t, \vec r), 
\]
while all other quantities stay real valued. 
Moreover, the local field typically has the space-time separable structure
\[
\begin{aligned}
&\begin{pmatrix} B^1_{\mathsf{x}}(t, \vec r) \\ B^1_{\mathsf{y}}(t, \vec r)\end{pmatrix}
= \begin{pmatrix} (c^+(\vec r) p(t))_{\mathsf{x}}\\(c^+(\vec r) p(t))_{\mathsf{y}}\end{pmatrix}
= \begin{pmatrix} \Re(c^+(\vec r) p(t))\\\Im(c^+(\vec r) p(t))\end{pmatrix}
\end{aligned}
\]
where both $c^+$ and $p$ are complex valued
\[
\begin{aligned}
&p(t) := p_{\mathsf{x}}(t) + i p_{\mathsf{y}}(t) \quad&& \text{\ldots radio frequency RF field}\\
&c^+(\vec r) := c^+_{\mathsf{x}}(\vec r) + i c^+_{\mathsf{y}}(\vec r) \quad&& \text{\ldots spatial
dependency of the transmitted field. 
%by the scanner hardware
}
\end{aligned}
\]

Using the material derivative notation
\[
\frac{D}{Dt} f(t, \vec r) =\frac{d}{dt} f(t, \vec r) + \vec v(t,\vec r)\cdot\nabla f(t,\vec r)
\]
as well as the abbreviations 
\begin{equation}\label{R2star}
\begin{aligned}
&R_1(\vec r)=\frac{1}{T_1(\vec r)}, \quad R_2(\vec r)=\frac{1}{T_2(\vec r)}&&\text{\ldots reciprocals of relaxation times}\\
&\omega_0=\gamma B^0 &&\text{\ldots Larmor frequency, }\\
&R_2^*(\vec r):=R_2(\vec r)+i\gamma\,\delta B^0 (\vec r) 
&&\text{\ldots reciprocal of effective}\\ 
&&&\text{\phantom{\ldots} spin-spin relaxation time }
\end{aligned}
\end{equation}
we can rewrite \eqref{eqn:bloch-torrey} as\footnote{$\Re(a+ib)=a$ denotes the real part, $\overline{a+ib}=a-ib$ the complex conjugate of a complex number $a+ib$.}
\begin{equation}\label{eqn:bloch-torrey-complex}
\begin{aligned}
&\frac{D}{Dt}M_\perp(t, \vec r) -\nabla\cdot\bigl(D(\vec r) \nabla M_\perp(t, \vec r)\bigr) +\bigl(R_2^*(\vec r)+i\omega_0  + i\gamma\vec r \cdot \vec G(t)\bigr)M_\perp(t, \vec r)\\
&\qquad-i\gamma M_z(t, \vec r)c^+(\vec r)p(t)=0\\    
&\frac{D}{Dt}M_z(t, \vec r) -\nabla\cdot\bigl(D(\vec r) \nabla M_z(t, \vec r)\bigr) +R_1(\vec r) (M_z(t, \vec r)-M^{eq}(\vec r))
\\
&\qquad-\gamma\Re\bigl(iM_\perp(t, \vec r)\overline{c^+(\vec r)p(t)}\bigr)=0.
\end{aligned}
\end{equation}

To remove the high frequency factor, this system is usually considered in the rotating frame by defining 
\begin{equation}\label{Mperp_prime}
\Mpp(t, \vec r)=e^{i\omega_0t}M_\perp(t, \vec r)
\end{equation}
which upon multiplication of the transversal part of \eqref{eqn:bloch-torrey-complex} with $e^{i\omega_0t}$, 
\Margin{R2 3)}
\revision{under the commonly made assumption} that 
$p(t)$ takes the form $p(t)=e^{-i\omega_0t}\pp(t)$, 
%with some envelope function $\pp(t)$,  
\revisionown{and re-defining $\frac{D}{Dt}=\frac{d}{dt} + \widetilde{\vec v}\cdot\nabla$ where $\widetilde{\vec v}(t, \vec r)=(e^{i\omega_0t}v_\perp(t, \vec r),\,v_z(t, \vec r))$,} 
satisfies
\begin{equation}\label{eqn:bloch-torrey-complex-rotated}
\begin{aligned}
&\frac{D}{Dt}\Mpp(t, \vec r) -\nabla\cdot\bigl(D(\vec r) \nabla \Mpp(t, \vec r)\bigr) +\bigl(R_2^*(\vec r) + i\gamma\vec r \cdot \vec G(t)\bigr)\Mpp(t, \vec r)\\
&\qquad-i\gamma  M_z(t, \vec r) \, c^+(\vec r)\pp(t)
=0
\\    
&\frac{D}{Dt}M_z(t, \vec r) -\nabla\cdot\bigl(D(\vec r) \nabla M_z(t, \vec r)\bigr) 
+R_1(\vec r) (M_z(t, \vec r)-M^{eq}(\vec r))
\\
&\qquad-\gamma\Re\bigl(i \Mpp(t, \vec r)\overline{c^+(\vec r)\pp(t)}\bigr)
=0.
\end{aligned}
\end{equation}
Commonly used envelope functions are 
\begin{itemize}
\item continuous $\pp(t)\equiv {\pp}_0$
\item rectangular pulse $\pp(t)={\pp}_0\,\tau_p^{-1} 1\!\mathrm{I}_{[t_p,t_p+\tau_p]}(t)$
\item sinc pulse $\pp(t)={\pp}_0\,\frac{\sin(t/\tau_p)}{t}$
\end{itemize}
cf. e.g., \cite[page 10]{Uecker2009NonlinearRM}, %see also page 18 of MRMB-03-ExcRelax.pdf,
and we will particularly rely on rectangular pulses here.
To emphasize smallness of the pulse duration, we will denote it by $\taup$ in the following.

The PDE models described above are parabolic -- potentially 
\revision{degenerate}
\Margin{R2 4)} 
in case of locally vanishing diffusion -- and are assumed to hold on a space-time domain $(0,T)\times\Omega\subseteq\mathbb{R}^3$ equipped with initial and boundary (in case of bounded $\Omega$) conditions.

In the (space-)constant coefficient case 
%and full space case 
\begin{equation}\label{fullspace}
\vec v(t,\vec r)=\vec v_0(t)%\in\mathbb{R}^3
, \quad 
R_1(\vec r)=R_{1,0}%\in\mathbb{R}
, \quad 
R_2^*(\vec r)=R_{2,0}^*%\in\mathbb{C}
, \quad
c^+(\vec r)=c^+_0,
\qquad\Omega=\mathbb{R}^3,
\end{equation}
extending by zero to all of $\mathbb{R}^3$, (cf. Remark~\ref{rem:transform_kspace} below,) 
we can take the Fourier transform with respect to space 
\begin{equation}\label{Fouriertransf}
(\mathcal{F}v)(\vec \xi)=\int_{\mathbb{R}^3} v(\vec r) \, e^{-2\pi i {\vec \xi}\cdot{\vec r}}  \,\textsf{d}\vec r
\end{equation}
and define $\Mphp(t,\vec\xi)=\mathcal{F}[\Mpp](t,\xi)$,
to obtain
\begin{equation}\label{eqn:bloch-torrey-complex_rot_FT}
\begin{aligned}
&\frac{d}{dt}\Mphp(t,\vec\xi) - \frac{\gamma}{2\pi}\,\vec G(t)\cdot \nabla_\xi \Mphp(t,\vec\xi)
+\bigl(R_{2,0}^*+\vec v_0(t)\cdot\xi+\xi\cdot D_0\xi \bigr)\Mphp(t,\vec\xi)\\
&\qquad-i\gamma \widehat{M_z}(t, \vec\xi)c^+_0 \pp(t)=0\\    
&\frac{d}{dt}\widehat{M_z}(t, \vec\xi) 
+\bigl(R_{1,0}+\vec v_0(t)\cdot\xi+\xi\cdot D_0\xi\bigr)\widehat{M_z}(t, \vec\xi)-R_{1,0}\widehat{M^{eq}}(\vec\xi))\\
&\qquad-\gamma\Re\bigl(i\Mphp(t, \vec\xi)\overline{c^+_0 \pp(t)}\bigr)=0.
\end{aligned}
\end{equation}
This is a variable coefficient linear transport equation (that is, a linear first order PDE) that can be solved by the method of characteristics, see Section~\ref{sec:explsol} below.
As common in the physics and engineering literature, we adopt the term ``k-space'' for the frequency domain with respect to space and use the $\vec r$ notation for the space variable (so that the letter $x$ will be available for the inverse problem solution).

The above PDE models contain the following \textbf{imaging quantities} as space-dependent coefficients
\begin{itemize}
\item spin density $\rho(\vec r)$; or equivalently, \footnote{since $\frac{\gamma^2 h^2}{4k\vartheta}B^0$ in \eqref{Meq-rho} is known,}
equilibrium magnetization $M^{eq}(\vec r)$; 
\item spin-lattice / spin-spin relaxation times $T_1(\vec r)$ / $T_2(\vec r)$ or equivalently, their reciprocals 
$R_1(\vec r)=\frac{1}{T_1(\vec r)}$, $R_2(\vec r)=\frac{1}{T_2(\vec r)}$;
\item field inhomogeneity $\delta B^0(\vec r)$ ;
\item diffusion 
%coefficients $D_{\perp}(\vec r)$, $D_{\mathsf{z}}(\vec r)$
tensor $D(\vec r)$.
\end{itemize}
Note that we have combined the real valued coefficients $R_2$ and $\delta B^0$ into one complex valued coefficient $R_2^*$, cf. \eqref{R2star}. 

The \textbf{observations} available for gaining information about these coefficients are the induced voltages according to Faraday's law, that are taken in parallel at $N$ coils   
\begin{equation}\label{eqn:mri_measurement}
    y_j(t) = \int_{\mathbb{R}^3} c^-_j(\vec r) M_{\perp}(t, \vec r) \,\textsf{d}\vec r 
= \mathcal{F}[c^-_j M_{\perp}(t)](0)   
 ,\qquad j=1,\ldots N
\end{equation}
with
\[
\begin{aligned}
&y_j(t)\quad&&\text{\ldots induced voltage at $j$-th coil}\\
&c^-_j(\vec r) \quad&&\text{\ldots coil sensitivity of $j$-th coil.}
\end{aligned}
\]
In order to make sense of the above integral in case of $\Omega$ being a strict subset of $\mathbb{R}^3$, we assume
\begin{equation}\label{supp_cj-}
c_j^-(\vec r)=0, \quad \vec r \in \mathbb{R}^3\setminus\Omega
\end{equation}
or equivalently, extend $M_\perp$ by zero (cf. Remark~\ref{rem:transform_kspace} below).
In terms of the solution to \eqref{eqn:bloch-torrey-complex_rot_FT}, we can write \eqref{eqn:mri_measurement} as
\begin{equation}\label{eqn:mri_measurement_Fourier}
y_j(t)= e^{-i\omega_0t}\,(\mathcal{F}[c^-_j]*\Mphp(t))(0)
=e^{-i\omega_0t}\,\langle \overline{\mathcal{F}[c^-_j]},\Mphp(t)\rangle_{L^2(\mathbb{R}^3)},
%,\qquad j=1,\ldots N,
\end{equation}
where $*$ denotes convolution and the last identity holds due to Plancherel's Theorem; in case of constant coil sensitivities, this becomes
\begin{equation}\label{eqn:mri_measurement_Fourier_cj0}
y_j(t)=c^-_{j,0}\,e^{-i\omega_0t}\,\Mphp(t,0)
,\qquad j=1,\ldots N.
\end{equation}

In this paper we focus on identifiability of $M^{eq}(\vec r)$, $R_1(\vec r)$ and  $R_2^*(\vec r)$ in \eqref{eqn:bloch-torrey-complex-rotated} from the observations \eqref{eqn:mri_measurement}, that is, of the spin density $\rho(\vec r)$, the relaxation times $T_1(\vec r)$, $T_2(\vec r)$, and the field inhomogeneity $\delta B^0(\vec r)$. 

Further coefficients that can be considered as unknowns in an inverse problem for \eqref{eqn:bloch-torrey-complex-rotated} are the diffusion tensor $D(\vec r)$ as another imaging quantity, as well as the hardware parameters 
$c^+(\vec r)$, %$\in \mathbb{C}$ but this can often assumed to be homogeneous $\equiv1$ according to Martin Uecker's email of 2025-02-11}
$c^-_j(\vec r)$, %$\in \mathbb{C}$, 
$j=1,\ldots N$
cf., e.g. \cite{UeckerHohageBlockFrahm2008}, and
the advection velocity $\vec v(t,\vec r)$.
\revision{While we consider these quantities to be known here (for example the hardware parameters have been determined by some a priori testing and calibration procedure) their reconstruction and in particular also uniqueness will be subject of future research.}
\Margin{R2 (2)}
\Margin{R2 5)}

The time dependent 
radio frequency RF field $p(t)$ and 
field gradient $\vec G(t)$
can be regarded as controls
cf., e.g., \cite{AignerClasonRundStollberger2016}
to optimize the data acquisition process in MRI.
\revision{This amounts to a task of optimal experimental design and will be subject of future research as well.}
\Margin{R2 (2)}

Conventional MRI imaging usually relies on simplification $\vec v=0$, $D=0$ of \eqref{eqn:bloch-torrey} to a pointwise-in-space ODE model, the Bloch equation, which in the rotated frame reads as
\begin{equation}\label{bloch-rotated}
\begin{aligned}
&\frac{d}{dt}\Mpp(t, \vec r) +\bigl(R_2^*(\vec r) + i\gamma\vec r \cdot \vec G(t)\bigr)\Mpp(t, \vec r)
%\\&\qquad
-i\gamma  M_z(t, \vec r) \, c^+(\vec r)\pp(t)
=0
\\    
&\frac{d}{dt}M_z(t, \vec r) +R_1(\vec r) (M_z(t, \vec r)-M^{eq}(\vec r))
%\\&\qquad
-\gamma\Re\bigl(i \Mpp(t, \vec r)\overline{c^+(\vec r)\pp(t)}\bigr)
=0.
\end{aligned}
\end{equation}
Explicitly solving this family of ODEs under appropriate initial conditions and assuming approximately constant $R_2^*\approx R_{2,0}^*$ yields the central formula underlying the interpretation of MRI as reconstruction from a sampled Fourier transform: 
\begin{equation}\label{MR-Fourier_rect_B_intro} 
-i\,e^{R_{2,0}^*(t-\taup)+ i\omega_0t} \, y_j(t) 
\approx (\mathcal{F}[c^-_j\,M^{eq}])({\vec k}(t))
,\qquad j=1,\ldots N
\end{equation}
with 
\begin{equation}\label{kt_intro}
{\vec k}(t)=\frac{\gamma}{2\pi}\int_0^t \vec G(\tau)\, d\tau 
\qquad\text{\ldots k-space trajectory}.
\end{equation}
Its derivation can be found in many references. For the convenience of the reader, and since it will play a role in the uniqueness analysis here, we provide it in \eqref{MR-Fourier_rect_B} below.

\medskip

\Margin{R2 (1)}
\revision{
The various levels of models described in this section will play their role in this paper as follows. Well-posedness of the forward problem (Proposition~\ref{prop:wellposed}) and Lipschitz continuous differentiability of the resulting parameter-to-state map $\mathcal{S}:(M^{eq},R_1,R_2,\delta B^0,D,\vec v,c^+,\vec G,\pp)\to \vec{M}$ (Propositions~\ref{prop:perturb_gen}, \ref{prop:perturb_gen_deriv}) will be shown in the most general setting of the convective Bloch-Torrey PDE \eqref{eqn:bloch-torrey}.
To derive approximate explicit reconstruction formulas, on the other hand we consider the most simple of these models, the Bloch ODE \eqref{bloch-rotated} in Subsection~\ref{sec:explsol}; the approximation is quantified in Propositions~\ref{prop:perturb_expl-sol}, \ref{prop:perturb_expl-sol_lin}.
For our uniqueness results, we return to the general convective Bloch Torrey PDE \eqref{eqn:bloch-torrey} and prove uniqueness of $(M^{eq},R_1)$ by a perturbation argument in Theorem~\ref{thm:uniqueness_perturb_red}, as well as linearized uniqueness of $(M^{eq},R_1,R_2,\delta B^0)$ by an infinite speed of propagation argument in Theorem~\ref{thm:uniqueness_diff}. While we assume $\vec v=0$ for the sake of transparency  in the proof of Theorem~\ref{thm:uniqueness_diff}, we expect the result to extend to the convective setting, at the expense of some additional technicalities.}

\section{The forward problem}\label{sec:forward}

In this section we provide some results on well-posedness of the forward model \eqref{eqn:bloch-torrey-complex-rotated} with initial and boundary conditions. In order to emphasize the space-time separable structure of certain terms, we notationally keep track of the independent variables $t$ and $\vec r$.

\subsection{Well-posedness of the PDE for given parameters}\label{sec:wellposed}
We consider the coupled linear system
\begin{equation}\label{eqn:bloch-torrey-complex_f}
\begin{aligned}
&\frac{d}{dt}M_\perp(t, \vec r)+{\vec v}(t, \vec r)\cdot\nabla M_\perp(t, \vec r) -\nabla\cdot\bigl(D(\vec r) \nabla M_\perp(t, \vec r)\bigr) \\
&+\bigl(R_2(\vec r)+i(\omega_0 + \gamma\delta B^0(\vec r) + \gamma\vec r \cdot \vec G(t))
\bigr)M_\perp(t, \vec r)-i\gamma M_z(t, \vec r)c^+(\vec r)p(t)=f_\perp(t, \vec r)\\    
&\frac{d}{dt}M_z(t, \vec r)+{\vec v}(t, \vec r)\cdot\nabla M_z(t, \vec r) -\nabla\cdot\bigl(D(\vec r) \nabla M_z(t, \vec r)\bigr) \\
&+R_1(\vec r) M_z(t, \vec r)
-\gamma\Re\bigl(iM_\perp(t, \vec r)\overline{c^+(\vec r)p(t)}\bigr)=f_z(t, \vec r)
\\
&t>0, \quad \vec r\in\Omega,
\end{aligned}
\end{equation}
cf. \eqref{eqn:bloch-torrey-complex}.
Introducing source terms \revision{$\vec{f}=(f_\perp,f_z)$} will allow us to study differentiability of the parameter-to-state map, as relevant for the inverse problem.

Throughout this \revision{subsection} 
\Margin{R1 1.}
we will assume the PDEs to hold on a bounded Lipschitz domain $\Omega$ and impose  homogeneous Dirichlet or impedance boundary conditions
\begin{equation}\label{bc}
\begin{aligned}
&M_\perp=0, \quad M_z=0 \quad \text{ on }\Gamma_D, \\
%&D(\vec r)\partial_\nu M_\perp(\vec r)+\beta_\perp(\vec r)M_\perp(\vec r)=0, \ 
%D(\vec r)\partial_\nu M_z(\vec r)+\beta_z(\vec r)M_z(\vec r)=0 \quad {\vec r}\in \Gamma_I\\
&D\partial_\nu M_\perp+\beta_\perp M_\perp =0, \ 
D \partial_\nu M_z +\beta_z M_z =0 \quad \text{ on } \Gamma_I\\
&\Gamma_D\cap\Gamma_I=\emptyset, \quad \Gamma_D\cup\Gamma_I=\partial\Omega
\end{aligned}
\end{equation}
with nonnegative functions $\beta_\perp$, $\beta_z$ $\in L^\infty(\Gamma_I,\mathbb{R}_0^+)$, which with $\beta_\perp=0$, $\beta_z=0$ also includes the Neumann boundary condition case.

Well-posedness in case of uniformly positive definite diffusion tensor follows by extending standard parabolic theory (see, e.g., \cite[Section 7.1]{EvansBook}) to the vectorial case. 
We here provide a well-posedness proof that works in a potentially degenerate diffusion setting, exploits the structure of the PDE, in particular the skew symmetric coupling, and imposes a condition that allows for large solenoidal components of the velocity 
\begin{equation}\label{divv}
{\vec v}={\vec v}_0+\delta {\vec v}\text{ such that }\nabla\cdot{\vec v}_0\leq0\text{ a.e. in }(0,T)\times\Omega, \quad \|\nabla\cdot\delta{\vec v}\|_{L^1(0,T;L^\infty(\Omega))}<1. 
\end{equation}
%\footnote{{According to  Martin Uecker's email of 2025-04-02 it is ok to stay with the simple homogeneous Dirichlet case}}.
%Note that formally setting $\omega_0=0$ and replacing $p$ by $\pp$, we cover the rotated PDE \eqref{eqn:bloch-torrey-complex-rotated}.
We use the abbreviation $R_2^*(\vec r):=R_2(\vec r)+i\gamma\,\delta B^0 (\vec r)$ cf. \eqref{R2star} and without loss of generality  consider the rotated system
\begin{equation}\label{eqn:bloch-torrey-complex_rot_f}
\begin{aligned}
&\frac{d}{dt}\Mpp(t, \vec r)+{\vec v}(t, \vec r)\cdot\nabla \Mpp(t, \vec r) -\nabla\cdot\bigl(D(\vec r) \nabla \Mpp(t, \vec r)\bigr) \\
&+\bigl(R_2^*(\vec r)+i\gamma\vec r \cdot \vec G(t)\bigr)\Mpp(t, \vec r)
-i\gamma M_z(t, \vec r)c^+(\vec r)\pp(t)= \fpp(t, \vec r)\\    
&\frac{d}{dt}M_z(t, \vec r)+{\vec v}(t, \vec r)\cdot\nabla M_z(t, \vec r) -\nabla\cdot\bigl(D(\vec r) \nabla M_z(t, \vec r)\bigr) \\
&+R_1(\vec r) M_z(t, \vec r)
-\gamma\Re\bigl(i\Mpp(t, \vec r)\overline{c^+(\vec r)\pp(t)}\bigr)=f_z(t, \vec r)
\end{aligned}
\end{equation}

The energy space $H_D^1(\Omega;\mathbb{C}\times\mathbb{R})$ is defined as the closure of $C_c^\infty(\Omega;\mathbb{C}\times\mathbb{R})$ with respect to the topology defined by the inner product
\begin{equation}\label{HD1}
\begin{aligned}
&\langle (v_1,w_1),(v_2,w_2)\rangle_{H_D^1} :=
\int_\Omega\Bigl\{ \Re\Bigl((D(\vec r) \nabla v_1(\vec r))\cdot \overline{\nabla v_2(\vec r)}
+R_2(\vec r)\,v_1(\vec r))\, \overline{v_2(\vec r)}\Bigr)\\    
&\phantom{\langle (v_1,w_1),(v_2,w_2)\rangle := \int_\Omega\Bigl(}
+(D(\vec r) \nabla w_1(\vec r))\cdot \nabla w_2(\vec r)
+R_1(\vec r) w_1(\vec r)w_2(\vec r)\Bigr\}\, d \vec r\\
&\phantom{\langle (v_1,w_1),(v_2,w_2)\rangle :=}
+\int_{\Gamma_I}\Bigl\{\Re\Bigl(\beta_\perp(\vec r)\,v_1(\vec r))\, \overline{v_2(\vec r)}\Bigr)
+\beta_z(\vec r)\,w_1(\vec r))\, w_2(\vec r)\Bigr\}\, dS(\vec r).
\end{aligned}
\end{equation}
For this purpose it suffices to assume that $D$ is pointwise nonnegative definite (allowing for degenerate diffusion) and that $R_1$, $R_2=\Re(R_2^*)$ are positive bounded away from zero on $\Omega$
\begin{equation}\label{definiteness}
\begin{aligned}
&D(\vec r)\text{ nonnegative definite }, \quad 
R_1(\vec r)\geq \underline{R}_1>0, \quad
R_2(\vec r)=\Re(R_2^*(\vec r))\geq \underline{R}_2>0\\
&\vec r\in\Omega,
\end{aligned}
\end{equation}
%(note that $H_0^1(\Omega;\mathbb{C}\times\mathbb{R})$ coincides with $H^1(\Omega;\mathbb{C}\times\mathbb{R})$ in case $\Omega=\mathbb{R}^3$) 
considering $\mathbb{C}\times\mathbb{R}$ as a real Hilbert space with inner product
$\langle (a_1,b_1),(a_2,b_2)\rangle =\Re(a_1\,\overline{a_2}+b_1\,b_2)$ that is just the real Euclidean inner product between $(\Re(a_1),\Im(a_1),b_1)$ and $(\Re(a_2),\Im(a_2),b_2)$.

\begin{proposition}\label{prop:wellposed}
%Let $\omega_0$, $\gamma\in\mathbb{R}$, and 
%$R_1,\, R_2,\,\delta B^0\in L^\infty(\Omega;\mathbb{R})$, (that is, $R_2^*\in L^\infty(\Omega;\mathbb{C})$, 
%$\vec G\in L^1(0,T;\mathbb{R}^3))$, 
%$\vec v\in L^1(0,T;W^{1,\infty}(\Omega))$ satisfying \eqref{divv},
%$D\in L^\infty(\Omega;\mathbb{R}^{3\times 3})$ with \eqref{definiteness},
%$p\in L^1(0,T)$ 
Let $T\in(0,\infty]$, $\gamma\in\mathbb{R}$, 
$R_1\in L^\infty(\Omega;\mathbb{R})$, $R_2^*,\, c^+\in L^\infty(\Omega;\mathbb{C})$, 
$\vec v\in L^1(0,T;W^{1,\infty}(\Omega;\mathbb{R}^3))$, 
$D\in L^\infty(\Omega;\mathbb{R}^{3\times 3})$, 
satisfying \eqref{divv}, \eqref{definiteness},
$\vec G\in L^1(0,T;\mathbb{R}^3)$, 
$\pp\in L^1(0,T;\mathbb{C})$, 
and 
\Margin{R2 6)}
$\revision{\vec{f}=(f_\perp,f_z)}\in L^1(0,T;L^2(\Omega;\mathbb{C}\times\mathbb{R}))+L^2(0,T;H_D^1(\Omega;\mathbb{C}\times\mathbb{R})^*)$.\footnote{Here ${}^*$ denotes the topological dual, for the definition of the space $H_D^1(\Omega)$ we point to \eqref{HD1}, and $f\in W_1+W_2$ means $f=f_1+f_2$ for some $f_1\in W_1$, $f_2\in W_2$.} 

Then the system 
%\eqref{eqn:bloch-torrey-complex_f}
\eqref{eqn:bloch-torrey-complex_rot_f}
has a unique weak solution
%\footnote{with the time derivative moved to the test function by integration by parts}
\begin{equation}\label{V_wellposed}
(\Mpp,M_z)\in \mathbb{V}:=L^\infty(0,T;L^2(\Omega;\mathbb{C}\times\mathbb{R}))\cap L^2(0,T;H_D^1(\Omega;\mathbb{C}\times\mathbb{R}))
\end{equation}
and the estimate 
\begin{equation}\label{enest_thm}
\begin{aligned}
&\|(\Mpp,M_z)\|_{L^q(0,T;L^2(\Omega))}
+\|(\Mpp,M_z)\|_{L^2(0,T;H_D^1(\Omega))}\\
&\leq C\Bigl(\|(\Mpp,M_z)(0, \cdot)\|_{L^2(\Omega)} + \min\{\|\revision{\vec{f}}\|_{L^1(0,T;L^2(\Omega))}, \,  \|\revision{\vec{f}}\|_{L^2(0,T;H_D^1(\Omega;\mathbb{C}\times\mathbb{R})^*)}\}\Bigr)
\end{aligned}
\end{equation}
holds for any $q\in[2,\infty]$ with a constant 
$C$ depending only on $q$ and $1-\|\nabla\cdot\delta{\vec v}\|_{L^1(0,T;L^\infty(\Omega))}$,
but not on $T$ nor the data $(\Mpp,M_z)(0, \cdot)$, $\revision{\vec{f}}$.

If 
$\revision{\vec{f}}\in L^2(0,T;H_D^1(\Omega;\mathbb{C}\times\mathbb{R})^*)$, 
$D^{-1/2}{\vec v}\in L^\infty(0,T;L^\infty(\Omega;\mathbb{R}^{3\times 3})$, 
$\vec G\in L^2(0,T;\mathbb{R}^3)$,
$\pp\in L^2(0,T;\mathbb{C})$, 
then additionally $(\Mpp,M_z)\in H^1(0,T;H_D^1(\Omega;\mathbb{C}\times\mathbb{R})^*)$ holds.
\end{proposition}

\begin{proof}
The proof does not require any compactness argument 
%(which could be problematic on an unbounded domain, e.g., $\Omega=\mathbb{R}^3$) 
but yet requires boundedness of the domain because of the term $\vec r\cdot \vec G(t)$, that cannot be controlled by energy estimates otherwise, due to the fact that $\frac{\vec r}{|\vec r|}\cdot \vec G(t)$ does not have a definite sign.
 
It relies on Galerkin approximation,  
%(with eigenfunctions of the symmetric elliptic operators $-\nabla\cdot(D \nabla\cdot +R_2)$ and $-\nabla\cdot(D \nabla\cdot +R_1)$), 
energy estimates and weak convergence, using the bilinear form 
\[
\begin{aligned}
&\mathcal{B}((v_1,w_1),(v_2,w_2),t)
:= 
%\int_\Omega\Bigl\{\Re\Bigl( ({\vec v}(t, \vec r)\cdot\nabla v_1(\vec r))v_2(\vec r)+(D(\vec r) \nabla v_1(\vec r))\cdot \overline{\nabla v_2(\vec r)}\\
%&\phantom{:= \int_\Omega\Re\Bigl(}+\Bigl(\bigl(R_2^*(\vec r)+i\gamma\vec r \cdot \vec G(t)\bigr)v_1(\vec r))-i\gamma w_1(\vec r)c^+(\vec r)\pp(t)\Bigr)\, \overline{v_2(\vec r)}\Bigr)\\    
%&\phantom{:= \int_\Omega\Bigl\{}+({\vec v}(t, \vec r)\cdot\nabla w_1(\vec r))w_2(\vec r)+(D(\vec r) \nabla w_1(\vec r))\cdot \nabla w_2(\vec r)\\
%&\phantom{:= \int_\Omega\Re\Bigl(}+\Bigl(R_1(\vec r) w_1(\vec r)-\gamma\Re\bigl(iv_1(\vec r)\overline{c^+(\vec r)\pp(t)}\bigr)\Bigr)w_2(\vec r)\Bigr\}\, d \vec r,
\langle (v_1,w_1),(v_2,w_2)\rangle_{H_D^1}\\
&+\int_\Omega\Bigl\{\Re\Bigl( {\vec v}(t, \vec r)\cdot\nabla v_1(\vec r)
+i\bigl((\Im(R_2^*(\vec r))+\gamma\vec r \cdot \vec G(t))v_1(\vec r)
-\gamma w_1(\vec r)c^+(\vec r)\pp(t)\Bigr)\, \overline{v_2(\vec r)}\Bigr)\\    
&\phantom{:= \int_\Omega\Bigl\{}
+\Bigl({\vec v}(t, \vec r)\cdot\nabla w_1(\vec r)
-\gamma\Re\bigl(iv_1(\vec r)\overline{c^+(\vec r)\pp(t)}\bigr)\Bigr)w_2(\vec r)\Bigr\}\, d \vec r,
\end{aligned}
\]
on $H_D^1(\Omega;\mathbb{C}\times\mathbb{R})$.
 
The crucial energy identity emerges from inserting $v_1=v_2=\Mpp(t)$, $w_1=w_2=M_z(t)$ (or actually their Galerkin approximants)  and integrating with respect to time, which yields
\begin{equation}\label{enid}
\begin{aligned}
&\tfrac12\|\Mpp(t, \cdot)\|_{L^2(\Omega)}^2
+\int_0^t\Bigl(\|\sqrt{D}\nabla \Mpp(s,\cdot)\|_{L^2(\Omega)}^2+\|\sqrt{R_2}\Mpp(s,\cdot)\|_{L^2(\Omega)}^2\bigr)\, ds\\
&+\tfrac12\|M_z(t, \cdot)\|_{L^2(\Omega)}^2+\int_0^t\Bigl(\|\sqrt{D}\nabla M_z(s,\cdot)\|_{L^2(\Omega)}^2+\|\sqrt{R_1}M_z(s,\cdot)\|_{L^2(\Omega)}^2\bigr)\, ds\\
&-\frac12 \int_0^t\int_\Omega (\nabla\cdot{\vec v})(s, \vec r)\,(|\Mpp(s, \vec r)|^2+|M_z(s, \vec r)|^2)\, d \vec r\, ds\\
=& \tfrac12\|M_\perp(0, \cdot)\|_{L^2(\Omega)}^2+\tfrac12\|M_z(0, \cdot)\|_{L^2(\Omega)}^2\\
&+\int_0^t\int_\Omega\Bigl(\Re ( \fpp(s, \vec r) \cdot\overline{\Mpp(s, \vec r)})+f_z (s, \vec r) M_z(s, \vec r)\Bigr)\, d \vec r\, ds,
\end{aligned}
\end{equation}
where the terms multiplied with $\gamma$ cancel
\[
M_z(t,\vec r)\Im\bigl(c^+(\vec r)\pp(t)\, \overline{\Mpp(t,\vec r)}\bigr)
+\Re\bigl(i\Mpp(t,\vec r)\overline{c^+(\vec r)\pp(t)}\bigr)\Bigr)M_z(t,\vec r)=0,
\]
and we have applied integration by parts to the velocity terms 
\[
\begin{aligned}
&\int_\Omega {\vec v}(s, \vec r)\cdot \nabla m(s, \vec r)\, \overline{m(s, \vec r)}\, d \vec r
=\frac12\int_\Omega {\vec v}(s, \vec r)\cdot \nabla |m(s, \vec r)|^2\, d \vec r,\\
&m\in\{\Mpp,M_z\}.
\end{aligned}
\]

Note that the terms with purely imaginary factors of $\Mpp$ are invisible to this energy estimate and so in the same way one would obtain the identity \eqref{enid} with $M_\perp$ in place of $\Mpp$ for a solution to \eqref{eqn:bloch-torrey-complex_f}.

Using Cauchy-Schwarz', and Young's inequalities taking the supremum with respect to time yields an energy estimate in $L^\infty(0,T;L^2(\Omega))\cap L^2(0,T;H_D^1(\Omega))$ as follows
\Margin{R1 2.}
\begin{equation}\label{enest}
\begin{aligned}
&\max\{(1-\mu-\|\nabla\cdot\delta{\vec v}\|_{L^1(0,T;L^\infty(\Omega))}) 
\|(\Mpp,M_z)\|_{L^\infty(0,T;L^2(\Omega))}^2, \,\\
&\phantom{\max\{(1-\mu-\|\nabla\cdot\delta{\vec v}\|_{L^1(0,T;L^\infty(\Omega))})}
\|(\Mpp,M_z)\|_{L^2(0,T;H_D^1(\Omega))}^2\}\\
&\leq \|(\Mpp,M_z)(0, \cdot)\|_{L^2(\Omega)}^2 + \min\{\tfrac{1}{\mu}\|\revision{\vec{f}}\|_{L^1(0,T;L^2(\Omega))}^2, \, 2 \|\revision{\vec{f}}\|_{L^2(0,T;H_D^1(\Omega)^*)}^2\},
\end{aligned}
\end{equation}
provided \eqref{divv} holds. We use this with $\mu=\frac12(1-\|\nabla\cdot\delta{\vec v}\|_{L^1(0,T;L^\infty(\Omega))})$ and use interpolation 
$\|\phi\|_{L^q(0,T)}\leq \|\phi\|_{L^\infty(0,T)}^{1-\frac{2}{q}} \|\phi\|_{L^2(0,T)}^{\frac{2}{q}}$ to obtain \eqref{enest_thm} with a constant independent of $T$ (which is in particular essential in case $T=\infty$).

To obtain bounds on the time derivatives in case $\revision{\vec{f}}\in L^2(0,T;H_D^1(\Omega;\mathbb{C}\times\mathbb{R})^*)$, we test the PDE with $\frac{d}{dt}(\Mpp,M_z)$, which yields
\[
\begin{aligned}
&\|\frac{d}{dt}\Mpp\|_{L^2(0,T;H_D^1(\Omega)^*)}\\
&=\|-{\vec v}\cdot\nabla \Mpp +\nabla\cdot\bigl(D \nabla \Mpp\bigr) 
-R_2^{**}\Mpp-i\gamma M_zc^+p-f_\perp\|_{L^2(0,T;H_D^1(\Omega)^*)}^2\\    
&\|\frac{d}{dt}M_z\|_{L^2(0,T;H_D^1(\Omega)^*)}\\
&=\|-{\vec v}\cdot\nabla M_z +\nabla\cdot\bigl(D \nabla M_z\bigr) 
-R_1 M_z
%+\Meqdot
+\gamma\Re\bigl(i\Mpp\overline{c^+p}\bigr)-f_z\|_{L^2(0,T;H_D^1(\Omega)^*)}
\end{aligned}
\]
with 
\begin{equation}\label{Rstartstar}
R_2^{**}(t,\vec r)
%=R_2(\vec r)+i(\omega_0 + \gamma\delta B^0(\vec r) + \gamma\vec r \cdot \vec G(t)),
=R_2^*(\vec r)+i\gamma\vec r \cdot \vec G(t),
\end{equation}
which under the given assumptions can be further bounded as follows.
We have
\[
\begin{aligned}
&\|{\vec v}\cdot\nabla m\|_{L^2(0,T;H_D^1(\Omega)^*)}
\leq\|{\vec v}\cdot\nabla m\|_{L^2(0,T;L^2(\Omega))}
=\|\bigl(D^{-1/2}{\vec v}\bigr)\cdot  D^{1/2}\nabla m\|_{L^2(0,T;L^2(\Omega))}\\
&\leq \|D^{-1/2}{\vec v}\|_{L^\infty(0,T;L^\infty(\Omega))} \|m\|_{L^2(0,T;H_D^1(\Omega))}
\\[1ex] 
&\|\nabla\cdot(D\nabla m)\|_{L^2(0,T;H_D^1(\Omega)^*)}\leq \|m\|_{L^2(0,T;H_D^1(\Omega))}\\ 
&m\in\{\Mpp,\,M_z\}
\end{aligned}
\]
\[
\begin{aligned}
&\|R\,m\|_{L^2(0,T;H_D^1(\Omega)^*)}
%=\Bigl(\sup_{v\in H_D^1(\Omega)\setminus\{0\}}\int_0^T\Bigl(\int_\Omega R(t,\vec r)\,m(t,\vec r)\,v(\vec r)\, dr \Bigr)^2\, dt\Bigr)^{-1/2}\|v\|_{H_D^1(\Omega)}^{-1}
\lesssim \|R\,m\|_{L^2(0,T;L^2(\Omega))}\\
&\leq \|R\|_{L^2(0,T;L^\infty(\Omega))} \|m\|_{L^\infty(0,T;L^2(\Omega))}
\quad (R,m)\in\{(R_2^{**},\Mpp),\,(R_1,M_z)\}
\end{aligned}
\]
Boundedness of $\Omega$ is required for bounding the $\gamma\vec r \cdot \vec G(t)$ term.

%The general case including a fourth order diffusion tensor $\tilde{D}(t,\vec r)$ with details on the Galerkin approximation can be found in \cite{MAthesisKogler}.
\Margin{R1 A.}
\end{proof}

Note that assuming uniform positive definiteness of $D$ we could reduce the regularity assumptions on the remaining coefficients by making use of $\|(\Mpp,M_z)\|_{H^1(\Omega)}$ more extensively.

Proposition~\ref{prop:wellposed} will allow us to establish well-definedness, continuity and smoothness of the parameter-to-state map $\mathcal{S}$, cf. Section~\ref{subsec:S}.

\begin{remark}\label{rem:pdelta}
In the low regularity regime with 
\revision{solution}
\Margin{R1 3.} 
space $\mathbb{V}$ defined by \eqref{V_wellposed} we can stay with $\pp\in L^1(0,T)$, which due to its non-reflexivity is an awkward space to work with for optimization, though. Replacing it by the space of Radon measures $\mathcal{M}(0,T)$ that can be identified with the dual of the separable space $C(0,T)$, as is often done in such situations, does not seem to be an option here though. To see this, consider the simple example 
\begin{equation}\label{counterexdelta}
\frac{d}{dt} u - \delta_{t_0}\cdot u =0, \ t\in(0,T)\quad u(0)=0
\end{equation}
with $t_0\in(0,T)$. One might be tempted to accept the Heaviside function $u=1\!\mathrm{I}_{[t_0,T)}$
%with jump at $t=t_0$ 
as a solution, by understanding $\frac{d}{dt}$ in a distributional sense. However, multiplication with $\delta_{t_0}$ in a 
\revision{distributional} 
\Margin{R1 4.}
sense is defined by 
\[
\forall \phi\in\mathcal{D}\,: \ \langle \delta_{t_0}\cdot u, \, \phi\rangle_{\mathcal{D}',\mathcal{D}}
:=\langle \delta_{t_0}, \, \phi\,u\rangle_{\mathcal{D}',\mathcal{D}}
\]
with $\mathcal{D}=C_c^\infty(0,T)$ being the space of test functions, and would thus require $u$ to be contained in $C^\infty(0,T)$, if the outcome is supposed to be a distribution, or at least $u$ to be continuous, so that $\phi\,u\in C(0,T)$ and thus the right hand side in the above identity can be made sense of as $\langle \delta_{t_0}, \, \phi\,u\rangle_{\mathcal{M},C}$. This excludes $u=1\!\mathrm{I}_{[t_0,T)}$ as a solution to \eqref{counterexdelta}, though.

As a consequence, we will use rectangular pulses rather than delta pulses when working with PDE estimates for obtaining uniqueness, that is, in Section~\ref{sec:uniqueness_perturbation}.
\end{remark}
\kspace{
\begin{remark}\label{rem:transform_kspace} %2025-06-10
In case of homogeneous Dirichlet boundary conditions on the whole boundary $\Gamma_D=\partial\Omega$, $\Gamma_I=\emptyset$, with a $C^1$ boundary $\partial\Omega$, the extension-by-zero operator $H^1_D(\Omega)\to H^1_D(\mathbb{R}^3)$ is continuous (cf., e.g., \cite{BrezziFortin}) and in case of (space-) constant $D$, $\vec v$, $R_1$, $R_2^*$, $c^+$ allows us to equivalently characterize solutions to \eqref{eqn:bloch-torrey-complex_rot_f} by PDEs in k-space
\begin{equation}\label{eqn:bloch-torrey-complex_rot_FT_f}
\begin{aligned}
&\frac{d}{dt}\Mphp(t,\vec\xi) - \frac{\gamma}{2\pi}\,\vec G(t)\cdot \nabla_\xi \Mphp(t,\vec\xi)
+\bigl(R_{2,0}^*+\vec v_0(t)\cdot\xi+\xi\cdot D_0\xi \bigr)\Mphp(t,\vec\xi)\\
&\qquad-i\gamma \widehat{M_z}(t, \vec\xi)c^+_0 \pp(t)=\mathcal{F}\fpp(t, \vec\xi)\\    
&\frac{d}{dt}\widehat{M_z}(t, \vec\xi) 
+\bigl(R_{1,0}+\vec v_0(t)\cdot\xi+\xi\cdot D_0\xi\bigr)\widehat{M_z}(t, \vec\xi)-R_{1,0}\widehat{M^{eq}}(\vec\xi))\\
&\qquad-\gamma\Re\bigl(i\Mphp(t, \vec\xi)\overline{c^+_0 \pp(t)}\bigr)=\mathcal{F}f_z(t, \vec\xi).
\end{aligned}
\end{equation}
\end{remark}
}
\begin{comment}
\begin{remark}\label{rem:wellposedness_kspace}
Well-posedness might also be established in case $\Omega=\mathbb{R}^3$ by working in k-space and applying a perturbation argument for \eqref{eqn:bloch-torrey-complex_rot_FT}, estimating the convolution terms 
\[
\begin{aligned}
&\mathcal{F}[\nabla\cdot(\widehat{\delta D} \nabla m])(\vec\xi)
=\vec\xi\cdot\int_{\mathbb{R}^3}\widehat{\delta D}(\vec\xi-\vec\eta)\,\vec\eta\,\widehat{m}(\vec\eta)\, d\eta  \\
&\mathcal{F}[\delta {\vec v} \nabla m])(\vec\xi)
=\int_{\mathbb{R}^3}\widehat{\delta {\vec v}}(\vec\xi-\vec\eta)\cdot\vec\eta\,\widehat{m}(\vec\eta)\, d\eta\\
&\mathcal{F}[\delta R\, m])(\vec\xi)
=\int_{\mathbb{R}^3}\widehat{\delta R}(\vec\xi-\vec\eta)\,\widehat{m}(\vec\eta)\, d\eta \quad R\in\{R_1,R_2^*,c^+\}\\
&m\in\{\Mpp,M_z\}\\
&\text{with }\delta D(\vec r)=D(\vec r)-D_0, \quad
\delta {\vec v}(t,\vec r)={\vec v}(t,\vec r)-{\vec v}_0(t), \quad
\delta R(\vec r)=R(\vec r)-R_0.
\end{aligned}
\]
However, the transport equation only yields quite weak spatial regularity for $\Mpp$ and none for $M_z$. 
\end{remark}
\end{comment}

\subsection{Explicit solution formulas with rectangular pulse RF fields}\label{sec:explsol}
Throughout this subsection we impose ${\vec v}=0$ and thus $\frac{D}{Dt}=\frac{d}{dt}$ in 
\eqref{eqn:bloch-torrey-complex-rotated}, \eqref{eqn:bloch-torrey-complex_rot_FT}, cf. \eqref{Mperp_prime}.
Moreover, $\Omega$ may be any bounded or unbounded subset of $\mathbb{R}^3$.
Only when comparing solutions to the Bloch-Torrey equation by means of Proposition~\ref{prop:wellposed} in Subsection~\ref{subsec:S}, we will have to assume $\Omega$ to be a bounded Lipschitz domain in order to be able to apply Proposition~\ref{prop:wellposed}.  
In case of $\Omega$ being a strict subset of $\mathbb{R}^3$, in order to make sense of the observations \eqref{eqn:mri_measurement}, we assume the support of the coil sensitivities to be contained in $\Omega$, cf. \eqref{supp_cj-}.

With a rectangular pulse 
\begin{equation}\label{pt_rect}
\begin{aligned}
&p(t)=e^{-i\omega_0t}\pp(t), \quad \pp(t)= {\pp}_0\,\taup^{-1} 1\!\mathrm{I}_{[t_0,t_0+\taup]}\\
&\text{ for some }\taup>0,\text{ typically }\taup<<1,
\quad {\pp}_0>0
\end{aligned}
\end{equation}
and assuming that ${\vec G}\equiv0$ on $[t_0,t_0+\taup]$
we can find explicit solutions 
\begin{itemize}
\item
%[(B)] 
in the Bloch ODE case \eqref{bloch-rotated} with $D=0$ and spatially variable relaxation times as relevant for uniqueness of $R_1(\vec r)$, $R_2(\vec r)$ recovery 
\begin{equation}\label{B_rect}
\begin{aligned}
&\frac{d}{dt}\Mpp(t, \vec r) +
R_2^*(\vec r) \Mpp(t, \vec r)-i\gamma\revision{{\pp}_0} \taup^{-1}\,  M_z(t, \vec r)c^+(\vec r) =0\\    
&\frac{d}{dt}M_z(t, \vec r)  
+R_1(\vec r) (M_z(t, \vec r)-M^{eq}(\vec r))\\
&\hspace*{3.6cm}+\gamma\revision{{\pp}_0} \taup^{-1}\, \Re\bigl(i\Mpp(t, \vec r)c^+(\vec r)\bigr)=0\\
&t\in[t_0,t_0+\taup], \quad \vec r\in\Omega, 
\end{aligned}
\end{equation}
\Margin{R1 B.}
\item
%[(BT${}_0$)] 
\kspace{
in the Bloch-Torrey PDE case with 
%$\Omega=\mathbb{R}^3$ and 
constant diffusion $D\equiv D_0$, relaxation times $R_1\equiv R_{1,0}$, $R_2^*\equiv R_{2,0}$, and $c^+(\vec r)\equiv c^+_0$ after taking the Fourier transform with respect to space \eqref{eqn:bloch-torrey-complex_rot_FT}
\begin{equation}\label{BT0_rect}
\begin{aligned}
&\frac{d}{dt}\Mphp(t, \vec\xi) 
+\bigl(\xi\cdot D_0\xi +R_{2,0}\bigr)\Mphp(t, \vec\xi)
-i\gamma\revision{{\pp}_0} \taup^{-1}\,  \widehat{M_z}(t, \vec\xi)c^+_0 =0\\    
&\frac{d}{dt}\widehat{M_z}(t, \vec\xi) 
+\Bigl(\bigl(\xi\cdot D_0\xi +R_{1,0}\bigr)\widehat{M_z}(t, \vec\xi)-R_{1,0}\widehat{M^{eq}}(\vec\xi))\Bigr)\\
&\hspace*{5.1cm}+\gamma\revision{{\pp}_0} \taup^{-1}\, \Re\bigl(i\Mphp(t, \vec\xi)c^+_0\bigr)=0\\
&t\in[t_0,t_0+\taup], \quad \xi\in\mathbb{R}^3.
\end{aligned}
\end{equation}
}
\end{itemize}  
Both lead to a system of ODEs, pointwise in $\Omega$ or in k-space, of the form
\begin{equation}\label{3by3ODEsys}
\frac{d}{dt}\left(\begin{array}{c}m_x(t)\\m_y(t)\\m_z(t)\end{array}\right)
+
\left(\begin{array}{ccc}
\alpha_2&0&-\bb_y\\ 
0&\alpha_2&\bb_x\\
\bb_y&-\bb_x&\alpha_1\end{array}\right)
\left(\begin{array}{c}m_x(t)\\m_y(t)\\m_z(t)\end{array}\right)
=\left(\begin{array}{c}0\\0\\f\end{array}\right)\quad t\in[t_0,t_0+\taup],
\end{equation}
where for the moment we return to the $\left(\begin{array}{c}m_x\\m_y\end{array}\right)$ instead of $m_x+im_y$ notation and  
\begin{equation}\label{alpha1alpha2b1}
\begin{aligned}
&(\alpha_1,\alpha_2,\bb)=(R_1(\vec r),R_2^*(\vec r),\gamma{\pp}_0\,\taup^{-1} c^+(\vec r))\text{ or }\\
&(\alpha_1,\alpha_2,\bb)=(\xi\cdot D_0\xi +R_{1,0},\xi\cdot D_0\xi +R_{2,0},\gamma{\pp}_0\,\taup^{-1} c^+_0).
\end{aligned}
\end{equation}

\begin{lemma}\label{lem_rect}
For the solutions $(\Mpp,M_z)$ of 
%\eqref{B_rect} 
\eqref{bloch-rotated} 
and $(\Mphp,\widehat{M_z})$ of 
%\eqref{BT0_rect} 
\eqref{eqn:bloch-torrey-complex_rot_FT}
with ${\pp}_0 c^+(\vec r)$ $\in \mathbb{R}^+$, $\vec G\equiv 0$, we have, with 
\revisionown{$\varphi(t,\vec r)=(t-t_0)\,\gamma\,{\pp}_0\,c^+(\vec r)\,\taup^{-1}$,} 
\[
\begin{aligned}
&|\Mpp(t,\vec r)\,-\bigl(e^{-R_2^*(r)(t-t_0)}\Re(\Mpp(t_0,\vec r))\\
&\qquad\qquad\qquad+i\bigl(\cos(\varphi(t,\vec r))\Im(\Mpp(t_0,\vec r))-\sin(\varphi(t,\vec r))M_z(t_0,\vec r)\bigr)\bigr)|
\\&\leq C\,\taup(|M(t_0,\vec r)|+ |M^{eq}(\vec r)|)
\\[1ex]
&|M_z(t,\vec r)\,-\bigl(\sin(\varphi(t,\vec r))\Im(\Mpp(t_0,\vec r))+\cos(\varphi(t,\vec r))M_z(t_0,\vec r)\bigr)|
\\&\leq C\,\taup(|M(t_0,\vec r)|+ |M^{eq}(\vec r)|)
\end{aligned}
\]
for a constant $C$ depending only on $|R_1(\vec r)|$, $|R_2^*(\vec r)|$
\kspace{ and
\[
\begin{aligned}
&|\Mphp(t,\vec\xi)\, -\bigl(e^{-(\xi\cdot D_0\xi+R_{2,0}^*)(t-t_0)}\Re(\Mphp(t_0,\vec\xi))\\
&\qquad\qquad\qquad+i\bigl(\cos(\varphi(t,\vec r))\Im(\Mphp(t_0,\vec\xi))-\sin(\varphi(t,\vec r))\widehat{M_z}(t_0,\xi)\bigr)\bigr)|
\\&\leq C\,\taup|\Mphp(t_0,\vec\xi)|+ |\widehat{M^{eq}}(\vec\xi)|)
\\[1ex]
&|\widehat{M_z}(t,\vec\xi)\,-\bigl(\sin(\varphi(t,\vec r))\Im(\Mphp(t_0,\xi))+\cos(\varphi(t,\vec r))\widehat{M_z}(t_0,\xi)\bigr)|
\\&\leq C\,\taup|\Mphp(t_0,\vec\xi)|+ |\widehat{M^{eq}}(\vec\xi)|)
\end{aligned}
\]
for a constant $C$ depending only on $|R_{1,0}|$, $|R_{2,0}^*|$.
}
\end{lemma}
\begin{proof}
See the appendix. 
%of \cite{MRI-uniqueness_arxiv}
\end{proof}
\Margin{R1 C.}
\Margin{R2 (3)}

That is, application of such a pulse up to an $O(|\bb|^{-1})=O(\taup)$ remainder results in a damping of the $x$ component of $m(t_0)$, while the 2-d vector $(m(t_0)_y,m(t_0)_z)$ is rotated around the $x$ axis by an angle of $\varphi(t,\vec r)$.
This conforms to the known fact that the flip angle is proportional to the integral over the envelope $\pp(t)$ of the pulse cf. \cite[page 10]{Uecker2009NonlinearRM}. 
In particular, for a rectangular pulse with 90${}^{\circ}$ or 180${}^{\circ}$ flip angle, recalling \eqref{alpha1alpha2b1}, we have the correspondences
\begin{equation}\label{90-180_rect}
\begin{aligned}
&\gamma{\pp}_0\,c^+=
%|\bb|\taup=
\frac{\pi}{2}\text{\ldots 90 ${}^{\circ}$ pulse}, \quad &&
M_y(t_0+\taup)\dot{=}-M_z(t_0), \ M_z(t_0+\taup)\dot{=}M_y(t_0)\\
&\gamma{\pp}_0\,c^+=
%|\bb|\taup=
\pi\text{\ldots 180 ${}^{\circ}$ pulse}, &&
M_y(t_0+\taup)\dot{=}-M_y(t_0), \ M_z(t_0+\taup)\dot{=}-M_z(t_0) 
\end{aligned}
\end{equation}
Here $\dot{=}$ means ``up to an $O(|\bb|^{-1})=O(\taup)$ remainder''.

\medskip

When using several RF pulses on intervals disjoint to the support of ${\vec G}$, the solution can be pieced together with the solution formula for the uncoupled problem in the 
Bloch ODE case \eqref{bloch-rotated}
\begin{equation}\label{B_uncoupled}
\begin{aligned}
&\frac{d}{dt}\Mpp(t, \vec r) +\bigl(R_2^*(\vec r)+ i\gamma\vec r \cdot \vec G(t)\bigr)\Mpp(t, \vec r) =0\\    
&\frac{d}{dt}M_z(t, \vec r) +R_1(\vec r) (M_z(t, \vec r)-M^{eq}(\vec r))=0\\
&t\in[t_0+\taup,t_0+\taup+\tau], \quad \vec r\in \Omega,
\end{aligned}
\end{equation}
\kspace{
or the Fourier transformed constant coefficient case \eqref{eqn:bloch-torrey-complex_rot_FT}
\begin{equation}\label{BT0_uncoupled}
\begin{aligned}
&\frac{d}{dt}\Mphp(t, \vec\xi) 
+\bigl(\xi\cdot D_0\xi +R_{2,0}\bigr)\Mphp(t, \vec\xi)
-\frac{\gamma}{2\pi}\vec G(t)\cdot\nabla_\xi\Mphp(t, \vec\xi)=0\\    
&\frac{d}{dt}\widehat{M_z}(t, \vec\xi) 
+\bigl(\xi\cdot D_0\xi +R_{1,0}\bigr) (\widehat{M_z}(t, \vec\xi)-\widehat{M^{eq}}(\vec\xi))=0\\
&t\in[t_0+\taup,t_0+\taup+\tau],\quad \vec\xi\in \mathbb{R}^3,
\end{aligned}
\end{equation}
}
\Margin{R1 5.} 
given by
\begin{equation}\label{MperpMz_B}
\begin{aligned}
&\Mpp(t,\vec r)=
\revision{\eexp{-(R_2^*(\vec r)(t-t_0+\taup)+2\pi i ({\vec k}(t)-{\vec k}(t_0+\taup))\cdot{\vec r})}}\,\Mpp(t_0+\taup,\vec r)\\
&M_z(t,\vec r)= M^{eq}(\vec r)+ \revision{\eexp{-R_1(\vec r)(t-t_0+\taup)}}\,(M_z(t_0+\taup, \vec r)-M^{eq}(\vec r)).
\end{aligned}
\end{equation}
\kspace{
and 
\begin{equation}\label{hatMperpMz_BT0}
\begin{aligned}
\Mphp(t,\vec\xi)=&\revision{\eexp{-\int_{t_0+\taup}^t((\vec\xi+\vec k(t)-\vec k(\sigma))\cdot D_0(\vec\xi+\vec k(t)-\vec k(\sigma))+R_{2,0}^*)\, d\sigma}}\\
&\qquad\times\Mphp(t_0+\taup,\vec\xi+\vec k(t)-\vec k(t_0+\taup))\\
\widehat{M_z}(t,\vec\xi)=& \widehat{M^{eq}}(\vec\xi)+ \revision{\eexp{-(\xi\cdot D_0\xi +R_{1,0})(t-t_0+\taup)}} (\widehat{M_z}(t_0+\taup, \vec\xi)-\widehat{M^{eq}}(\vec\xi)).
\end{aligned}
\end{equation}
respectively, 
}
with
\begin{equation}\label{kt}
{\vec k}(t)=\frac{\gamma}{2\pi}\int_0^t \vec G(\tau)\, d\tau.
\end{equation}

\medskip

We now study two situations that are tailored to the recovery of $M^{eq}$ (and, via time differentiation, also of $R_2^*$) as well as $R_1$ as functions of space, and thus for imaging these quantities by means of the available observations, which will also be demonstrated by (approximate) reconstruction formulas. 
\\
\revision{Motivated by} Lemma~\ref{lem_rect} in the following $\dot{=}$ means ``up to an $O(\taup)$ remainder''.
\Margin{R1 6.}
  
\medskip

\noindent
%\underline{Sequence for $M^{eq}$ and $R_2$ recovery: $90{}^{\circ}$ pulse:}\\
\textit{Sequence for $M^{eq}$ and $R_2$ recovery: $90{}^{\circ}$ pulse:}\\[1ex]
With $c^+(\vec r)=c^+_0\in\mathbb{R}^+$ and
\begin{equation}\label{90_rect}
\begin{aligned}
&p(t)= e^{-i\omega_0t}\, {\pp}_{90^\circ}(t), \quad  {\pp}_{90^\circ}=(2c^+_0\gamma\taup)^{-1}1\!\mathrm{I}_{[0,\taup]}, \\
&\vec{G}(t)=0\text{ on } [0,\taup],\quad
\Mpp(0)=0, \quad M_z(0)=M^{eq}
\end{aligned}
\end{equation}
in the Bloch ODE case \eqref{B_rect}, formula \eqref{90-180_rect} yields
\begin{equation}\label{init90eps}
\begin{aligned}
&\Mpp(\taup,\cdot)\dot{=}-iM_z(0,\cdot)=-iM^{eq}\\
&M_z(\taup,\cdot)\dot{=}\Im(\Mpp(0,\cdot))=0
\end{aligned}
\end{equation}
hence by \eqref{MperpMz_B}
\begin{equation}\label{Mperp_ex_90_rect_B}
\begin{aligned}
\Mpp(t,\vec r)=& 
%e^{-(R_2^*(\vec r)(t-\taup)+2\pi i ({\vec k}(t)-{\vec k}(\taup))\cdot{\vec r})}\,\Mpp(\tau,\vec r)\\
\eexp{-(R_2^*(\vec r)(t-\taup)+2\pi i {\vec k}(t)\cdot{\vec r})}\,\Mpp(\tau,\vec r)\\
\dot{=}& 
%-ie^{-(R_2^*(\vec r)(t-\taup)+2\pi i ({\vec k}(t)-{\vec k}(\taup))\cdot{\vec r})}M^{eq}(\vec r)\,\\
-i\eexp{-(R_2^*(\vec r)(t-\taup)+2\pi i {\vec k}(t)\cdot{\vec r})}M^{eq}(\vec r)\,\\
=&:{\Mpp}_{90^\circ}(t,\vec r)
\qquad \text{ for }t>\taup,\quad \vec r\in\Omega, 
\end{aligned}
\end{equation}
where we have used the fact that under our assumption $\vec{G}(t)=0\text{ on } [0,\taup]$ we have ${\vec k}(\taup)=0$.

Decomposing the effective relaxation time into a mean value and spatial fluctuations
\begin{equation}\label{R2effdecomp}
R_2^*(\vec r) = R_{2,0}^*+\delta R_2^*(\vec r),
\end{equation}
\footnote{Note that as opposed to the decomposition \eqref{R2star}, both components may be complex valued in \eqref{R2effdecomp}.}
in case of approximately constant $R_2^*$, that is, $\delta R_2^*\approx0$, 
due to \eqref{Mperp_ex_90_rect_B},
one can thus recover the Fourier transform of $c^-_j\,M^{eq}$ along the trajectories ${\vec k}(t)$ from the observed data \eqref{eqn:mri_measurement} as
\begin{equation}\label{MR-Fourier_rect_B} 
\begin{aligned}
&i\,e^{R_{2,0}^*(t-\taup)+ i\omega_0t} \, y_j(t) 
\dot{=} \int_{\mathbb{R}^3} e^{-\delta R_2^*(\vec r)(t-\taup)} c^-_j(\vec r)\,M^{eq}(\vec r) \, e^{-2\pi i {\vec k}(t)\cdot{\vec r}}  \,\textsf{d}\vec r\\
&\approx \int_{\mathbb{R}^3} c^-_j(\vec r)\,M^{eq}(\vec r) \, e^{-2\pi i {\vec k}(t)\cdot{\vec r}}  \,\textsf{d}\vec r
=(\mathcal{F}[c^-_j\,M^{eq}])({\vec k}(t))
,\qquad j=1,\ldots N\\
&\text{ for }t>\taup.
\end{aligned}
\end{equation}
This is in fact the key formula underlying the interpretation of MRI as reconstruction from a sampled Fourier transform, \cite{Nishimura2010,Uecker2009NonlinearRM} and the references therein, see also Section~\ref{subsec:MRIiso} below.

Note that in \eqref{MR-Fourier_rect_B}, two types of approximation take place: $\dot{=}$ stands for ``up to an $O(\taup)$ remainder'', whereas $\approx$ means ``up to an $O(\delta R_2^*)$ remainder''.

\smallskip

\kspace{
There is also a k-space version of this formula, that allows for nonvanishing diffusion but requires constant coefficients. For \eqref{eqn:bloch-torrey-complex_rot_FT}, by using \eqref{hatMperpMz_BT0} in place of \eqref{MperpMz_B} we obtain
\begin{equation}\label{Mperp_ex_90_rect_BT0}
\begin{aligned}
&\Mphp(t,\vec\xi)\\
&\dot{=}
-i\,\eexp{-\int_{\taup}^t((\vec\xi+\vec k(t)-\vec k(\sigma))\cdot D_0(\vec\xi+\vec k(t)-\vec k(\sigma))+R_{2,0}^*)\, d\sigma}\,\widehat{M^{eq}}(\vec\xi+\vec k(t))
\\
&=:\widehat{{\Mpp}_{90^\circ}}(t,\vec\xi)\qquad
\text{ for }\vec\xi\in\mathbb{R}^3, \quad t>\taup,
\end{aligned}
\end{equation}
which 
%with constant coil sensitivities $c_j^-= c_{j,0}^-1\!\mathrm{I}_{\Omega}$ 
entails cf. \eqref{eqn:mri_measurement_Fourier}
\begin{equation}\label{MR-Fourier_rect_BT0} 
\begin{aligned}
&i\,\eexp{R_{2,0}^*(t-\taup)+ i\omega_0t} \, y_j(t) \\
&\dot{=} 
%\int_{\mathbb{R}^3}
%\mathcal{F}[c_j^-](\vec\xi)\,
%e^{-\int_{\taup}^t(\vec\xi+\vec k(t)-\vec k(\sigma))\cdot D_0(\vec\xi+\vec k(t)-\vec k(\sigma))\, d\sigma}
%\,\widehat{M^{eq}}(\vec\xi+\vec k(t))\, d\xi\\
%&=
\int_{\mathbb{R}^3}
\mathcal{F}[c_j^-](\vec\xi-\vec k(t))\,
\eexp{-\int_{\taup}^t(\vec\xi-\vec k(\sigma))\cdot D_0(\vec\xi-\vec k(\sigma))\, d\sigma}\,\widehat{M^{eq}}(\vec\xi)\, d\xi\\
%c_{j,0}^-\,(\mathcal{F}[M^{eq}])(\vec k(t))
&\text{ for }t>\taup,\qquad j=1,\ldots N.
\end{aligned}
\end{equation}
}

\smallskip

Also $R_2^*(\vec r)$ can be obtained by differentiation of \eqref{MR-Fourier_rect_B} with respect to time, see the appendix. 
%of \cite{MRI-uniqueness_arxiv}.
While reconstruction of a relaxation time from the rate of the signal appears a very natural thing to do, differentiation renders the resulting formula (mildly) ill-posed and we thus in Section~\ref{sec:uniqueness} here follow a different reconstruction approach, using the same simple experiment.   

\noindent
%\underline{Sequence for $R_1$ recovery: $180{}^{\circ}$ pulse -- $\tau$ intermission -- $90{}^{\circ}$ pulse:}\\
\textit{Sequence for $R_1$ recovery: $180{}^{\circ}$ pulse -- $\tau$ intermission -- $90{}^{\circ}$ pulse:}\\[1ex]
With $c^+(\vec r)=c^+_0\in\mathbb{R}^+$ and
\begin{equation}\label{180-tau-90_rect}
\begin{aligned}
&p(t)= e^{-i\omega_0t}\, {\pp}_{180^\circ-\tau-90^\circ}(t), \quad 
{\pp}_{180^\circ-\tau-90^\circ}=(c^+_0\gamma\taup)^{-1}\bigl(1\!\mathrm{I}_{[0,\taup]} + \frac12 1\!\mathrm{I}_{[\tau+\taup,\tau+2\taup]}\bigr),\\  
&\vec{G}(t)=0\text{ on } [0,\tau+2\taup],\quad
\Mpp(0)=0, \quad M_z(0)=M^{eq}
\end{aligned}
\end{equation}
in the Bloch ODE case \eqref{B_rect}, formulas \eqref{90-180_rect} and \eqref{MperpMz_B} yield
(noting that $\Re({\Mpp}(t,\vec r))\dot{\equiv}0$ results for all $t>0$)
\begin{equation}\label{init180tau90eps}
\begin{aligned}
&\Mpp(\taup,\cdot)\dot{=}-\Mpp(0,\cdot)=0\qquad
&&\Mpp(\taup+\tau,\cdot)\dot{=}0\qquad\\
&M_z(\taup,\cdot)\dot{=}-M_z(0,\cdot)=-M^{eq}
&&M_z(\taup+\tau,\cdot)\dot{=}(1-2e^{-R_1\tau})M^{eq}
\end{aligned}
\end{equation}
hence, 
\begin{equation}\label{init180tau90tauplus2eps}
\begin{aligned}
&\Mpp(\tau+2\taup,\cdot)\dot{=}-iM_z(\taup+\tau,\cdot)%\dot{=}-i(1-2e^{-R_1\tau})M^{eq}
\\
&M_z(\tau+2\taup,\cdot)\dot{=}\Im(\Mpp(\taup+\tau,\cdot))\dot{=}0,
\end{aligned}
\end{equation}
and, again using \eqref{MperpMz_B},
\begin{equation}\label{Mperp_ex_180-tau-90_rect_B}
\begin{aligned}
&\Mpp(t,\vec r)\\
&= 
\eexp{-(R_2^*(\vec r)(t-(\tau+2\taup))+2\pi i ({\vec k}(t)-{\vec k}(\tau+2\taup))\cdot{\vec r})}
%e^{-(R_2^*(\vec r)(t-(\tau+2\taup))+2\pi i {\vec k}(t)\cdot{\vec r})}
\,\Mpp(\tau+2\taup,\vec r)
\\
&\dot{=} 
-i\eexp{-(R_2^*(\vec r)(t-(\tau+2\taup))+2\pi i ({\vec k}(t)-{\vec k}(\tau+2\taup))\cdot{\vec r})}\,
%-ie^{-(R_2^*(\vec r)(t-(\tau+2\taup))+2\pi i {\vec k}(t)\cdot{\vec r})}\,
\\&\hspace*{7cm}
(1-2e^{-R_1(\vec r)\tau})M^{eq}(\vec r)\\
&=:{\Mpp}_{180^\circ-\tau-90^\circ}(t,\vec r)\qquad
\text{ for }\vec r\in\Omega, \quad t>\tau+2\taup.
\end{aligned}
\end{equation}
%since under our assumption $\vec{G}(t)=0\text{ on } [0,\tau+2\taup]$ we have ${\vec k}(\tau+2\taup)=0$.
Analogously to 
%\eqref{mtexplicit0},  % is in appendix
\eqref{MR-Fourier_rect_B}, 
we thus obtain  
\begin{equation}\label{obsPhi1_rect_B} 
\begin{aligned}
&-i\,\eexp{R_{2,0}^*(t-(\tau+2\taup))+ i\omega_0t} \, y_j(t)\\ 
&= \int_{\mathbb{R}^3} \eexp{-\delta R_2^*(\vec r)(t-(\tau+2\taup))} c^-_j(\vec r)\,\Phi_\tau(\vec r) \, \eexp{-2\pi i ({\vec k}(t)-{\vec k}(\tau+2\taup))\cdot{\vec r}}  \,\textsf{d}\vec r\\
&\approx \int_{\mathbb{R}^3} c^-_j(\vec r)\,\Phi_\tau(\vec r) \, \eexp{-2\pi i ({\vec k}(t)-{\vec k}(\tau+2\taup))\cdot{\vec r}}  \,\textsf{d}\vec r\\
&=(\mathcal{F}[c^-_j\,\Phi_\tau])(({\vec k}(t)-{\vec k}(\tau+2\taup)))
,\qquad j=1,\ldots N,
\\
&\text{for }t>\tau+2\taup\\
&\text{with }\Phi_\tau(\vec r):=(1-2e^{-R_1(\vec r)\tau})M^{eq}(\vec r)
\end{aligned}
\end{equation}
From this, $R_1(\vec r)\in\mathbb{R}$ can be recovered by using the already reconstructed $M^{eq}(\vec r)$. \revisionown{Alternatively, a formula for obtaining $R_1(\vec r)$ from two different intermission times $\tau\in\{\tau_1,\tau_2\}$ can be found in the appendix. 
%of \cite{MRI-uniqueness_arxiv}.
}

\kspace{
The k-space version with diffusion $D_0$ of \eqref{Mperp_ex_180-tau-90_rect_B} is 
\begin{equation}\label{Mperp_ex_180-tau-90_rect_BT0}
\begin{aligned}
&\Mphp(t,\vec\xi)\\
&\dot{=} 
-i\,\eexp{-\int_{\tau+2\taup}^t((\vec\xi+\vec k(t)-\vec k(\sigma))\cdot D_0(\vec\xi+\vec k(t)-\vec k(\sigma))+R_{2,0}^*)\, d\sigma}\, \widehat{\Phi}_\tau(\vec\xi+\vec k(t))
\\
&=:\widehat{\Mpp}_{180^\circ-\tau-90^\circ}(t,\vec\xi)\qquad
\text{ for }\vec\xi\in\mathbb{R}^3, \quad t>\tau+2\taup.
\end{aligned}
\end{equation}
cf. \eqref{Mperp_ex_90_rect_BT0} with 
\[
\widehat{\Phi}_\tau(\vec\xi)=
\frac{1}{\vec\xi\cdot D_0\vec\xi+R_{1,0}}
\Bigl(R_{1,0}-(\vec\xi\cdot D_0\vec\xi+2R_{1,0})\, \eexp{-(\vec\xi\cdot D_0\vec\xi+R_{1,0})\tau}\Bigr)\,\widehat{M^{eq}}(\vec\xi).
\]
}

\subsection{Perturbation results}

Part of the uniqueness results in Section~\ref{sec:uniqueness} will rely on closeness of the actual parameter-to-state maps to the simplified ones above that allowed to derive reconstruction formulas for $M^{eq}$, $R_1$ and $R_2^*$. 

These uniqueness proofs will not make use of the approximations $\approx$ for small $\delta R_2^*$ in 
\eqref{MR-Fourier_rect_B}, 
%\eqref{reconR2_rect},  % is in appendix
\eqref{obsPhi1_rect_B}, but only rely on smallness of $\taup$. 
The error resulting from the latter is quantified in the following subsection.

\subsubsection{Error estimates in the explicit 
Bloch ODE 
%(B) 
and k-space Bloch-Torrey 
%(BT${}_0$) 
solution formulas}\label{subsec:err-ex}
We first of all provide a quantification of the $\dot{=}$ statements above, based on Lemma~\ref{lem_rect}.
To this end, we define the actual and the approximately explicit  parameter-to-state maps 
%$\mathcal{S}\in\{\mathcal{S}_{90^\circ},\mathcal{S}_{180^\circ-\tau-90^\circ}\}$ 
\begin{equation}\label{S_90_180}
\begin{aligned}
&\mathcal{S}_{90^\circ}:(M^{eq},R_1,R_2^*) \mapsto \Mpp 
\text{ such that $(\Mpp,M_z)$ solves \eqref{bloch-rotated} with 
\eqref{90_rect}}\\
&\mathcal{S}_{90^\circ}^{ex}:(M^{eq},R_1,R_2^*) \mapsto {\Mpp}_{90^\circ}\quad
\text{ cf. \eqref{Mperp_ex_90_rect_B}}
\\[1ex]
&\mathcal{S}_{180^\circ-\tau-90^\circ}:(M^{eq},R_1,R_2^*) \mapsto \Mpp 
\text{ such that $(\Mpp,M_z)$ solves \eqref{bloch-rotated} with 
\eqref{180-tau-90_rect}}\\
&\mathcal{S}_{180^\circ-\tau-90^\circ}^{ex}:(M^{eq},R_1,R_2^*) \mapsto {\Mpp}_{180^\circ-\tau-90^\circ}\quad
\text{ cf. \eqref{Mperp_ex_180-tau-90_rect_B}}
\end{aligned}
\end{equation}

\revision{Lemma~\ref{lem_rect} directly yields the following perturbation result.}
\Margin{R1 D.} 

\begin{proposition}\label{prop:perturb_expl-sol}
Assume that $\text{supp}(\vec G)\subseteq[\taup,T]$ or $[\tau+2\taup,T]$, respectively. 
For the mappings $\mathcal{S}_{90^\circ}$, $\mathcal{S}_{90^\circ}^{ex}$, $\mathcal{S}_{180^\circ-\tau-90^\circ}$, $\mathcal{S}_{180^\circ-\tau-90^\circ}^{ex}$, defined above, there exists a constant $C_{\mathcal{S}_{ex}}>0$ depending only on $\|R_1\|_{L^\infty(\Omega)}$, $\|R_2^*\|_{L^\infty(\Omega)}$ such that
\[
\begin{aligned}
&\|(\mathcal{S}_{90^\circ}-\mathcal{S}_{90^\circ}^{ex})(M^{eq},R_1,R_2^*)\|_{L^\infty((\taup,T)\times\Omega)}\leq C_{\mathcal{S}_{ex}}\,\taup\|M^{eq}\|_{L^2(\Omega)}
\\
&\|(\mathcal{S}_{180^\circ-\tau-90^\circ}-\mathcal{S}_{180^\circ-\tau-90^\circ}^{ex})(M^{eq},R_1,R_2^*)\|_{L^\infty((\tau+2\taup,T)\times\Omega)}\leq C_{\mathcal{S}_{ex}}\,\taup\|M^{eq}\|_{L^2(\Omega)}.
\end{aligned}
\]
holds. 
\end{proposition}

A similar result holds for the derivatives of these maps, which are given by the solution 
$(\uldMpp,\uld{M}_z)=\mathcal{S}_{90^\circ}'(M^{eq},R_1,R_2^*)(\uld{M}^{eq},\uld{R}_1,\uld{R}_2^*)$ to 
\begin{equation}\label{Sprime90deg}
\begin{aligned}
&\frac{d}{dt}\uldMpp(t, \vec r) +\bigl(R_2^*(\vec r) + \vec r \cdot \vec G(t))\bigr)\uldMpp(t, \vec r)-i\gamma  \uld{M}_z(t, \vec r) \, c_0^+{\pp}_{90^\circ}(t)\\
&\qquad\qquad=-\uld{R}_2^*(\vec r)\Mpp(t, \vec r)
\\    
&\frac{d}{dt}\uld{M}_z(t, \vec r) +R_1(\vec r) \uld{M}_z(t, \vec r)-\gamma\Re\bigl(i \uldMpp(t, \vec r)\overline{c_0^+{\pp}_{90^\circ}(t)}\bigr)\\
&\qquad\qquad=-\uld{R}_1(\vec r) (M_z(t, \vec r)-M^{eq}(\vec r))+R_1(\vec r)\uld{M}^{eq}(\vec r))
\end{aligned}
\end{equation}
and its limit as $\taup\to0$, 
$(\uldMpp^{ex},\uld{M}_z^{ex})={\mathcal{S}_{90^\circ}^{ex\,\prime}}(M^{eq},R_1,R_2^*)(\uld{M}^{eq},\uld{R}_1,\uld{R}_2^*)$, respectively; 
likewise for $\mathcal{S}_{180^\circ-\tau-90^\circ}'$, ${\mathcal{S}_{180^\circ-\tau-90^\circ}^{ex\,\prime}}$. 
\begin{proposition}\label{prop:perturb_expl-sol_lin}
There exists a constant $C_{\mathcal{S}_{ex}'}>0$ depending only on $\|R_1\|_{L^\infty(\Omega)}$, $\|R_2^*\|_{L^\infty(\Omega)}$ such that
\[
\begin{aligned}
&\|(\mathcal{S}_{90^\circ}'-{\mathcal{S}_{90^\circ}^{ex\,\prime}})(M^{eq},R_1,R_2^*)\|_{
L^2(\Omega)\times L^\infty(\Omega)^2\,\to\,L^\infty((\taup,T)\times\Omega)}\leq C_{\mathcal{S}_{ex}'}\,\taup\|M^{eq}\|_{L^2(\Omega)}
\\
&\|(\mathcal{S}_{180^\circ-\tau-90^\circ}'-{\mathcal{S}_{180^\circ-\tau-90^\circ}^{ex\,\prime}})(M^{eq},R_1,R_2^*)\|_{L^2(\Omega)\times L^\infty(\Omega)^2\,\to\,L^\infty((\tau+2\taup,T)\times\Omega)}
\\&\hspace*{8.7cm}
\leq C_{\mathcal{S}_{ex}'}\,\taup\|M^{eq}\|_{L^2(\Omega)}.
\end{aligned}
\]
\end{proposition}
\begin{proof}
See the appendix. 
%of \cite{MRI-uniqueness_arxiv}.
\Margin{R1 C.}
\Margin{R2 (3)}
\end{proof}

\begin{remark}\label{rem:perturb_kspace} 	%2025-06-12
Due to Lemma~\ref{lem_rect}, analogous perturbation results hold for the parameter-to-state operators $\mathcal{S}^{D_0}_{90^\circ}$, $\mathcal{S}^{D_0}_{180^\circ-\tau-90^\circ}$ relying on the explicit formulas in k-space \eqref{Mperp_ex_90_rect_BT0} and \eqref{Mperp_ex_180-tau-90_rect_BT0}.
\end{remark} 

\subsubsection{Smoothness of the Bloch-Torrey parameter-to-state map}\label{subsec:S}

Next, we broaden dependency of the parameter-to-state map and aim to establish (Lipschitz) continuity with respect to all the relevant quantities, that is, we define
\begin{equation}\label{S_gen}
\begin{aligned}
&\mathcal{S}:(x^{img};x^{mod})\mapsto \Mpp \text{ such that $(\Mpp,M_z)$ solves \eqref{eqn:bloch-torrey-complex-rotated}} \\  
&\text{where }x^{img}=(M^{eq},R_1,R_2^*), \quad x^{mod}=(D,\vec v,c^+,\vec G,\pp)
\end{aligned}
\end{equation}
Results of this kind will not only be an essential ingredient to our uniqueness proof in Section~\ref{sec:uniqueness_perturbation}, but also be useful for analyzing reconstruction methods as well as optimal experimental design of the system via the controls $\vec G$ and $\pp$. 

Since we will make use of Proposition~\ref{prop:wellposed} in this subsection, we assume $\Omega$ to be a bounded Lipschitz domain.

To first of all derive Lipschitz continuity of  $\mathcal{S}$ in appropriate function spaces, we apply the energy estimate \eqref{enest} to the system satisfied by the difference 
$(\delta \Mpp,\delta M_z)$ such that $\delta \Mpp=\mathcal{S}(x^{img};x^{mod})-\mathcal{S}(\tilde{x}^{img};\tilde{x}^{mod})$, which is \eqref{eqn:bloch-torrey-complex_rot_f} with $(\Mpp,M_z)$ replaced by $(\delta \Mpp,\delta M_z)$ and
\begin{equation}\label{fdiff}
\begin{aligned}
-\fpp(t, \vec r) =& 
({\vec v}(t, \vec r)-\vec{\tilde{v}}(t, \vec r))\cdot\nabla \Mpptil(t, \vec r) -\nabla\cdot\bigl((D(\vec r)-\tilde{D}(\vec r)) \nabla \Mpptil(t, \vec r)\bigr)\\
&+\bigl((R_2^*-\tilde{R}_2^*)+i\gamma \vec r \cdot (\vec G(t)-\vec{\tilde{G}}(t))\bigr)\Mpptil(t, \vec r)\\
&-i\gamma \tilde{M}_z(t, \vec r)\bigl((c^+-\tilde{c}^+)\pp+\tilde{c}^+(\pp-\pptil)\bigr)\\
-f_z(t, \vec r)=&
({\vec v}(t, \vec r)-\vec{\tilde{v}}(t, \vec r))\cdot\nabla \tilde{M}_z(t, \vec r) -\nabla\cdot\bigl((D(\vec r)-\tilde{D}(\vec r)) \nabla \tilde{M}_z(t, \vec r)\bigr)\\ 
&+(R_1(\vec r)-\tilde{R}_1(\vec r))\tilde{M}_z(t, \vec r)\\
&-\gamma\Re\bigl(i\tilde{M}_\perp(t, \vec r)\overline{\bigl((c^+(\vec r)-\tilde{c}^+(\vec r))\pp(t)+\tilde{c}^+(\vec r)(\pp(t)-\pptil(t))\bigr)}\bigr)\\
&-R_1(\vec r)(M^{eq}(\vec r)-\tilde{M}^{eq}(\vec r))-(R_1(\vec r)-\tilde{R}_1(\vec r))\tilde{M}^{eq}(\vec r).
\end{aligned}
\end{equation}
Indeed we can estimate the individual $f$ terms in $L^2(0,T;H_D^1(\Omega)^*)$ or in $L^1(0,T;L^2(\Omega))$, noting that due to linearity of the PDE this yields an overall estimate by superposition.
With $m\in\{\Mpptil,\tilde{M}_z\}$ we have
\begin{equation}\label{estv}
\begin{aligned}
&\|({\vec v}-\vec{\tilde{v}})\cdot\nabla m\|_{L^q(0,T;L^2(\Omega))}
=\|\bigl(D^{-1/2}({\vec v}-\vec{\tilde{v}})\bigr)\cdot  D^{1/2}\nabla m\|_{L^q(0,T;L^2(\Omega))}\\
&\leq \|D^{-1/2}({\vec v}-\vec{\tilde{v}})\|_{L^{\hat{q}}(0,T;L^\infty(\Omega))} \|m\|_{L^2(0,T;H_D^1(\Omega))}, \quad q\in\{1,2\}, \quad\hat{q}= \frac{2q}{2-q}
\end{aligned}
\end{equation}
\begin{equation*}%\label{estD}
\begin{aligned}
&\|-\nabla\cdot\bigl((D-\tilde{D})\nabla m\bigr)\|_{L^2(0,T;H_D^1(\Omega)^*)}\\
&\skipsimple{
=\Bigl(\int_0^T\Bigl(\sup_{w\in C_c^\infty(\Omega)\setminus\{0\}} \frac{1}{\|w\|_{H_D^1(\Omega)}}
\int_\Omega (D^{-1/2}(D-\tilde{D})D^{-1/2}\, D^{1/2}\nabla m) \cdot D^{-1/2}\nabla w\, d{\vec r}\Bigr)^2 \, dt\Bigr)^{1/2}
}
%\\ &
\leq \|D^{-1/2}(D-\tilde{D})D^{-1/2}\|_{L^\infty(\Omega;\mathbb{R}^{3\times 3})}\, \|m\|_{L^2(0,T;H_D^1(\Omega))}
\end{aligned}
\end{equation*}
\begin{equation}\label{estR}
\begin{aligned}
&\|R\,m\|_{L^q(0,T;L^2(\Omega)}\leq \|R\|_{L^q(0,T;L^\infty(\Omega)}\|m\|_{L^\infty(0,T;L^2(\Omega)}, \quad q\in\{1,2\}\\ 
&R\in\{R_1-\tilde{R}_1,\,R_2^*-\tilde{R}_2^*,\,\vec r \cdot (\vec G(t)-\vec{\tilde{G}}(t)),\,(c^+-\tilde{c}^+)\pp+\tilde{c}^+(\pp-\pptil)\}. 
\end{aligned}
\end{equation}
Finiteness of the ($D$ dependent) norms defined by
\[ 
\begin{aligned}
&\|\delta\vec v\|_{L^\infty_{D^{-1/2}}(\Omega;\mathbb{R}^3)}:=
\|D^{-1/2}{\delta\vec v}\|_{L^\infty(\Omega)}, 
\\
&\|\delta D\|_{L^\infty_{D^{-1/2}}(\Omega;\mathbb{R}^{3\times 3})}:=
\|D^{-1/2}\delta D\,D^{-1/2}\|_{L^\infty(\Omega;\mathbb{R}^{3\times 3})}
\end{aligned}
\]
in the estimates above is enabled, e.g., by a scenario where in regions of degenerate $D$
no motion occurs and the diffusion tensor does not change.  

Corresponding to these estimates and the regularity assumptions from Proposition~\ref{prop:wellposed}, we use the parameter spaces 
\begin{equation}\label{X}
\begin{aligned}
&\mathbb{X}^{img}=L^2(\Omega;\mathbb{R})\times L^\infty(\Omega;\mathbb{R})\times L^\infty(\Omega;\mathbb{C})\\
&\mathbb{X}^{mod}=L^\infty_{D^{-1/2}}(\Omega;\mathbb{R}^{3\times 3})\times L^2(0,T;L^\infty_{D^{-1/2}}(\Omega;\mathbb{R}^3))\times L^\infty(\Omega;\mathbb{C})\\
&\hspace*{7cm}\times L^1(0,T;\mathbb{R}^3)\times L^1(0,T)
\end{aligned}
\end{equation}
the domain
\begin{equation}\label{DS}
\begin{aligned}
\mathcal{D}(\mathcal{S})=\{&(M^{eq},R_1,R_2^*,D,\vec v,c^+,\vec G,\pp)\in \mathbb{X}^{img}\times \mathbb{X}^{mod}\, : \,\\ 
&\vec v\in L^1(0,T;W^{1,\infty}(\Omega;\mathbb{R}^3)) \text{ and \eqref{definiteness}, \eqref{divv}  hold}\}
\end{aligned}
\end{equation}
and the solution space (cf. \eqref{V_wellposed})
\begin{equation}\label{V}
\mathbb{V}_\perp=L^\infty(0,T;L^2(\Omega;\mathbb{C}))\cap L^2(0,T;H_D^1(\Omega;\mathbb{C})).
\end{equation}

\revision{With this, we can conclude from \eqref{fdiff}, \eqref{estv}, \eqref{estR} the following.}
\Margin{R1 E.}

\begin{proposition}\label{prop:perturb_gen}
Let $\Omega$ to be a bounded Lipschitz domain.
For all ${x}^{img}$, $\tilde{x}^{img}$ $\in\mathbb{X}^{img}$, ${x}^{mod}$, $\tilde{x}^{mod}$ $\in\mathbb{X}^{mod}$, there exists $C_{\mathcal{S}}>0$ depending only on $\|\tilde{x}^{img}\|_{\mathbb{X}^{img}}$, $\|{x}^{mod}\|_{\mathbb{X}^{mod}}$, $\|\tilde{x}^{mod}\|_{\mathbb{X}^{mod}}$ such that 
\[
\|\mathcal{S}(x^{img};x^{mod})-\mathcal{S}(\tilde{x}^{img};\tilde{x}^{mod})\|_{\mathbb{V}_\perp}
\leq C_{\mathcal{S}}(\|x^{img}-\tilde{x}^{img}\|_{\mathbb{X}^{img}}+\|x^{mod}-\tilde{x}^{mod}\|_{\mathbb{X}^{mod}})
\]
holds.
The same holds true in a higher regularity setting where $\mathbb{X}^{mod}$, $\mathbb{V}_\perp$ are replaced by 
\begin{equation}\label{XVhi}
\begin{aligned}
&\mathbb{X}^{mod}_{hi}=L^\infty_{D^{-1/2}}(\Omega;\mathbb{R}^{3\times 3})\times L^\infty(0,T;L^\infty_{D^{-1/2}}(\Omega;\mathbb{R}^3))\times L^\infty(\Omega;\mathbb{C})\\
&\hspace*{7cm}\times L^2(0,T;\mathbb{R}^3)\times L^2(0,T)\\
&\mathbb{V}_{\perp\,hi}=H^1(0,T;H_D^1(\Omega)^*)\cap \mathbb{V}_\perp
\end{aligned}
\end{equation}
\end{proposition}
As can be read off from \eqref{XVhi}, somewhat higher temporal regularity in $\mathbb{X}^{mod}_{hi}$ is required to enable solutions in $\mathbb{V}_{\perp\,hi}$.

Note that 
%the $H_D^1$ norm on the left hand side and the $L^\infty_{D^{-1/2}}$ norms on the right hand side of this estimate 
\revisionown{the norms defined in \eqref{X}, \eqref{V}, \eqref{XVhi}}
depend on $x^{mod}$ via $D$. Under a uniform positive definiteness condition on $D$ included in the definition of the domain $\mathcal{D}(\mathcal{S})$ \revisionown{\eqref{DS}} they could be made independent of $x^{mod}$.\\
Dependence of $C_{\mathcal{S}}$ on $\|{x}^{mod}\|_{\mathbb{X}^{mod}}$ \textit{and} $\|\tilde{x}^{mod}\|_{\mathbb{X}^{mod}}$ is due to the trilinear terms $M_z(t, \vec r)c^+(\vec r)\pp$, $\Re\bigl(i\Mpp(t, \vec r)\overline{c^+(\vec r)\pp(t)}\bigr)$. All other terms are bilinear.

A similar result holds for the derivative $(\uldMpp,\uld{M}_z)=$ $\mathcal{S}'(x^{img};x^{mod})$ $(\uld{x}^{img};\uld{x}^{mod})$, which is given as the solution to \eqref{eqn:bloch-torrey-complex_rot_f} with $(\Mpp,M_z)$ replaced by $(\uldMpp,\uld{M}_z)$ and
\begin{equation}\label{fderiv}
\begin{aligned}
&-\fpp(t, \vec r)= 
\uld{\vec v}(t, \vec r)\cdot\nabla \Mpp(t, \vec r) 
-\nabla\cdot\bigl(\uld{D}(\vec r) \nabla \Mpp(t, \vec r)\bigr) \\
&+\bigl(\uld{R}_2^*(\vec r)+i\gamma\vec r \cdot \uld{\vec G}(t)\bigr)\Mpp(t, \vec r)
-i\gamma M_z(t, \vec r)(\uld{c}^+(\vec r)\pp(t)+c^+(\vec r)\uld{\pp}(t))
\\[1ex]    
&-f_z(t, \vec r)=\uld{\vec v}(t, \vec r)\cdot\nabla M_z(t, \vec r) -\nabla\cdot\bigl(\uld{D}(\vec r) \nabla M_z(t, \vec r)\bigr) \\
&+\uld{R}_1(\vec r) (M_z(t, \vec r)-M^{eq}(\vec r))-R_1(\vec r) \uld{M}^{eq}(\vec r)\\
&-\gamma\Re\bigl(i\Mpp(t, \vec r)\overline{(\uld{c}^+(\vec r)\pp(t)+c^+(\vec r)\uld{\pp}(t))}\bigr)
\end{aligned}
\end{equation}
cf. \eqref{fdiff}, where $\Mpp=\mathcal{S}(x^{img};x^{mod})$ and $M_z$ is the corresponding longitudinal part.

\begin{proposition}\label{prop:perturb_gen_deriv}
Let $\Omega$ to be a bounded Lipschitz domain.
For all ${x}^{img}$, $\tilde{x}^{img}$ $\in\mathbb{X}^{img}$, ${x}^{mod}$, $\tilde{x}^{mod}$ $\in\mathbb{X}^{mod}$, $\mathcal{S}$ is Fr\'{e}chet differentiable at $(x^{img};x^{mod})$ and $(\tilde{x}^{img};\tilde{x}^{mod})$, and there exists $C_{\mathcal{S}'}>0$ depending only on $\|\tilde{x}^{img}\|_{\mathbb{X}^{img}}$, $\|{x}^{mod}\|_{\mathbb{X}^{mod}}$, $\|\tilde{x}^{mod}\|_{\mathbb{X}^{mod}}$ such that 
\[
\begin{aligned}
&\|\mathcal{S}'(x^{img};x^{mod})-\mathcal{S}'(\tilde{x}^{img};\tilde{x}^{mod})\|_{L(\mathbb{X}^{img}\times\mathbb{X}^{mod},\mathbb{V}_\perp)}\\
&\leq C_{\mathcal{S}'}(\|x^{img}-\tilde{x}^{img}\|_{\mathbb{X}^{img}}+\|x^{mod}-\tilde{x}^{mod}\|_{\mathbb{X}^{mod}})
\end{aligned}
\]
holds.
\end{proposition}
\begin{proof}
Fr\'{e}chet differentiability is straightforward to see, relying on the bilinearity of most of the terms and the fact that one of the spaces underlying the only trilinear term (the one involving ${c}^+\pp$) is a Banach algebra (namely $L^\infty(\Omega)$ for ${c}^+$). 

To show the Lipschitz estimate, we use Proposition~\ref{prop:wellposed} together with the fact that
$(\ulddMpp,\uld{dM}_z)=[\mathcal{S}'(x^{img};x^{mod})-\mathcal{S}'(\tilde{x}^{img};\tilde{x}^{mod})](\uld{x}^{img};\uld{x}^{mod})$ solves \eqref{eqn:bloch-torrey-complex_rot_f} with $(\Mpp,M_z)$ replaced by $(\ulddMpp,\uld{dM}_z)$ and 
\[
\begin{aligned}
-\fpp(t, \vec r)=& 
[{\vec v}-\vec{\tilde{v}}](t, \vec r)\cdot\nabla \uldMpp(t, \vec r) 
-\nabla\cdot\bigl([D-\tilde{D}](\vec r) \nabla \uldMpp(t, \vec r)\bigr) \\
&+\bigl([R_2^*-\tilde{R}_2^*](\vec r)+i\gamma\vec r \cdot [{\vec G}-\vec{\tilde{G}}](t)\bigr)\uldMpp(t, \vec r)\\
&-i\gamma \uld{M}_z(t, \vec r)([c^+-\tilde{c}^+](\vec r)\pp(t)+c^+(\vec r)[\pp-\pptil](t))\\
&+\uld{\vec v}(t, \vec r)\cdot\nabla [\Mpp-\Mpptil](t, \vec r) 
-\nabla\cdot\bigl(\uld{D}(\vec r) \nabla [\Mpp-\Mpptil](t, \vec r)\bigr) \\
&+\bigl(\uld{R}_2^*(\vec r)+i\gamma\vec r \cdot \uld{\vec G}(t)\bigr)[\Mpp-\Mpptil](t, \vec r)\\
&-i\gamma [M_z-\tilde{M}_z](t, \vec r)(\uld{c}^+(\vec r)\pp(t)+c^+(\vec r)\uld{\pp}(t))
\\[1ex]    
-f_z(t, \vec r)&
=[{\vec v}-\vec{\tilde{v}}](t, \vec r)\cdot\nabla \uld{M}_z(t, \vec r) 
-\nabla\cdot\bigl([D-\tilde{D}](\vec r) \nabla \uld{M}_z(t, \vec r)\bigr) \\
&+[R_1-\tilde{R}_1](\vec r) (\uld{M}_z(t, \vec r)-M^{eq}(\vec r))\\
&-\gamma\Re\bigl(i\uldMpp(t, \vec r)\overline{([c^+-\tilde{c}^+](\vec r)\pp(t)+c^+(\vec r)[\pp-\pptil](t))}\bigr)\\
&+\uld{\vec v}(t, \vec r)\cdot\nabla M_z(t, \vec r) -\nabla\cdot\bigl(\uld{D}(\vec r) \nabla M_z(t, \vec r)\bigr) \\
&+\uld{R}_1(\vec r) [M_z-\tilde{M}_z](t, \vec r)\\
&-\gamma\Re\bigl(i\Mpp(t, \vec r)\overline{(\uld{c}^+(\vec r)\pp(t)+c^+(\vec r)\uld{\pp}(t))}\bigr),
\end{aligned}
\]
where $(\uldMpp,\uld{M}_z):=\mathcal{S}'(x^{img};x^{mod})(\uld{x}^{img};\uld{x}^{mod})$,
$(\Mpp,M_z):=\mathcal{S}(x^{img};x^{mod})$,
$(\Mpptil,\tilde{M}_z):=\mathcal{S}(\tilde{x}^{img};\tilde{x}^{mod})$.
Each of these terms can be estimated analogously to \eqref{estv}
%, \eqref{estD}, 
-- \eqref{estR}.
\end{proof}
%Finally, we also estimate the error indicated by $\approx$ in \eqref{MR-Fourier_rect_B}, \eqref{reconR2_rect}, \eqref{obsPhi1_rect_B}; Quantification of $\approx$ in the reconstruction formulas by an $O(\delta R_2^*)$ estimate. This amounts to an approximation of the overall forward operator $F_{ex}:(x^{img};x^{mod})\mapsto y$.

\section{Uniqueness}\label{sec:uniqueness}
In this section we will prove two types of uniqueness results for $x^{img}=(M^{eq},R_1,R_2^*)$.
In Section~\ref{sec:uniqueness_perturbation}, uniqueness and stability of reconstruction  \revisionown{will be established by means of a perturbation argument and the (approximate) reconstruction formulas from Section~\ref{sec:explsol} for the Bloch ODE model.} This will be carried out in a reduced formulation of the inverse problem, involving the parameter-to-state map.
Differently from that, Section~\ref{sec:uniqueness_diffusion} employs methodology typically used for proving uniqueness of space dependent source terms and initial conditions in (sub-)diffusion equations. Here we will consider an all-at-once formulation of the inverse problems, that views states and parameters as unknowns in a larger system of equations containing the PDE models and the observation equations. 
The linearization of the so defined forward operator at properly chosen reference states and parameters will be shown to be injective.
 
%Throughout this section we assume that $\vec G$ is piecewise continuous, hence $\vec k$ defined by \eqref{kt} continuous and piecewise differentiable.  

\subsection{Uniqueness and stability via perturbation of Bloch equation}\label{sec:uniqueness_perturbation}
The inverse problem of reconstructing $x^{img}=(M^{eq},R_1,R_2^*)$ in the Bloch-Torrey equation as space-dependent functions can be written as an operator equation
\begin{equation}\label{Fxy}
F(x^{img})=\underline{y}
\end{equation}
where 
\begin{equation}\label{F_gen}
F(x^{img})=\revisionown{F^{BT}(x^{img})=}\left(\begin{array}{c}
\mathcal{C}\mathcal{S}(x^{img};\check{x}^{mod},{\pp}_{(I)})\\
\mathcal{C}\mathcal{S}(x^{img};\check{x}^{mod},{\pp}_{(II)})\\
\mathcal{C}\mathcal{S}(x^{img};\check{x}^{mod},{\pp}_{(III)})
\end{array}\right),
\end{equation}
\revisionown{defined via solutions of the \underline{B}loch-\underline{T}orrey PDE}
cf. \eqref{S_gen}, with fixed model parameters $\check{x}^{mod}=(D,\vec v,c^+,\vec G)$, whereas variability is allowed in the last component of ${x}^{mod}=(\check{x}^{mod},\pp)$.
In case  
${\pp}_{(I)}\approx {\pp}_{90^\circ}$ cf. \eqref{90_rect},
${\pp}_{(II)}\approx {\pp}_{180^\circ-\tauone-90^\circ}$, 
${\pp}_{(III)}\approx {\pp}_{180^\circ-\tautwo-90^\circ}$, cf. \eqref{180-tau-90_rect}, this can be viewed as a perturbation of 
\begin{equation}\label{F}
\revisionown{F^{B}}(x^{img})=\left(\begin{array}{c}
\mathcal{C}\mathcal{S}_{90^\circ}(x^{img})\\
\mathcal{C}\mathcal{S}_{180^\circ-\tauone-90^\circ}(x^{img})\\
\mathcal{C}\mathcal{S}_{180^\circ-\tautwo-90^\circ}(x^{img})
\end{array}\right),
\end{equation}
\revisionown{defined via solutions of the \underline{B}loch ODE}
cf. \eqref{S_90_180}, which, in its turn is close (in the sense of Proposition~\ref{prop:perturb_gen}) to the explicit formula version for the Bloch ODE model
\begin{equation}\label{F_ex}
F^{ex}(x^{img})=\left(\begin{array}{c}
\mathcal{C}\mathcal{S}_{90^\circ}^{ex}(x^{img})\\
\mathcal{C}\mathcal{S}_{180^\circ-\tauone-90^\circ}^{ex}(x^{img})\\
\mathcal{C}\mathcal{S}_{180^\circ-\tautwo-90^\circ}^{ex}(x^{img})
\end{array}\right).
\end{equation}
%We will also make use of a version 
%\begin{equation}\label{F}
%F^{ex}_{D_0}(x^{img})=\left(\begin{array}{c}
%\mathcal{C}\mathcal{S}^{D_0}_{90^\circ}(x^{img})\\
%\mathcal{C}\mathcal{S}^{D_0}_{180^\circ-\tauone-90^\circ}(x^{img})\\
%\mathcal{C}\mathcal{S}^{D_0}_{180^\circ-\tautwo-90^\circ}(x^{img})
%\end{array}\right),
%\end{equation}
%relying on the explicit formulas in k-space \eqref{Mperp_ex_90_rect_BT0} and \eqref{Mperp_ex_180-tau-90_rect_BT0}.

Here, the observation operator $\mathcal{C}$ is defined according to \eqref{eqn:mri_measurement}, but in the rotated frame
\begin{equation}\label{obsop}    
\begin{aligned}
\mathcal{C}:\,&\mathbb{V}_\perp\to\mathbb{Y}^N\\ 
&\Mpp\mapsto \Bigl(\int_{\mathbb{R}^3} c^-_j(\vec r) \Mpp(\cdot, \vec r) \,\textsf{d}\vec r\Bigr)_{j=1,\ldots N}
\end{aligned}
\end{equation}
and the parameter-to-state maps according to \eqref{S_90_180} or \eqref{S_gen}.

\subsubsection{The MRI isomorphism}\label{subsec:MRIiso}
Since it is a fundamental building block of the explicit (approximate) reconstruction formulas \eqref{MR-Fourier_rect_B}, 
%\eqref{reconR2_rect}, % is in appendix
\eqref{obsPhi1_rect_B}, see also Section~\ref{subsec:MRIex} below, we consider the interpolation problem 
\[
\begin{aligned}
&\text{Given $y\in\mathbb{Y}\subseteq L^2(0,T)$, determine } x\in \Xpar \subseteq L^2(\Omega)\cap L^\infty(\Omega)\\
&\text{such that }\forall t\in[0,T]: \  (\mathcal{F}x)(\vec{k}(t))=y(t).
\end{aligned} 
\]
The Fourier transform is one of the most prominent isomorphisms between $L^2(\mathbb{R}^3)$ and itself.
However, here we only have its sampled version $\mathcal{F}\vert_{\mathcal{K}}$ along the curve 
\begin{equation}\label{K}
\mathcal{K}=\{\vec k(t):t\in(0,T)\}
\end{equation}
defined by a fixed given field gradient $\vec G(t)$ via \eqref{kt}, and we have to take into account the fact that this curve is traversed chronologically, with a regularity induced by the parameter-to-state map.\footnote{In our setting the maximal temporal regularity being $H^1(0,T)$ unless we impose higher regularity on the coefficients; we will actually use the $L^\infty(0,T)$ topology that requires less parameter regularity; see the definition of $\mathbb{V}_\perp$ as compared to $\mathbb{V}_{\perp\,hi}$ in \eqref{V}, \eqref{XVhi}.}    
We thus restrict attention to a linear subspace $\Xpar$ of $L^2(\Omega)\cap L^\infty(\Omega)$ that allows unique and stable inversion of the sampled Fourier transform.
An infinite dimensional choice yielding uniqueness would be the space of analytic functions; however, reconstruction of an analytic function from its values along a curve is unstable in general 
\revision{cf., e.g., \cite{LavrentievRomanovShishatskij1986}.}
\Margin{R2 7)}

One can derive a maximal choice of $\Xpar$ yielding both uniqueness and stability by a metric projection argument as follows.
For given $\mathcal{K}$ \eqref{K} with ${\vec k}$ according to \eqref{kt} for ${\vec G}\in C^{s-1}$, thus satisfying the regularity assumptions of the Trace and Inverse Trace Theorems, 
%\TODO{cf. \cite[Theorem*** with $m=???$, $k=???$]{AdamsFournier:2003},} 
we have that 
\[
\text{tr}_{\mathcal{K}}:H^s(\mathbb{R}^3)\to H^{s-1}(\mathcal{K}), \quad
\hat{x}\mapsto\hat{x}\circ{\vec k}
\]
(note the descent by two space dimensions) is bounded and has bounded right inverse $\text{tr}_{\mathcal{K}}^{-1}:H^{s-1}(\mathcal{K})\to H^s(\mathbb{R}^3)$, 
so that $\text{tr}_{\mathcal{K}}\text{tr}_{\mathcal{K}}^{-1}=\text{id}_{H^{s-1}(\mathcal{K})}$.
A stable choice of $\Xpar$ is therefore induced by the minimization problem
\[
x=\mathcal{F}^{-1}\hat{x} \text{ with }\hat{x}\in\text{argmin}\{\|\hat{x}\|_{H^s(\mathbb{R}^3)}\, : \,  \text{ such that }\text{tr}_{\mathcal{K}}\hat{x}=y\},
\]
as 
\[
\Xpar=\mathcal{F}^{-1}\mathcal{N}(\text{tr}_{\mathcal{K}})^{\bot_{H^s}}
=\mathcal{F}^{-1}\mathcal{R}(\text{tr}_{\mathcal{K}}^*)
\]
due to the Closed Range Theorem 
and the fact that $\mathcal{R}(\text{tr}_{\mathcal{K}})=H^{s-1}(\mathcal{K})$ is closed.
Here $\mathcal{N}(A)$ and $\mathcal{R}(A)$ denote the nullspace and range of a linear operator $A$, ${}^{\bot_Z}$ the orthogonal complement in the topology of the Hilbert space $Z$, and ${}^*$ the Hilbert space adjoint. 
In the function space $\check{H}^s(\mathbb{R}^3)=\mathcal{F}^{-1}H^s(\mathbb{R}^3)$, equipped with the norm
\[
\|x\|_{\check{H}^s(\mathbb{R}^3)}
=\left(\int_{\mathbb{R}^3}(1+|\vec r|^2)^s\, |x(\vec r)|^2\, d\vec r\right)^{1/2},
\]
this choice therefore satisfies the stability estimate
\[
\sup_{x\in\Xpar\setminus\{0\}} \|x\|_{\check{H}^s(\mathbb{R}^3)}\,\|(\mathcal{F}x)\circ{\vec k}\|_{H^{s-1}(0,T)}^{-1}\,<\infty.
\]
With a bounded domain $\Omega\subseteq\mathcal{B}_\varrho(0)$ we can make use of equivalence of Sobolev norms for band limited functions 
$\|x\|_{\check{H}^s(\mathbb{R}^3)}\leq (1+\varrho^2)^{s/2}\|x\|_{L^2(\mathbb{R}^3)}$, 
to obtain, for 
\[
\Xpar=\mathcal{F}^{-1}\mathcal{N}(\text{tr}_{\mathcal{K}})^{\bot_{H^s}}\cap \{x\in L^2(\Omega)\, : \, x=0\text{ in }\mathbb{R}^3\setminus\Omega\}
\]
the estimate 
\begin{equation}\label{CI_idealcoils}
\sup_{x\in\Xpar\setminus\{0\}} \|x\|_{L^2(\Omega)}\,\|(\mathcal{F}x)\circ{\vec k}\|_{H^{s-1}(0,T)}^{-1}\,<\infty.
\end{equation}

Since we have to take into account the parallel acquisition scheme as well as the coil sensitivities, the interpolation operator $\mathcal I_{\mathcal{K}}$ we will actually need is defined by 
\[
\mathcal I_{\mathcal{K}}(y)=x \text{ with }x\text{ such that }\forall\,j\in\{1,\ldots,N\}\,\forall t\in[0,T]: \  (\mathcal{F}[c^-_j x]) ({\vec k}(t))=y_j(t)
\]
and, extending \eqref{CI_idealcoils}, we assume it to be a bounded operator
\begin{equation}\label{CI}
\sup_{x\in\Xpar\setminus\{0\}} \|x\|_{L^2(\Omega)}\,\left(\sum_{j=1}^N\|(\mathcal{F}[c^-_j x])\circ{\vec k}\|_{\mathbb{Y}}\right)^{-1}\,=:C_{\mathcal{I}}<\infty.
\end{equation}
In case of nonvanishing diffusion, we generalize this to 
\begin{equation}\label{CI_D0}
\sup_{x\in\Xpar\setminus\{0\}} \|x\|_{L^2(\Omega)}\,\left(\sum_{j=1}^N\|\int_{\mathbb{R}^3}\widetilde{c^-_j}(\cdot,\vec\xi) \mathcal{F}[x](\vec\xi)\, d\vec\xi\|_{\mathbb{Y}}\right)^{-1}\,=:C_{\mathcal{I}}<\infty,
\end{equation}
where 
\[
\widetilde{c^-_j}(t,\vec\xi)
=\mathcal{F}[c_j^-](\vec\xi-\vec k(t))\,
e^{-\int_{\taup}^t(\vec\xi-\vec k(\sigma))\cdot D_0(\vec\xi-\vec k(\sigma))\, d\sigma}
\]
cf. \eqref{MR-Fourier_rect_BT0}, and $D_0$ is a constant reference diffusion tensor; in case of $D_0=0$, obviously \eqref{CI} is recovered.

In \eqref{CI}, \eqref{CI_D0}, as mentioned above, the definition of the solution space $\mathbb{V}_\perp$ according to \eqref{V} restricts the strength of the topology $\mathbb{Y}$ so that the observation operator $\mathcal{C}:\mathbb{V}_\perp\to\mathbb{Y}^{N}$ is bounded
\begin{equation}\label{CC}
\sup_{v\in\mathbb{V}_\perp\setminus\{0\}} \|\mathcal{C}v\|_{\mathbb{Y}^{N}}\,{\|v\|_{\mathbb{V}_\perp}}^{-1}\,=:C_{\mathcal{C}}<\infty, 
\end{equation}
that is, 
\begin{equation}\label{Y}
\mathbb{Y}= L^q(0,T)\text{ for some }q\in[2,\infty].
\end{equation} 
Since only values, but not derivative values of $y_j$ can be measured, this is anyway a reasonable data space topology from a practical point of view.

\subsubsection{Recovery of $M^{eq}$, $R_1$ and $R_2^*$ from Bloch ODE explicit solution formulas}\label{subsec:MRIex}

The aim of this subsection is to prove a local Lipschitz stability estimate
\begin{equation}\label{stab_ex}
\forall x^{img},\tilde{x}^{img}\, \in \Xpar^3\cap U:
\|x^{img}-\tilde{x}^{img}\|_{\mathbb{X}^{img}}\leq C_{ex}
\|F^{ex}(x^{img})-F^{ex}(\tilde{x}^{img})\|_{\mathbb{Y}^{3N}}
\end{equation}
for $\mathbb{X}^{img}$, $F^{ex}$ defined in \eqref{X}, \eqref{F_ex}.
For this purpose, we show regularity of the linearization 
\begin{equation}\label{stab_ex_lin}
\forall  \uld{x}\in \Xpar^3:
\|\uld{x}^{img}\|_{\mathbb{X}^{img}}\leq C_{ex}'
\|{F^{ex}}'(x^{img}_{\text{ref}})\uld{x}^{img}\|_{\mathbb{Y}^{3N}} 
\end{equation}
at a reference point 
$x^{img}_{\text{ref}}=(M^{eq}_{\text{ref}},R_{1,\text{ref}},R_{2,\text{ref}}^*)$ which we choose such that 
\begin{equation}\label{MrefR1refR2ref}
|M^{eq}_{\text{ref}}(\vec r)|\geq \underline{M}^{eq}_{\text{ref}}>0, \quad {\vec r}\in\Omega, 
\quad \nabla R_{1,\text{ref}}=0, \quad \nabla R_{2,\text{ref}}^*=0,
\end{equation}
e.g., $R_{1,\text{ref}}\equiv R_{1,0}$, $R_{2,\text{ref}}^*\equiv R_{2,0}^*$, with $R_{2,0}^*$ as in \eqref{R2effdecomp}.

From
\begin{equation}\label{F_ex_formulas}
\begin{aligned}
&(F_{(I)}^{ex}(x^{img}))(t)=
-i \int_{\mathbb{R}^3} e^{-R_2^*(\vec r)(t-\taup)} c^-_j(\vec r)\,M^{eq}(\vec r) \, e^{-2\pi i {\vec k}(t)\cdot{\vec r}}  \,\textsf{d}\vec r\\
&(F_{(J)}^{ex}(x^{img}))(t)=
-i \int_{\mathbb{R}^3} e^{-R_2^*(\vec r)(t-t_J)} c^-_j(\vec r)\,
(1-2e^{-R_1(\vec r)\tauone})M^{eq}(\vec r) \, e^{-2\pi i {\vec k}(t)\cdot{\vec r}}  \,\textsf{d}\vec r\\
&J\in\{II,III\}
%&(F_{(II)}^{ex}(x^{img}))(t)=
%-i \int_{\mathbb{R}^3} e^{-R_2^*(\vec r)(t-\tII)} c^-_j(\vec r)\,
%(1-2e^{-R_1(\vec r)\tauone})M^{eq}(\vec r) \, e^{-2\pi i {\vec k}(t)\cdot{\vec r}}  \,\textsf{d}\vec r\\
%&(F_{(III)}^{ex}(x^{img}))(t)=
%-i \int_{\mathbb{R}^3} e^{-R_2^*(\vec r)(t-\tIII)} c^-_j(\vec r)\,
%(1-2e^{-R_1(\vec r)\tautwo})M^{eq}(\vec r) \, e^{-2\pi i {\vec k}(t)\cdot{\vec r}}  \,\textsf{d}\vec r
\end{aligned}
\end{equation}
with 
$t_J=\tau^{(J)}+2\taup$,
%$\tII=\tauone+2\taup$, $\tIII=\tautwo+2\taup$, 
cf. \eqref{Mperp_ex_90_rect_B}, \eqref{Mperp_ex_180-tau-90_rect_B}, we obtain 
\begin{equation}\label{F_ex_derivativeI}
\begin{aligned}
&({F_{(I)}^{ex}}'(x^{img})\uld{x}^{img})(t)\\
&=-i \int_{\mathbb{R}^3} e^{-R_{2,\text{ref}}^*(t-\taup)} c^-_j(\vec r)
\Bigl(-\uld{R}_2^*(\vec r)(t-\taup)\,M^{eq}_{\text{ref}}(\vec r)+\uld{M}^{eq}(\vec r)\Bigr) \, e^{-2\pi i {\vec k}(t)\cdot{\vec r}}  \,\textsf{d}\vec r\\
&=-i e^{-R_{2,\text{ref}}^*(t-\taup)} \left(\mathcal{F}\left[c^-_j
\Bigl(-\uld{R}_2^*(t-\taup)\,M^{eq}_{\text{ref}}+\uld{M}^{eq}\Bigr)\right]\right) ({\vec k}(t))
\end{aligned}
\end{equation}
\[
\begin{aligned}
&({F_{(J)}^{ex}}'(x^{img})\uld{x}^{img})(t)\\
&=-i \int_{\mathbb{R}^3} e^{-R_{2,\text{ref}}^*(t-t_J)} c^-_j(\vec r)
%\\&\quad
\Bigl((1-2e^{-R_{1,\text{ref}}\tau^{(J)}})\bigl(-\uld{R}_2^*(\vec r)(t-t_J)\,M^{eq}_{\text{ref}}(\vec r)+\uld{M}^{eq}(\vec r)\bigr)\\
&\hspace*{5cm}+2e^{-R_{1,\text{ref}}\tau^{(J)}}\uld{R}_1(\vec r)\tau^{(J)}\,M^{eq}_{\text{ref}}(\vec r) \Bigr) \, e^{-2\pi i {\vec k}(t)\cdot{\vec r}}  \,\textsf{d}\vec r\\
&=-i e^{-R_{2,\text{ref}}^*(t-t_J)} \Bigl(\mathcal{F}\Bigl[c^-_j
\Bigl((1-2e^{-R_{1,\text{ref}}\tau^{(J)}})\bigl(-\uld{R}_2^*(t-t_J)\,M^{eq}_{\text{ref}}+\uld{M}^{eq}\bigr)\\
&\hspace*{5cm}
+2e^{-R_{1,\text{ref}}\tau^{(J)}}\uld{R}_1\tau^{(J)}\,M^{eq}_{\text{ref}} \Bigr)\Bigr]\Bigr) ({\vec k}(t))
\\
&J\in\{II,III\}.
\end{aligned}
\]
\begin{comment}
\begin{equation}\label{F_ex_derivativeII}
\begin{aligned}
&({F_{(II)}^{ex}}'(x^{img})\uld{x}^{img})(t)\\
&=-i \int_{\mathbb{R}^3} e^{-R_{2,\text{ref}}^*(t-\tII)} c^-_j(\vec r)
%\\&\quad
\Bigl((1-2e^{-R_{1,\text{ref}}\tauone})\bigl(-\uld{R}_2^*(\vec r)(t-\tII)\,M^{eq}_{\text{ref}}(\vec r)+\uld{M}^{eq}(\vec r)\bigr)\\
&\hspace*{5cm}+2e^{-R_{1,\text{ref}}\tauone}\uld{R}_1(\vec r)\tauone\,M^{eq}_{\text{ref}}(\vec r) \Bigr) \, e^{-2\pi i {\vec k}(t)\cdot{\vec r}}  \,\textsf{d}\vec r\\
&=-i e^{-R_{2,\text{ref}}^*(t-\tII)} \Bigl(\mathcal{F}\Bigl[c^-_j
\Bigl((1-2e^{-R_{1,\text{ref}}\tauone})\bigl(-\uld{R}_2^*(t-\tII)\,M^{eq}_{\text{ref}}+\uld{M}^{eq}\bigr)\\
&\hspace*{5cm}
+2e^{-R_{1,\text{ref}}\tauone}\uld{R}_1\tauone\,M^{eq}_{\text{ref}} \Bigr)\Bigr]\Bigr) ({\vec k}(t))
\end{aligned}
\end{equation}
and likewise for $F^{ex}_{(III)}$.
\end{comment}
This is indeed the Fr\'{e}chet derivative of $F^{ex}$ in the topology of the spaces $\mathbb{X}^{img}$, $\mathbb{Y}^{3N}$ defined in \eqref{X}, \eqref{Y}. 

Dependency on $t$ under the Fourier transform operator can be eliminated by differentiation with respect to time 
%(as we did in \eqref{reconR2_rect}) 
(see, e.g., the appendix), 
%of \cite{MRI-uniqueness_arxiv}),
or by taking combinations of these values, as we do here in order to avoid ill-posedness due to differentiation. (In fact, differentiation would be impossible in our function space setting $\mathbb{Y}=L^q(0,T)$.)
\begin{equation}\label{R1fromF}
\begin{aligned}
&y_{(J)(I)}(t)\\
:=&ie^{R_{2,\text{ref}}^*(t-t_J)}({F_{(J)}^{ex}}'(x^{img})\uld{x}^{img})(t)\\
&-(1-2e^{-R_{1,\text{ref}}\tau^{(J)}})ie^{R_{2,\text{ref}}^*(t-\taup)}({F_{(I)}^{ex}}'(x^{img})\uld{x}^{img})(t)\\
=&
\left(\mathcal{F}\left[c^-_j
\Bigl((1-2e^{-R_{1,\text{ref}}\tau^{(J)}})\,\uld{R}_2^*(t_J-\taup)\,M^{eq}_{\text{ref}}
+2 e^{-R_{1,\text{ref}}\tau^{(J)}}\tau^{(J)}\,\uld{R}_1\,M^{eq}_{\text{ref}} \Bigr)\right]\right) ({\vec k}(t))
\\
&J\in\{II,III\}
\end{aligned}
\end{equation}
\begin{comment}
\begin{equation}\label{R1fromF}
\begin{aligned}
&y_{(II)(I)}(t)\\
:=&ie^{R_{2,\text{ref}}^*(t-\tII)}({F_{(II)}^{ex}}'(x^{img})\uld{x}^{img})(t)\\
&-(1-2e^{-R_{1,\text{ref}}\tauone})ie^{R_{2,\text{ref}}^*(t-\taup)}({F_{(I)}^{ex}}'(x^{img})\uld{x}^{img})(t)\\
=&
\left(\mathcal{F}\left[c^-_j
\Bigl((1-2e^{-R_{1,\text{ref}}\tauone})\,\uld{R}_2^*(\tII-\taup)\,M^{eq}_{\text{ref}}
+2 e^{-R_{1,\text{ref}}\tauone}\tauone\,\uld{R}_1\,M^{eq}_{\text{ref}} \Bigr)\right]\right) ({\vec k}(t))\\
\end{aligned}
\end{equation}
\[
\begin{aligned}
&y_{(III)(I)}(t)\\
:=&ie^{R_{2,\text{ref}}^*(t-\tIII)}({F_{(III)}^{ex}}'(x^{img})\uld{x}^{img})(t)\\
&-(1-2e^{-R_{1,\text{ref}}\tautwo})ie^{R_{2,\text{ref}}^*(t-\taup)}({F_{(I)}^{ex}}'(x^{img})\uld{x}^{img})(t)\\
=&
\left(\mathcal{F}\left[c^-_j
\Bigl((1-2e^{-R_{1,\text{ref}}\tautwo})\,\uld{R}_2^*(\tIII-\taup)\,M^{eq}_{\text{ref}}
+2 e^{-R_{1,\text{ref}}\tautwo}\tautwo\,\uld{R}_1\,M^{eq}_{\text{ref}} \Bigr)\right]\right) ({\vec k}(t))
\end{aligned}
\]
\end{comment}
This in its turn with $\tII-\taup=\tauone+\taup$, $\tIII-\taup=\tautwo+\taup$ yields
\begin{equation}\label{R2fromF}
\begin{aligned}
&\frac{e^{R_{1,\text{ref}}\tautwo}}{\tautwo} y_{(III)(I)}(t)
-\frac{e^{R_{1,\text{ref}}\tauone}}{\tauone} y_{(II)(I)}(t)\\
&=
%\Bigl(e^{R_{1,\text{ref}}\tautwo}(1-2e^{-R_{1,\text{ref}}\tautwo})(1+\tfrac{\taup}{\tautwo})
%-e^{R_{1,\text{ref}}\tauone}(1-2e^{-R_{1,\text{ref}}\tauone})(1+\tfrac{\taup}{\tauone})\Bigr)\\
%&\qquad\times
(\psi(\tautwo)-\psi(\tauone)
\left(\mathcal{F}\left[c^-_j \,\uld{R}_2^*M^{eq}_{\text{ref}} \right]\right) ({\vec k}(t)).
\end{aligned}
\end{equation}
with $\psi(\tau)=e^{R_{1,\text{ref}}\tau}(1-2e^{-R_{1,\text{ref}}\tau})(1+\tfrac{\taup}{\tau})$.
Therefore, we successively obtain 
$\mathcal{F}\left[c^-_j \,\uld{R}_2^*M^{eq}_{\text{ref}} \right]\circ{\vec k}$, 
$\mathcal{F}\left[c^-_j \,\uld{R}_1M^{eq}_{\text{ref}} \right]\circ{\vec k}$, 
and $\mathcal{F}\left[c^-_j \,\uld{M}^{eq} \right]\circ{\vec k}$ from \eqref{R2fromF}, \eqref{R1fromF}, \eqref{F_ex_derivativeI}. 
This together with \eqref{CI} and the fact that we have defined the topology on $\mathbb{X}^{img}$ by $L^q(\Omega)$ spaces and $M^{eq}_{\text{ref}}$ is bounded away from zero, implies existence of a constant $C_{ex}'$ such that \eqref{stab_ex_lin} holds.
The Implicit Function Theorem thus allows us to conclude that \eqref{stab_ex} holds in a neighborhood of $x^{img}_{\text{ref}}$.
\begin{proposition}
Let $c_j^-$, $j\in\{1,\ldots,N\}$, ${\vec G}\in L^\infty(0,T;\mathbb{R}^3)$ piecewise continuous, $\vec k$ according to \eqref{kt} and $\Xpar\subseteq L^2(\Omega)\cap L^\infty(\Omega)$ be chosen such that \eqref{CI} holds; moreover, let $0<\tauone<\tautwo$.
\\
For any $x^{img}_{\text{ref}}$ satisfying \eqref{MrefR1refR2ref}, 
%for some $\underline{M}^{eq}_{\text{ref}}>0$, 
there exist $\rho>0$, $C_{ex}>0$, $C_{ex}'>0$, such that \eqref{stab_ex}, \eqref{stab_ex_lin} hold with $U=\mathcal{B}_\rho^{\mathbb{X}^{img}}(x^{img}_{\text{ref}})$.
\end{proposition}

\subsubsection{Recovery of 
\allximgDzero{$M^{eq}$, $R_1$ and $R_2^*$ from Bloch-Torrey}
$(M^{eq},R_1)$ from Bloch-Torrey and $(M^{eq},R_1,R_2^*)$ from Bloch 
parameter-to state map}\label{subsec:MRI_BT}
We are now in the position to conclude uniqueness and stability for the general inverse problem \eqref{Fxy} with $F$ defined as in \eqref{F_gen} by a perturbation argument.
To this end we choose the data space according to \eqref{Y}, so that with $\mathbb{V}_\perp$ as in \eqref{V}, boundedness \eqref{CC} of the observation operator holds. 
\begin{theorem}[\revisionown{uniqueness of $(M^{eq},R_1,R_2,\delta B^0)$ in Bloch ODE}]\label{thm:uniqueness_perturb_D0}
Let $\Omega$ be a bounded Lipschitz domain, 
assume that $c_j^-\in L^\infty(\mathbb{R}^3)$ satisfies \eqref{supp_cj-}, $j\in\{1,\ldots,N\}$, 
let ${\vec G}\in L^\infty(0,T;\mathbb{R}^3)$ be piecewise continuous with $\vec k$ according to \eqref{kt}, 
and let $\Xpar\subseteq L^2(\Omega)\cap L^\infty(\Omega)$ be chosen such that \eqref{CI} holds; 
moreover, let $0<\tauone<\tautwo$.
\\
For any $x^{img}_{\text{ref}}$ satisfying \eqref{MrefR1refR2ref} for some $R_{1,\text{ref}},\,R_{2,\text{ref}}>0$, $\underline{M}^{eq}_{\text{ref}}>0$, there exist $\rho_{\taup}>0$, $\rho_{\pp}>0$, 
%$\rho_v>0$, $\rho_D>0$, %2025-06-12
$\rho_{c^+}>0$, $\rho>0$, $C>0$, such that for 
\[
\begin{aligned}
&0<\taup\,<\rho_{\taup}\\
&\max\{\|{\pp}_{(I)}-{\pp}_{90^\circ}\|_{L^2(0,T)}, \, 
\max_{J\in\{II,\,III\}}\|{\pp}_{(J)}-{\pp}_{180^\circ-\tau^{(J)}-90^\circ}\|_{L^2(0,T)}\} \, < \rho_{\pp},\quad \\
%&\vec v\in L^1(0,T;W^{1,\infty}(\Omega)) \text{ satisfying \eqref{divv}}, \quad 
%\|\vec v\|_{L^\infty_{D^{-1/2}}(\Omega;\mathbb{R}^3)}<\rho_v\\ 	%2025-06-12
%&D\in L^\infty(\Omega;\mathbb{R}^{3\times 3})\text{ nonnegative definite , \ 
%\|D\|_{L^\infty_{D^{-1/2}}(\Omega;\mathbb{R}^{3\times 3})}<\rho_D}\\	%2025-06-12
&\vec v=0, \qquad D=0\\		%2025-06-12
&c^+\in L^\infty(\Omega;\mathbb{C}), \ 
\|c^+-c^+_0\|_{L^\infty(\Omega;\mathbb{C}\revision{)}}<\rho_{c^+} \text{ with $c^+_0$ as in \eqref{90_rect}, \eqref{180-tau-90_rect},}
\end{aligned}
\]
\Margin{R2 8)}
the local stability and uniqueness estimate
\[
\forall x^{img},\tilde{x}^{img}\, \in \Xpar^3\cap \mathcal{B}_\rho^{\mathbb{X}^{img}}(x^{img}_{\text{ref}}):
\|x^{img}-\tilde{x}^{img}\|_{\mathbb{X}^{img}}\leq C
\|F(x^{img})-F(\tilde{x}^{img})\|_{\mathbb{Y}^{3N}}
\] 
holds.
\end{theorem}
\begin{proof}
The perturbation estimates from Propositions~\ref{prop:perturb_expl-sol}, \ref{prop:perturb_gen_deriv} together with boundedness of the observation operator \eqref{CC} yield
\[
\|F'(x^{img}_{\text{ref}})-F_{ex}'(x^{img})\|_{L(\mathbb{X}^{img},\mathbb{Y}^{3N})}
\leq C_{\mathcal{C}}\bigl(C_{\mathcal{S}_{ex}}\,\rho_{\taup} +C_{\mathcal{S}}(\rho_{\pp}
%+\rho_v+\rho_D 	%2025-06-12
+\rho_{c^+})\bigr)=:c<\frac{1}{C_{ex}'}
\]
%with $c:=C_{\mathcal{C}}\bigl(C_{\mathcal{S}_{ex}}\,\rho_{\taup} +C_{\mathcal{S}}(\rho_{\pp}+\rho_D+\rho_{c^+})\bigr)<\frac{1}{C_{ex}}$ 
for sufficiently small $\rho_{\taup}>0$, $\rho_{\pp}>0$, 
%$\rho_v$, $\rho_D>0$, 	%2025-06-12
$\rho_{c^+}>0$.
Combining this with \eqref{stab_ex_lin} and applying the Implicit Function Theorem implies the assertion with any $C>\frac{C_{ex}'}{1-cC_{ex}'}$ and $\rho>0$ small enough. 
\end{proof}

Since we had to assume that $\vec v=0$, $D=0$, Theorem~\ref{thm:uniqueness_perturb_D0} is is still a result for the Bloch equation, but with excitations that may deviate from the elementary pulse sequences used for the explicit reconstruction formulas in %Sections~\ref{sec:explsol} and ~\ref{sec:appendix_reconR1R2}).
Section~\ref{sec:explsol}.
The reason for this is that a perturbation argument would require smallness of $\vec{v}$ and $D$ in the $D$-dependent norm $L^\infty_{D^{-1/2}}(\Omega)$, whereas it is obvious that 
$\|D\|_{L^\infty_{D^{-1/2}}(\Omega;\mathbb{R}^{3\times 3})}=1$.

This can be amended by making use of the k-space explicit solution formulas, that allow for nonzero (but constant) $D_0$, along with Remark~\ref{rem:transform_kspace} as well as linearization around a reference point $x^{img}_{\text{ref}}$ satisfying \eqref{MrefR1refR2ref}.
Indeed, for $x^{mod}=(D_0,0,c^+_0,G,\pp)$, the derivative $\uldMpp=\frac{\partial\mathcal{S}}{\partial x^{img}}(x^{img}_{\text{ref}})\uld{x}^{img}$ is defined by the solution to \eqref{eqn:bloch-torrey-complex_rot_FT_f} with   
\[
\begin{aligned}
&\fpp(t,\vec r)=-\uld{R_2^*}(\vec r){\Mpp}_{\text{ref}}(t,\vec r), \\
&f_z(t,\vec r)=-\uld{R_1}(\vec r)({M_z}_{\text{ref}}(t,\vec r)-M^{eq}_{\text{ref}}(\vec r))+
R_{1,0}(\vec r)\uld{M^{eq}}(\vec r),
\end{aligned}
\]
%\eqref{BT0_uncoupled}, \eqref{hatMperpMz_BT0} 
where ${\Mpp}_{\text{ref}}=\mathcal{S}(x^{img}_{\text{ref}};x^{mod}_{\text{ref}})$ and ${M_z}_{\text{ref}}$ is the corresponding longitudinal part.

This in the constant coil sensitivity case \eqref{eqn:mri_measurement_Fourier_cj0} yields (see the appendix 
%of \cite{MRI-uniqueness_arxiv} 
for details)
\Margin{R1 C.}
\Margin{R2 (3)}
\begin{equation}\label{FIprimeD0}
\begin{aligned}
&\frac{1}{c_{j,0}^-}\Bigl(F_{(I)j}'(x^{img}_{\text{ref}})\uld{x}^{img}\Bigr)(t)
%= \uld{\Mphp}(t,0) \text{ with }\pp={\pp}_{90^\circ}
\\
&\dot{=}
%-i\,e^{-\int_{\taup}^t(\Dterm{({\vec k}(t)-\vec k(\sigma))}+R_{2,0}^*)\, d\sigma}\,
%\widehat{\uld{M}^{eq}}({\vec k}(t)-{\vec k}(\taup))
%\\&+i\,\int_{\taup}^t e^{-\int_a^t(\Dterm{({\vec k}(t)-\vec k(\sigma))}+R_{2,0}^*)\, d\sigma}
%\int_{\mathbb{R}^3}\widehat{\uld{R_2^*}}({\vec k}(t)-{\vec k}(a)-\zeta)\\
%&\qquad\qquad e^{-\int_{\taup}^a(\Dterm{(\vec\zeta+{\vec k}(a)-\vec k(\sigma))} +R_{2,0}^*)\, d\sigma}\,\widehat{M^{eq}}(\vec\zeta+{\vec k}(a)-\vec k(\taup))
%\,d\zeta\, da\\
%&\text{ with $\zeta'=\zeta+{\vec k}(a)$ and $\vec k(\taup)=0$}\\
%&=
-i\,e^{-\int_{\taup}^t(\Dterm{({\vec k}(t)-\vec k(\sigma))}+R_{2,0}^*)\, d\sigma}\,
\widehat{\uld{M}^{eq}}({\vec k}(t))
\\&+i\,\int_{\taup}^t e^{-\int_a^t(\Dterm{({\vec k}(t)-\vec k(\sigma))}+R_{2,0}^*)\, d\sigma}
\int_{\mathbb{R}^3}\widehat{\uld{R_2^*}}({\vec k}(t)-\zeta')\\
&\qquad\qquad e^{-\int_{\taup}^a(\Dterm{(\vec\zeta'-\vec k(\sigma))} +R_{2,0}^*)\, d\sigma}\,\widehat{M^{eq}}(\vec\zeta')
\,d\zeta'\, da
\end{aligned}
\end{equation}
\begin{equation}\label{FIIprimeD0}
\begin{aligned}
&\frac{1}{c_{j,0}^-}\Bigl(F_{(II)j}'(x^{img}_{\text{ref}})\uld{x}^{img}\Bigr)(t)
%= \uld{\Mphp}(t,0) \text{ with }\pp={\pp}_{180^\circ-\tau-90^\circ}
\\
&\dot{=}
%-i\,e^{-\int_{\tau+2\taup}^t(\Dterm{({\vec k}(t)-\vec k(\sigma))}+R_{2,0}^*)\, d\sigma}\\
%&\qquad\qquad\Bigl(\widehat{\mathbb{L}_\tau \uld{M}^{eq}}
%-\mathcal{F}\left[\int_\taup^{\tau+\taup}e^{-\mathcal{A}_1(\tau+\taup-s)}\uld{R}_1(\mathbb{L}_{s-\taup}-1)M^{eq}\right]\Bigr)({\vec k}(t)-{\vec k}(\tau+2\taup))
%\\&+i\,\int_{\tau+2\taup}^t e^{-\int_a^t(\Dterm{({\vec k}(t)-\vec k(\sigma))}+R_{2,0}^*)\, d\sigma}
%\int_{\mathbb{R}^3}\widehat{\uld{R_2^*}}({\vec k}(t)-{\vec k}(a)-\zeta)\\
%&\qquad\qquad e^{-\int_{\tau+2\taup}^a(\Dterm{(\vec\zeta+{\vec k}(a)-\vec k(\sigma))} +R_{2,0}^*)\, d\sigma}\,\widehat{\mathbb{L}_\tau M^{eq}}(\vec\zeta+{\vec k}(a)-\vec k(\tau+2\taup))
%\,d\zeta\, da\\
%&\text{ with $\zeta'=\zeta+{\vec k}(a)$ and $\vec k(\tau+2\taup)=0$}\\
%&=
-i\,e^{-\int_{\tau+2\taup}^t(\Dterm{({\vec k}(t)-\vec k(\sigma))}+R_{2,0}^*)\, d\sigma}\\
&\qquad\qquad\Bigl(\widehat{\mathbb{L}_\tau \uld{M}^{eq}}
-\mathcal{F}\left[\int_\taup^{\tau+\taup}e^{-\mathcal{A}_1(\tau+\taup-s)}\uld{R}_1(\mathbb{L}_{s-\taup}-1)M^{eq}\right]\Bigr)({\vec k}(t))
\\&+i\,\int_{\tau+2\taup}^t e^{-\int_a^t(\Dterm{({\vec k}(t)-\vec k(\sigma))}+R_{2,0}^*)\, d\sigma}
\int_{\mathbb{R}^3}\widehat{\uld{R_2^*}}({\vec k}(t)-\zeta')\\
&\qquad\qquad e^{-\int_{\tau+2\taup}^a(\Dterm{(\vec\zeta'-\vec k(\sigma))} +R_{2,0}^*)\, d\sigma}\,\widehat{\mathbb{L}_\tau M^{eq}}(\vec\zeta')
\,d\zeta\, da\\
\end{aligned}
\end{equation}
where 
\[
\begin{aligned}
&\mathbb{L}_\tau=R_{1,0}\mathcal{A}_1^{-1}(1-e^{-\mathcal{A}_1\tau})-e^{-\mathcal{A}_1\tau}\\
&\mathcal{A}_1=-\nabla\cdot(D_0\nabla\, \cdot)+R_{1,0},
\end{aligned}
\]
and analogously for $F_{(III)}$.

Here the $\dot{=}$ approximations can be quantified to be $O(\taup)$, analogously to Proposition~\ref{prop:perturb_expl-sol_lin}, cf. Remark~\ref{rem:perturb_kspace}.

\allximgDzero{
With the same elimination procedure as in Section~\ref{subsec:MRIex}, this leads to a uniqueness and stability estimate 
\begin{equation}\label{stab_0_lin}
\forall  \uld{x}\in \Xpar^3:
\|\uld{x}^{img}\|_{\mathbb{X}^{img}}\leq C_0'
\|F'(x^{img}_{\text{ref}})\uld{x}^{img}\|_{\mathbb{Y}^{3N}}.
\end{equation}
for $D\equiv D_0$, $\vec{v}\equiv 0$,$c^+\equiv c^+_0$, $c_j^-\equiv {c_j}^-_0$.
}

An elimination procedure as in Section~\ref{subsec:MRIex} does not appear feasible here any more. However, for $\uld{R_2^*}=0$, one can obviously 
recover $\uld{M^{eq}}$ from $F_{(I)}'(x^{img}_{\text{ref}})\uld{x}^{img}$ and then 
recover $\uld{R_1}$ from $F_{(II)}'(x^{img}_{\text{ref}})\uld{x}^{img}$.
Thus reducing $x^{img}$ and $F$ to 
\begin{equation}\label{xred_Fred}
\xred^{img}=(M^{eq},R_1),
\qquad 
\Fred(\xred^{img})=\left(\begin{array}{c}
\mathcal{C}\mathcal{S}(\xred^{img},R_2^*;\check{x}^{mod},{\pp}_{(I)})\\
\mathcal{C}\mathcal{S}(\xred^{img},R_2^*;\check{x}^{mod},{\pp}_{(II)})\\
\end{array}\right),
\end{equation}
we obtain a uniqueness and stability estimate 
\begin{equation}\label{stab_0_lin}
\forall  \uld{\xred}\in \Xpar^2:
\|\uld{\xred}^{img}\|_{\mathbb{X}^{img}}\leq C_0'
\|\Fred'(\xred^{img}_{\text{ref}})\uld{\xred}^{img}\|_{\mathbb{Y}^{2N}}.
\end{equation}
for $D\equiv D_0$, $\vec{v}\equiv 0$,$c^+\equiv c^+_0$, $c_j^-\equiv {c_j}^-_0$.

This analogously to Theorem~\ref{thm:uniqueness_perturb_D0} leads to

\begin{theorem}[\revisionown{uniqueness of $(M^{eq},R_1)$ in Bloch-Torrey PDE}]\label{thm:uniqueness_perturb_red}
Let $\Omega$ be a bounded Lipschitz domain, 
assume that $c_j^-\in L^\infty(\mathbb{R}^3)$ satisfies \eqref{supp_cj-}, $j\in\{1,\ldots,N\}$, 
let ${\vec G}\in L^\infty(0,T;\mathbb{R}^3)$ be piecewise continuous with $\vec k$ according to \eqref{kt}, 
and let $\Xpar\subseteq L^2(\Omega)\cap L^\infty(\Omega)$ be chosen such that \eqref{CI} holds; 
moreover, let $0<\tauone$ and $D_0\in\mathbb{R}^{3\times3}$ positive semidefinite.
\\
For any $x^{img}_{\text{ref}}$ satisfying \eqref{MrefR1refR2ref} for some $R_{1,\text{ref}},\,R_{2,\text{ref}}>0$, $\underline{M}^{eq}_{\text{ref}}>0$, there exist $\rho_{\taup}>0$, $\rho_{\pp}>0$, 
$\rho_2>0$, 
$\rho_v>0$, $\rho_D>0$, 
$\rho_{c^+}>0$, $\rho>0$, $C>0$, such that for 
\[
\begin{aligned}
&0<\taup\,<\rho_{\taup}\\
&\max\{\|{\pp}_{(I)}-{\pp}_{90^\circ}\|_{L^2(0,T)}, \, 
\|{\pp}_{(II)}-{\pp}_{180^\circ-\tau^{(II)}-90^\circ}\|_{L^2(0,T)}\} \, < \rho_{\pp},\quad \\
&R_2^*\in L^\infty(\Omega;\mathbb{C}) , \ 
\|R_2^*-R_{2,0}^*\|_{L^\infty_(\Omega;\mathbb{C})}<\rho_2\\	
&\vec v\in L^1(0,T;W^{1,\infty}(\Omega)) \text{ satisfying \eqref{divv}}, \quad 
\|\vec v\|_{L^\infty_{D_0^{-1/2}}(\Omega;\mathbb{R}^3)}<\rho_v\\ 	
&D\in L^\infty(\Omega;\mathbb{R}^{3\times 3})\text{ nonnegative definite} , \ 
\|D-D_0\|_{L^\infty_{D_0^{-1/2}}(\Omega;\mathbb{R}^{3\times 3})}<\rho_D\\	
&c^+\in L^\infty(\Omega;\mathbb{C}), \ 
\|c^+-c^+_0\|_{L^\infty(\Omega;\mathbb{C}}<\rho_{c^+} \text{ with $c^+_0$ as in \eqref{90_rect}, \eqref{180-tau-90_rect},}
\end{aligned}
\]
the local stability and uniqueness estimate
\[
\forall \xred^{img},\tilde{\xred}^{img}\, \in \Xpar^2\cap \mathcal{B}_\rho^{\check{\mathbb{X}}^{img}}(\xred^{img}_{\text{ref}}):
\|\xred^{img}-\tilde{\xred}^{img}\|_{\check{\mathbb{X}}^{img}}\leq C
\|\Fred(\xred^{img})-\Fred(\tilde{\xred}^{img})\|_{\mathbb{Y}^{2N}}
\] 
holds.
\end{theorem}

\medskip

The above results in fact state well-posedness of the inverse problem and as a consequence, relying on Lipschitz continuous differentiability of the forward operator due to Proposition~\ref{prop:perturb_gen_deriv}, we can use the Newton-Kantorowich Theorem to conclude local and quadratic convergence of Newton's method
\begin{equation}\label{Newton}
x^{img}_{n+1}=x^{img}_n+F'(x^{img}_n)^{-1}(\underline{y}-F(x^{img}_n))
\end{equation} 
or local linear convergence of a frozen Newton method
\begin{equation}\label{frozenNewton}
x^{img}_{n+1}=x^{img}_n+F'(x^{img}_0)^{-1}(\underline{y}-F(x^{img}_n)).
\end{equation} 
In this convergence result we denote by $x^{img}_*$ an exact solution of \eqref{Fxy}
\begin{equation}\label{xstar}
F(x^{img}_*)=\underline{y}.
\end{equation}

\begin{corollary}\label{cor:Newton}
Under the conditions of Theorem~\ref{thm:uniqueness_perturb_D0}, there exists $\rho_0>0$ such that with $\|x^{img}_{\text{ref}}-x^{img}_*\|_{\mathbb{X}^{img}}<\rho_0$, Newton's method \eqref{Newton} (or the frozen Newton method \eqref{frozenNewton}) starting at $x^{img}_0$ with $\|x^{img}_0-x^{img}_*\|_{\mathbb{X}^{img}}<\rho_0$ converges quadratically (or linearly) to $x^{img}_*$.\\
An analogous statement holds under the conditions of Theorem~\ref{thm:uniqueness_perturb_red} for $x$, $F$ replaced by $\xred$, $\Fred$ as in \eqref{xred_Fred}.
\end{corollary}
\begin{remark}\label{rem:noise}
When applied with noisy data $\underline{y}^\delta$ in place of $\underline{y}$, under the conditions of Corollary~\ref{cor:Newton}, we obtain convergence of \eqref{Newton} or \eqref{frozenNewton} to  $x^{img}_\delta=F^{-1}(\underline{y}^\delta)\in\mathbb{X}^{img}$ which satisfies $\|x^{img}_\delta-x^{img}_*\|_{\mathbb{X}^{img}}\leq C\|\underline{y}^\delta-\underline{y}\|_{\mathbb{Y}^{3N}}$ with some constant $C$ depending only on $x^{img}_*$.
Again, the same holds true for $\xred$, $\Fred$ in place of $x$, $F$.
\end{remark}

\subsection{Linearized uniqueness by diffusion}\label{sec:uniqueness_diffusion}

As opposed to the uniqueness proof from the previous section, that relied on sampling of the k-space by varying $\vec G$, along with smallness of diffusion in a perturbation argument, we here make explicit use of diffusion and the resulting infinite speed of propagation to prove uniqueness from observations on a time interval with constant $\vec G\equiv \vec G_0$.
For this purpose, additionally to \eqref{definiteness}, we assume $D$ to be uniformly positive definite
\begin{equation}\label{Dposdef}
\lambda_{\min}(D(\vec r))\geq\underline{\lambda}>0, \quad \vec r\in\Omega.
\end{equation}
To avoid technical difficulties resulting from advection, we here impose $\vec v(t,\vec r)\equiv0$.

While uniqueness will be derived with $\vec G$ taking a single value here, so that only one line in k-space is traversed according to \eqref{kt}, this result comes with no stability at all and so it is advisable to combine this with sampling in k-space by using a piecewise constant field gradient. Such a combination can be carried out within a so-called Kaczmarz approach \cite{BKL:2010,NewtonKaczmarz,HLS:07,HKLS:07,aao_time,KoSc:2002,Ngu:2019}, that is, by cyclically iterating over the subproblems resulting from subdividing the time line into intervals of constant $\vec G$. Since the convergence analysis of Kaczmarz methods requires linearized 
\revision{uniqueness} 
\Margin{R2 9)}
for each subproblem \cite{NewtonKaczmarz}, the uniqueness result obtained here will be very useful for such a convergence proof.
We refer to Remark~\ref{Kaczmarz90deg} for some further thoughts on this, but postpone a full analysis of this Kaczmarz approach to future work.

The models considered in this section are obtained by means of the formal limit $\taup\searrow0$. This is feasible in spite of Remark~\ref{rem:pdelta}, due to the fact that we are not using PDE estimates in this subsection.

\subsubsection{Recovery of $M^{eq}$}
First of all, in order to illustrate the principle of proof, uniqueness is shown for the linear inverse problem of reconstructing $M^{eq}$ alone from observations obtained after a rectangular $90^\circ$ pulse, neglecting diffusion during the pulse, that is, formally letting $\taup\to0$. 
We consider a time interval $[t_0,t_1]$ after the pulse, on which the PDEs for $\Mpp$ and $M_z$ are therefore decoupled, and on which $\vec G$ is constant. 

This inverse problem hence amounts to reconstructing the initial data $\Mpp^{t_0}(\vec r)=\Mpp(t_0,\vec r)$ in the parabolic initial boundary value problem
\begin{equation}\label{ibvp}
\begin{aligned}
&\frac{d}{dt}\Mpp
%+{\vec v}\cdot\nabla M_\perp 
-\nabla\cdot\bigl(D \nabla \Mpp\bigr) +\bigl(R_2^*+i\gamma\vec r \cdot \vec G_0\bigr)\Mpp=0\text{ in }(t_0,t_1)\times\Omega\\
&\Mpp(t_0)=0\text{ in }(t_0,t_1)\times\partial\Omega\\
&\Mpp(t_0)=\Mpp^{t_0}.
\end{aligned}
\end{equation}
Problems of this type have been studied extensively in the literature and uniqueness from final time or boundary time trace data has been established, see, e.g., \cite{BukhgeimKlibanov:1981,ImanuvilovYamamoto:1998,Isakov:2006}.
However, here we only have finitely many nonlocal observations
\begin{equation}\label{obs_D}
{\yp}_j(t)=\int_{\mathbb{R}^3} c^-_j(\vec r) \Mpp(t, \vec r) \,\textsf{d}\vec r, \quad j\in\{1,\ldots N\}, \quad t\in[t_0,t_1]
\end{equation}
which provide less information, as a simple dimension count reveals.
This fact will influence the type of uniqueness we can obtain here.

Returning for the moment to the 
$\Mpp=\left(\begin{array}{c}{\Mpp}_x\\ {\Mpp}_y\end{array}\right)$ instead of 
${\Mpp}_x+i{\Mpp}_y$ 
notation, we rewrite \eqref{ibvp} as 
\begin{equation}\label{abstractODE}
\begin{aligned}
&\frac{d}{dt}\Mpp+\mathcal{A}\Mpp =0\text{ in }(t_0,t_1)\\
&\Mpp(t_0)=\Mpp^{t_0}
\end{aligned}
\end{equation}
with the non symmetric operator\footnote{More precisely, it is the sum of a symmetric elliptic operator and a lower order skew symmetric operator}
\begin{equation}\label{mathcalA}
\mathcal{A}=
\left(\begin{array}{cc}
-\nabla\cdot\bigl(D \nabla \,.\,\bigr)+\Re(R_2^*)&\Im(R_2^*)+\gamma\vec r \cdot \vec G_0\\ 
-\Im(R_2^*)-\gamma\vec r \cdot \vec G_0&-\nabla\cdot\bigl(D \nabla \,.\,\bigr)+\Re(R_2^*)
\end{array}\right)
\end{equation}
%%{\vec v}\cdot\nabla \,.\, 
%-\nabla\cdot\bigl(D \nabla \,.\,\bigr) +\bigl(R_2^*+i\gamma\vec r \cdot \vec G_0\bigr)\,.\,
equipped with homogeneous Dirichlet boundary conditions.

Proving uniqueness of $\Mpp^{t_0}$ from the data \eqref{obs_D}, due to linearity amounts to showing that ${\yp}_j=0$ for $j\in\{1,\ldots N\}$ in \eqref{obs_D} together with \eqref{ibvp} implies $\Mpp^{t_0}=0$.

To this end, we apply some spectral arguments taken from \cite{JiangLiPauronYamamoto2023}, where more generally, a fractional time derivative is considered, but uniqueness is shown from time trace boundary observations, which allows the authors of \cite{JiangLiPauronYamamoto2023} to employ a unique continuation argument. When it comes to this point, our analysis will therefore have to deviate from the one in \cite{JiangLiPauronYamamoto2023}.

Under condition \eqref{definiteness}, the operator $\mathcal{A}$ is bounded and coercive on $H_D^1(\Omega)$ (cf. Section~\ref{sec:wellposed}), which is isomorphic to $H^1(\Omega)$ under our uniform positivity assumption \eqref{Dposdef} on $D$. 
Its inverse $\mathcal{A}^{-1}:L^2(\Omega)\to L^2(\Omega)$ is therefore compact\footnote{Here we again need boundedness of $\Omega$} and thus 
its spectrum consists of countably many eigenvalues $\mu_m$ (which may take complex values due to nonsymmetry of $\mathcal{A}^{-1}$) tending to zero, but not including zero, since  $\mathcal{A}^{-1}$ is bijective. By Riesz theory, the eigenspaces have finite dimension $d_m$ and they coincide with the eigenspaces of $\mathcal{A}$ corresponding to the eigenvalues $\lambda_m=\frac{1}{\mu_m}$.
%has a singular value decomposition \footnote{resulting from application of the Spectral Theorem to the selfadjoint operator $(\mathcal{A}^{-1})^*\mathcal{A}^{-1}$} $(\sigma_m,(\varphi_j^\ell)_{\ell\in \{1,\ldots L^j\}},(\psi_j^\ell)_{\ell\in \{1,\ldots L^j\}})$ where $L^j$ is the (finite) dimension of the eigenspace corresponding to the singular value $\sigma_j>0$ such that
%\[\begin{aligned}
%&(\mathcal{A}^{-1})^*\mathcal{A}^{-1}\varphi_j^\ell = \sigma_j^2 \varphi_j^\ell, \quad
%&&\mathcal{A}^{-1}(\mathcal{A}^{-1})^*\psi_j^\ell = \sigma_j^2 \psi_j^\ell\\  
%&\mathcal{A}^{-1}\varphi_j^\ell = \sigma_j \psi_j^\ell, \quad
%&&(\mathcal{A}^{-1})^*\psi_j^\ell = \sigma_j \varphi_j^\ell\\  
%&\mathcal{A}^{-1} v =\sum_{j\in\mathbb{N}} \sigma_j \sum_{\ell=1}^{L^j} \langle \varphi_j^\ell, v\rangle_{L^2(\Omega)} \psi_j^\ell %&& v\in L^2(\Omega)\\
%&\mathcal{A} w =\sum_{j\in\mathbb{N}} \sigma_j^{-1} \sum_{\ell=1}^{L^j} \langle \psi_j^\ell, w\rangle_{L^2(\Omega)} \varphi_j^\ell %&& w\in \mathcal{D}(\mathcal{A})\subseteq L^2(\Omega)\\
%\end{aligned}\] 
%This allows us to express the solution to \eqref{abstractODE} as
%\[	\Mpp(t)=\sum_{j\in\mathbb{N}} e^{-\sigma_j t}\langle \psi_j^\ell,\Mpp^{t_0}\rangle_{L^2(\Omega)} \varphi_j^\ell\]
%\textcolor{red}{NO!! this only solves \eqref{abstractODE} if $\varphi_j^\ell=\psi_j^\ell$ \\ 
%$\leadsto$ adopt Masahiro's approach for nonsymmetric operators!}

Analyticity with respect to time allows us to extend the identities \eqref{abstractODE} and \eqref{obs_D} (with ${\yp}_j(t)=0$) to the interval $(t_0,\infty)$ and to apply the shifted Laplace transform defined by 
\[
(\mathcal{L}_{t_0}v)(s)=\int_{t_0}^\infty v(t) e^{-st}\,dt, 
\]
which implies
\[
(s+\mathcal{A})(\mathcal{L}_{t_0}\Mpp)(s)=e^{-st_0}\Mpp^{t_0},\quad s\in\mathbb{C}\setminus\{-\lambda_m\,:\,m\in\mathbb{N}\}.
\]
Plugging this into the observations under the assumption \eqref{supp_cj-} and multiplying with $e^{st_0}$ we obtain
%\[0=\int_{\mathbb{R}^3} c^-_j(\vec r) ((s+\mathcal{A})^{-1}e^{-st_0}\Mpp^{t_0})(\vec r) \,\textsf{d}\vec r, \quad j\in\{1,\ldots N\}, \quad s\in\mathbb{C}\setminus\{-\lambda_m\,:\,m\in\mathbb{N}\},\]
%That is, if ${\yp}_j$ vanishes for all $j\in\{1,\ldots N\}$, 
%hence 
\begin{equation}\label{splusAinvobs}
\int_{\mathbb{R}^3} c^-_j(\vec r) ((s+\mathcal{A})^{-1}\Mpp^{t_0})(\vec r) \,\textsf{d}\vec r=0, \quad j\in\{1,\ldots N\}, \quad s\in\mathbb{C}\setminus\{-\lambda_m\,:\,m\in\mathbb{N}\}.
\end{equation}
Now, for any $\ell\in\mathbb{N}$ and a sufficiently small circle $\Gamma_\ell$ encircling $-\lambda_\ell$, the eigenprojector $P_\ell:L^2(\Omega)\to \mathcal{N}(-\lambda_\ell\text{id}+\mathcal{A})$ can be represented as 
\[
P_\ell v=\int_{\Gamma_\ell}(s+\mathcal{A})^{-1}v\,ds.
\]
Thus, integrating over $\Gamma_\ell$ with respect to $s$ in \eqref{splusAinvobs}, we obtain that 
\[
\langle \overline{c_j^-} , P_\ell \Mpp^{t_0}\rangle_{L^2}=\int_{\mathbb{R}^3} c^-_j(\vec r)\, (P_\ell \Mpp^{t_0})(\vec r) \,\textsf{d}\vec r=0, \quad j\in\{1,\ldots N\}, \quad \ell\in\mathbb{N}.
\]
Thus, we can disentangle the components of $\Mpp^{t_0}$ belonging to different eigenspaces.
 The observations available here do not allow to use unique continuation for further disentangling the components inside each eigenspace, though. By restricting the ansatz space (as we also did for the MRI isomorphism in Section~\ref{subsec:MRIiso}), we can enforce uniqueness, though.
The assumption required for this purpose is
%\[P_\ell\Xpar \cap \text{span}\{c_1^-,\ldots,c_N^-\}^\perp_{L^2} \, = \{0\} \quad \ell\in\mathbb{N}\]
%or equivalently (by taking orthogonal complements and using $(U\cap V)^\perp=U^\perp+V^\perp$)
%\[(P_\ell\Xpar)^\perp +\text{span}\{c_1^-,\ldots,c_N^-\} \, = L^2(\Omega) \quad \ell\in\mathbb{N}\]
%or equivalently 
(using $\langle \overline{c_j^-} , P_\ell \Mpp^{t_0}\rangle_{L^2}= \langle P_\ell \overline{c_j^-} , \Mpp^{t_0}\rangle_{L^2}$)
\[
\bigcap_{\ell\in\mathbb{N}} \bigcap_{j\in \{1,\ldots,N\}} (P_\ell \overline{c_j^-})^\perp \cap \Xpar =\{0\}
\]
or equivalently (by taking orthogonal complements and using $(U\cap V)^\perp=U^\perp+V^\perp$)
\[
\overline{\text{span}\{P_\ell \overline{c_j^-}\, : \, \ell\in\mathbb{N}, \ j\in \{1,\ldots,N\}\}}+\Xpar^\perp \, = L^2(\Omega) 
\]
A sufficient condition for this, relying on $\overline{\Xpar}+\Xpar^\perp \, = L^2(\Omega)$ is 
\begin{equation}\label{assXpar_uniquenessD}
\Xpar\subseteq\text{span}\{P_\ell \overline{c_j^-}\, : \, \ell\in\mathbb{N}, \ j\in \{1,\ldots,N\}\}. 
\end{equation}

\begin{theorem}
%[\revisionown{uniqueness of $M^{eq}$ in Bloch-Torrey PDE}]
\label{thm:uniqueness_diff_Meq}
Let $\Omega$ be a bounded Lipschitz domain, 
let $c_j^-\in L^\infty(\mathbb{R}^3)$ with \eqref{supp_cj-}, $j\in\{1,\ldots,N\}$, 
let $D\in L^\infty(\Omega;\mathbb{R}^{3\times3})$ with \eqref{Dposdef},
let $R_2^*\in L^\infty(\Omega;\mathbb{C})$ with $\Re(R_2^*)\geq0$, 
and assume that the ansatz space $\Xpar$ satisfies \eqref{assXpar_uniquenessD}. 

Then the observations \eqref{obs_D} uniquely determine $\Mpp^{t_0}\in\Xpar$ in \eqref{ibvp}. 
\end{theorem}

Theorem~\ref{thm:uniqueness_diff_Meq} states that the linear operator 
\[
F:\Xpar\to L^2(0,T), \ \Mpp^{t_0}\mapsto \int_{\mathbb{R}^3} c^-_j(\vec r) \Mpp(\cdot, \vec r) \,\textsf{d}\vec r,\ \text{ where $\Mpp$ solves \eqref{ibvp}}
\]
is injective.
Note that this holds without the need of explicit knowledge of $R_2^*$; it just needs to have nonnegative real part.
In this sense, it is a result of reconstruction ``in an unknown medium'' in the spirit of, e.g., \cite{JinKian:2023}.

\begin{remark}\label{Kaczmarz90deg}
To embed this in a Kaczmarz context as indicated above, one can apply the above uniqueness argument on a sequence of intervals $[t_\ell,t_{\ell+1}]$ in place of $[t_0,t_1]$.
Consider a setting of $\vec G$ being (approximately) piecewise constant, that is, $\vec G(t)=\vec G_\ell$ on $[t_\ell,t_\ell+\tau_\ell]$, with a $90^\circ$ pulse in between $t_{\ell+1}= t_\ell+\tau_\ell+\taup$, $\tau_\ell>>\taup$, which due to (partial) recovery of $M^{eq}$ in longitudinal direction during $\tau_\ell$ implies that the $90^\circ$ pulse in $[t_{\ell+1}-\taup,t_{\ell+1}]$ turns a nonvanishing  fraction of $M^{eq}$ into transversal direction, which is thus visible to $M_\perp$ on $[t_{\ell+1},t_{\ell+1}+\tau_{\ell+1}]$.
Based on this fact, the uniqueness result proven here will allow to conclude uniqueness of $M^{eq}$ also from observations on the next interval $[t_{\ell+1},t_{\ell+1}+\tau_{\ell+1}]$.
A full substantiation of this argument and its use for proving convergence of a Kaczmarz method that is based on the subdivision of the timeline into the intervals $[t_\ell,t_{\ell+1})$, $\ell=0,1,2,\ldots$ will be subject of future research.
\end{remark}

\subsubsection{Recovery of $(M^{eq},R_1,R_2^*)$; linearized uniqueness}

\medskip
We extend Theorem~\ref{thm:uniqueness_diff_Meq} to a result on linearized uniqueness for the nonlinear inverse problem of simultaneously identifying $(M^{eq},R_1,R_2^*)$ when complementing the evolution \eqref{ibvp} after a $90^\circ$ pulse with two further $180^\circ-\tau-90^\circ$ pulse evolutions with different intermission times $\tau$. This is the same set of experiments as in Section~\ref{sec:uniqueness_perturbation}.

We consider the limit as $\taup\searrow0$ so that the dot in $\dot{=}$ of, e.g., \eqref{init90eps}, \eqref{init180tau90eps}, \eqref{init180tau90tauplus2eps} can be skipped. For this purpose, we will use the shorthand notation $v(t_0^+)=\lim_{t\searrow t_0}v(t)$, $v(t_0^-)=\lim_{t\nearrow t_0}v(t)$.
Neglecting diffusion during the infinitesimally short pulses but taking it into account during the intermission times and after the second pulse, we arrive at the following model, which we cast into an all-at-once formulation
\[
\begin{aligned}
&\mathbb{F}(x)=\ul{\ul{y}}\\
&\text{ where }\mathbb{F}=(\mathbb{F}^{(I)}_{mod},\mathbb{F}^{(I)}_{obs},\mathbb{F}^{(II)}_{mod},\mathbb{F}^{(II)}_{obs},\mathbb{F}^{(III)}_{mod},\mathbb{F}^{(III)}_{obs})\\
&\phantom{\text{ where }}
x=(\Mpp^{(I)},\Mpp^{(II)},M_z^{(II)},\Mpp^{(III)},M_z^{(III)},M^{eq},R_1,R_2^*)\\
&\phantom{\text{ where }}
\ul{\ul{y}}=(0,0,{\yp^{(I)}}_{\!\!1},\ldots,{\yp^{(I)}}_{\!\!N},0,0,0,0,{\yp^{(II)}}_{\!\!\!\!\!1},\ldots,{\yp^{(II)}}_{\!\!\!\!\!N},0,0,0,0,{\yp^{(III)}}_{\!\!\!\!\!\!\!\!1},\ldots,{\yp^{(III)}}_{\!\!\!\!\!\!\!\!N}\ )
\end{aligned}
\]
\Margin{R2 10)}
\begin{equation}\label{Faao}
\begin{aligned}
&\mathbb{F}^{(I)}_{mod}(x)
%(\Mpp^{(I)},\Mpp^{(II)},M_z^{(II)},\Mpp^{(III)},M_z^{(III)},M^{eq},R_1,R_2^*)\hspace*{5cm}\\&
=\left(\begin{array}{l}
\frac{d}{dt}\Mpp^{(I)}-\nabla\cdot\bigl(D \nabla \Mpp^{(I)}\bigr) +\bigl(R_2^*+i\gamma\vec r \cdot \vec G_0\bigr)\Mpp^{(I)}
\quad\text{ in }(0,T]\\
\Mpp^{(I)}(0^+)+iM^{eq}
\end{array}\right)
\\
%\end{aligned}\]
%\[\begin{aligned}
&\mathbb{F}^{(I)}_{obs,j}(x)
%(\Mpp^{(I)},\Mpp^{(II)},M_z^{(II)},\Mpp^{(III)},M_z^{(III)},M^{eq},R_1,R_2^*)\hspace*{5cm}\\&
=\int_{\mathbb{R}^3} c^-_j(\vec r) \Mpp^{(I)}(t, \vec r) \,\textsf{d}\vec r, \quad j\in\{1,\ldots N\}, \quad t\in(0,T]
\end{aligned}
\end{equation}
\[
\begin{aligned}
&\mathbb{F}^{(J)}_{mod}(x)
%(\Mpp^{(I)},\Mpp^{(II)},M_z^{(II)},\Mpp^{(III)},M_z^{(III)},M^{eq},R_1,R_2^*)\hspace*{4.7cm}\\&
=\left(\begin{array}{l}
\frac{d}{dt}M_z^{(J)}-\nabla\cdot\bigl(D\nabla M_z^{(J)}\bigr)+R_1 (M_z^{(J)}-M^{eq})
\quad\text{ in }(0,\tau^{(J)}]\\
M_z^{(J)}(0^+)+M^{eq}\\
\frac{d}{dt}\Mpp^{(J)}-\nabla\cdot\bigl(D \nabla \Mpp^{(J)}\bigr) +\bigl(R_2^*+i\gamma\vec r \cdot \vec G_0\bigr)\Mpp^{(J)}
\quad\text{ in }(\tau^{(J)},T]\\
\Mpp^{(J)}({\tau^{(J)}}^+)+iM_z^{(J)}({\tau^{(J)}}^-)
\end{array}\right)
\\
&\mathbb{F}^{(J)}_{obs,j}(x)
%(\Mpp^{(I)},\Mpp^{(II)},M_z^{(II)},\Mpp^{(III)},M_z^{(III)},M^{eq},R_1,R_2^*)\hspace*{4.7cm}\\&
=\int_{\mathbb{R}^3} c^-_j(\vec r) \Mpp^{(J)}(t, \vec r) \,\textsf{d}\vec r, \quad j\in\{1,\ldots N\}, \quad t\in(\tau^{(J)},T]\\
&J\in \{II,\, III\}
\end{aligned}
\]
For the initial conditions at 
%$\taup$ and $2\taup+\tau^{(J)}$, 
$0^+$ and $\tau^{(J)}$, 
we refer to \eqref{init90eps}, \eqref{init180tau90eps}, \eqref{init180tau90tauplus2eps}.

To prove linearized uniqueness at some reference point 
\[
x_{\text{ref}}=({\Mpp}_{\text{ref}}^{(I)},{\Mpp}_{\text{ref}}^{(II)},M_{z,\text{ref}}^{(II)},{\Mpp}_{\text{ref}}^{(III)},M_{z,\text{ref}}^{(III)},M^{eq}_{\text{ref}},R_{1\,\text{ref}},R_{2\,\text{ref}}^*),
\] 
we aim to draw the conclusion
\begin{equation}\label{Fprimex0impliesdx0}
\mathbb{F}'(x_{\text{ref}})\uld{x}=0
\quad \stackrel{!}{\Rightarrow}\quad 
\uld{x}=0.
\end{equation}

With the operators (cf. \eqref{mathcalA})
\begin{equation}\label{mathcalA12}
\begin{aligned}
&\mathcal{A}_z=
-\nabla\cdot\bigl(D \nabla \,.\,\bigr)+R_{1,\text{ref}}\\
&\mathcal{A}_\perp=
\left(\begin{array}{cc}
-\nabla\cdot\bigl(D \nabla \,.\,\bigr)+\Re(R_{2,\text{ref}}^*)&\Im(R_{2,\text{ref}}^*)+\gamma\vec r \cdot \vec G_0\\ 
-\Im(R_{2,\text{ref}}^*)-\gamma\vec r \cdot \vec G_0&-\nabla\cdot\bigl(D \nabla \,.\,\bigr)+\Re(R_2^*)
\end{array}\right)
\end{aligned}
\end{equation}
with eigenvalues $(\lambda_z^\ell)_{\ell\in\mathbb{N}}$, $(\lambda_\perp^\ell)_{\ell\in\mathbb{N}}$, and the notation
\[
\left(\begin{array}{cc}
\Re(\uld{R}_2^*)&\Im(\uld{R}_2^*)\\ 
-\Im(\uld{R}_2^*)&\Re(\uld{R}_2^*)
\end{array}\right)
\left(\begin{array}{c}
\Re({\Mpp}_{\text{ref}}^{(I)})\\
\Im({\Mpp}_{\text{ref}}^{(I)})
\end{array}\right)
=\uld{R}_2^*\,{\Mpp}_{\text{ref}}^{(I)},
\]
the model part of the 
\revision{premise}
\Margin{R1 7.} 
in \eqref{Fprimex0impliesdx0} can be written as 
\begin{equation}\label{FImodprime0}
\begin{aligned}
&\frac{d}{dt}{\uldMpp}^{(I)}+\mathcal{A}_\perp{\uldMpp}^{(I)} 
=-\uld{R}_2^*\,{\Mpp}_{\text{ref}}^{(I)}
\text{ in }(0,T]\\
&{\uldMpp}^{(I)}(0^+)=-i\uld{M}^{eq}
\end{aligned}
\end{equation}
\begin{equation}\label{FJmodprime0}
\begin{aligned}
&\frac{d}{dt}\uld{M}_z^{(J)}+\mathcal{A}_z\uld{M}_z^{(J)} 
=-\uld{R}_1\,(M_{z,\text{ref}}^{(J)}-M^{eq}_{\text{ref}})+R_{1\,\text{ref}}\,\uld{M}^{eq}
\text{ in }(0,\tau^{(J)}]\\
&\uld{M}_z^{(J)}(0^+)=-\uld{M}^{eq}
\\
%\end{aligned}\]
%\[\begin{aligned}
&\frac{d}{dt}{\uldMpp}^{(J)}+\mathcal{A}_\perp{\uldMpp}^{(J)} 
=-\uld{R}_2^*\,{\Mpp}_{\text{ref}}^{(J)}
\text{ in }(\tau^{(J)},T]\\
&{\uldMpp}^{(J)}({\tau^{(J)}}^+)=-i\uld{M}_z^{(J)}({\tau^{(J)}}^-)\\
&J\in\{II,III\}.
\end{aligned}
\end{equation}

We choose the longitudinal reference states to be space-time separable
\begin{equation}\label{spacetime-sep}
\begin{aligned}
%&{\Mpp}_{\text{ref}}^{(J)}(t, \vec r)=m_{\perp,\text{ref}}^{(J)}(t)\, \Phi_\perp(\vec r), \quad J\in\{I,II,III\}, \\
&M_{z,\text{ref}}^{(J)}(t, \vec r)-M^{eq}_{\text{ref}}(\vec r)=m_{z,\text{ref}}^{(J)}(t)\, \Phi_z(\vec r), \quad J\in\{II,III\}.
\end{aligned}
\end{equation} 
with 
%$\Phi_\perp$, 
$\Phi_z$ bounded away from zero and further conditions on 
%the Laplace transform of $m_{\perp,\text{ref}}^{(J)}$, 
$m_{z,\text{ref}}^{(J)}$ to be imposed below.

By means of the semigroup $(e^{-\mathcal{A}_z\,t})_{t>0}$ generated by $\mathcal{A}_z$ we can express $\uld{M}_z^{(J)}(\tau^{(J)})$ as 
\[
\begin{aligned}
&\uld{M}_z^{(J)}({\tau^{(J)}}^-)
=\mathbb{E}^{\tau^{(J)}}[\uld{M}^{eq}]-\widetilde{\mathbb{E}}^{\tau^{(J)}}[\uld{R}_1 \Phi_z]\\
&\text{with }\mathbb{E}^{\tau^{(J)}}=-e^{-\mathcal{A}_z\,\tau^{(J)}}+(1-e^{-\mathcal{A}_z\,\tau^{(J)}})\mathcal{A}_z^{-1}[R_{1,\text{ref}}\,.\,] \\
&\phantom{\text{with }}\widetilde{\mathbb{E}}^{\tau^{(J)}}
=\int_0^{\tau^{(J)}} m_{z,\text{ref}}^{(J)}(\sigma)\,e^{-\mathcal{A}_z\,(\tau^{(J)}-\sigma)}\, d\sigma.
\end{aligned}
\]
%where we have used \eqref{spacetime-sep}.
For the $\uldMpp$ equations, analyticity allows us to extend the time intervals in these abstract ODEs up to $+\infty$ and to apply Laplace transforms (starting at the left hand interval endpoint). Plugging the resulting solutions into the observations 
%(which are already linear) 
and multiplying the equations for $J\in\{II,III\}$ with $e^{s\,\tau^{(J)}}$ yields
\[
\begin{aligned}
&0=\int_{\mathbb{R}^3} c^-_j(\vec r) (s+\mathcal{A}_\perp)^{-1} \, w^{(J)}(s,\vec r) \,\textsf{d}\vec r
\\&\hspace*{3cm}
j\in\{1,\ldots N\}, \quad s\in\mathbb{C}\setminus\{-\lambda_\perp^m\,:\,m\in\mathbb{N}\}, \quad J\in\{I,II,III\}, \\
&\text{ where } \\
&w^{(I)}(s,\vec r)=-i\,\uld{M}^{eq}(\vec r)
%-(\mathcal{L}_0\,m_{\perp,\text{ref}}^{(I)})(s)\,(\uld{R}_2^*\Phi_\perp)(\vec r),\\
-(\mathcal{L}_0{\Mpp}_{\text{ref}}^{(I)})(s, \vec r)\,\uld{R}_2^*(\vec r)\\
&w^{(J)}(s,\vec r)=
-i\bigl[\mathbb{E}^{\tau^{(J)}}[\uld{M}^{eq}]-\widetilde{\mathbb{E}}^{\tau^{(J)}}[\uld{R}_1 \Phi_z]\bigr](\vec r)
%-(\mathcal{L}_{\tau^{(J)}}m_{\perp,\text{ref}}^{(J)})(s) \,(\uld{R}_2^*\Phi_\perp)(\vec r)
-e^{s\,\tau^{(J)}}(\mathcal{L}_{\tau^{(J)}}{\Mpp}_{\text{ref}}^{(J)})(s, \vec r)\,\uld{R}_2^*(\vec r)
\\& J\in\{II,III\}
 \end{aligned}
\]
Integrating over a circle $\Gamma_\ell(\rho)$ around $-\lambda_\perp^\ell$ and then shrinking its radius $\rho$ to zero we
obtain that 
\begin{equation}\label{orthoPell}
\langle \overline{c_j^-}, P_\ell w^{(J)}(-\lambda_\perp^\ell)\rangle_{L^2}=0, 
\quad j\in\{1,\ldots N\}, \quad \ell\in\mathbb{N}, \quad J\in\{I,II,III\}.
\end{equation}
Indeed, by Lebesgue's Dominated Convergence Theorem, we have 
\[
\lim_{\rho\to0} \int_{\Gamma_\ell(\rho)} c^-_j(\vec r) ((s+\mathcal{A}_\perp)^{-1} (\vec r) [w^{(J)}(s,\vec r)-w^{(J)}(-\lambda_\perp^\ell),\vec r)]\,\textsf{d}\vec r =0.
\] 
With the shorthand notation
\[
\mu_\ell^{(I)}(\vec r)=
%(\mathcal{L}_0\,m_{\perp,\text{ref}}^{(I)})(-\lambda_\perp^\ell), 
(\mathcal{L}_0{\Mpp}_{\text{ref}}^{(I)})(-\lambda_\perp^\ell, \vec r)
\quad 
\mu_\ell^{(J)}(\vec r)=
%(\mathcal{L}_{\tau^{(J)}}m_{\perp,\text{ref}}^{(J)})(-\lambda_\perp^\ell), 
(\mathcal{L}_{\tau^{(J)}}{\Mpp}_{\text{ref}}^{(J)})(-\lambda_\perp^\ell, \vec r)
\quad J\in\{II,III\},
\] 
the orthogonality relations \eqref{orthoPell}, that is, $P_\ell \overline{c_j^-}\perp w^{(J)}(-\lambda_\perp^\ell)$, read as
\begin{equation}\label{Pellcjperp}
\begin{aligned}
P_\ell \overline{c_j^-}\perp&-i\,\uld{M}^{eq}-\mu_\ell^{(I)}\,(\uld{R}_2^*
%\Phi_\perp
),\\
P_\ell \overline{c_j^-}\perp&-i\,\bigl[\mathbb{E}^{\tau^{(J)}}[\uld{M}^{eq}]-\widetilde{\mathbb{E}}^{\tau^{(J)}}[\uld{R}_1 \Phi_z]\bigr]
-e^{-\lambda_\perp^\ell\,\tau^{(J)}}\,\mu_\ell^{(J)}\,(\uld{R}_2^*
%\Phi_\perp
)\quad J\in\{II,III\}
\end{aligned}
\end{equation}
 
In order to allow for some elimination and achieve independence of $\ell$ in the right hand side, we choose 
\begin{equation}\label{assmperp}
%m_{\perp,\text{ref}}^{(J)}(t):=\bar{\mu}^{(J)}m_{\perp,\text{ref}}^{(I)}(t-\tau^{(J)}), 
{\Mpp}_{\text{ref}}^{(J)}(t,\vec r):=\bar{\mu}^{(J)} {\Mpp}_{\text{ref}}^{(I)}(t-\tau^{(J)},\vec r)
\quad t\geq\tau^{(J)}, \quad J\in\{II,III\}
\end{equation}
for some 
constants 
%time-independent functions 
$\bar{\mu}^{(II)}$, $\bar{\mu}^{(III)}$, 
so that 
\[
\begin{aligned}
e^{-\lambda_\perp^\ell\,\tau^{(J)}}\,\mu_\ell^{(J)}(\vec r)
&
%=e^{-\lambda_\perp^\ell\,\tau^{(J)}}\int_{\tau^{(J)}}^\infty {\Mpp}_{\text{ref}}^{(J)}(t,\vec r) e^{\lambda_\perp^\ell t}\, dt
=\bar{\mu}^{(J)}\,e^{-\lambda_\perp^\ell\,\tau^{(J)}}\int_{\tau^{(J)}}^\infty {\Mpp}_{\text{ref}}^{(I)}(t-\tau^{(J)},\vec r) e^{\lambda_\perp^\ell t}\, dt
\\&=
\bar{\mu}^{(J)}\, \mu_\ell^{(I)}(\vec r), \quad J\in\{II,III\}.
\end{aligned}
\] 
Note that our all-at-once formulation provides the freedom to make these choices of the reference states; in Remark~\ref{rem:refstates} we will show that \eqref{assmperp} is satisfied for ${\Mpp}_{\text{ref}}^{(J)}={\Mpp}_{180^\circ-\tau^{(J)}-90^\circ}$ as defined in \eqref{Mperp_ex_180-tau-90_rect_B}.
\\
Elimination of 
%$(\uld{R}_2^*\Phi_\perp)$ 
$\uld{R}_2^*$ 
thus yields 
\begin{equation}\label{Pellcj-w}
\begin{aligned}
&P_\ell \overline{c_j^-}\perp \ \mathbb{E}^{\tau^{(J)}}[\uld{M}^{eq}]-\widetilde{\mathbb{E}}^{\tau^{(J)}}[\uld{R}_1 \Phi_z]
-\bar{\mu}^{(J)}\uld{M}^{eq}\\ 
&J\in\{II,III\}\quad \ell\in\mathbb{N}
 \end{aligned}
\end{equation}
Therewith, analogously to above, \eqref{Pellcj-w} implies that 
%the expressions 
%\[
%\mathbb{E}^{\tau^{(J)}}[\uld{M}^{eq}]-\widetilde{\mathbb{E}}^{\tau^{(J)}}[\uld{R}_1 \Phi_z]-\bar{\mu}^{(J)}\uld{M}^{eq} \quad J\in\{II,III\}
%\]
with the operator defined by 
\begin{equation}\label{mathbfE}
\mathbf{E}=
\left(\begin{array}{cc}
\mathbb{E}^{\tau^{(II)}}\!\!\!-\bar{\mu}^{(II)} & \widetilde{\mathbb{E}}^{\tau^{(II)}}\\
\mathbb{E}^{\tau^{(III)}}\!\!\!-\bar{\mu}^{(III)} & \widetilde{\mathbb{E}}^{\tau^{(III)}}
\end{array}\right)
\end{equation}
both components of $\mathbf{E}(\uld{M}^{eq},\uld{R}_1 \Phi_z)$ must vanish, provided they are contained in $\Xpar$ satisfying \eqref{assXpar_uniquenessD}.

In terms of the eigensystem $(\lambda_z^\ell,(\varphi^\ell_k)_{k\in\{1,\ldots,d_\ell\}})_{\ell\in\mathbb{N}}$ of $\mathcal{A}_z$, and the generalized Fourier coefficients
\[
a^\ell_k:=\langle \uld{M}^{eq},\varphi^\ell_k\rangle, \quad
b^\ell_k:=\langle \uld{R}_1\Phi_z,\varphi^\ell_k\rangle, 
\]
with $R_{1,\text{ref}}(\vec r)\equiv R_{1,0}$, the equation $-\mathbf{E}(\uld{M}^{eq},\uld{R}_1 \Phi_z)=0$ can be written as a sequence of $2\times2$ systems
\begin{equation}\label{sysab}
A_\ell^{(J)} \, a^\ell_k 
+ B_\ell^{(J)} \, b^\ell_k=0 \quad J\in\{II,III\} \quad \ell\in\mathbb{N}
\end{equation}
with 
\begin{equation}\label{AellJBellJ}
\begin{aligned}
&A_\ell^{(J)}=e^{-\lambda_z^\ell\,\tau^{(J)}}
-\frac{R_{1,0}}{\lambda_z^\ell}(1-e^{-\lambda_z^\ell\,\tau^{(J)}})+\bar{\mu}^{(J)}\\
&B_\ell^{(J)}=\int_0^{\tau^{(J)}} m_{z,\text{ref}}^{(J)}(\sigma)\,e^{-\lambda_z^\ell\,(\tau^{(J)}-\sigma)}\, d\sigma
\end{aligned}
\end{equation}
We will require ${\Mpp}_{\text{ref}}^{(J)}$, ${M_z}_{\text{ref}}^{(J)}$ to be chosen such that the determinants of each of these systems are nonzero
\begin{equation}\label{det_ell}
A_\ell^{(II)}\, B_\ell^{(III)}-A_\ell^{(III)}\, B_\ell^{(II)}\not=0, \quad \ell\in\mathbb{N}
\end{equation}
and in Remark~\ref{rem:refstates} will show that in fact it this is satisfied with the canonical choice of the reference states by the explicit solution formulas from Subsection~\ref{sec:explsol}.

\begin{comment}
A sufficient condition for the determinants of these systems to be nonzero is
\begin{equation}\label{assmz}
m_{z,\text{ref}}^{(II)}(\sigma)=m_{z,\text{ref}}^{(III)}(\sigma)\equiv \bar{m}_z, \quad 
\bar{\mu}^{(II)}(\vec r)=\bar{\mu}^{(III)}(\vec r)\equiv\bar{\mu}\not=-1, \quad 
\tau^{(II)}\not=\tau^{(III)}.
\end{equation}
Indeed, in this setting, \eqref{sysab} becomes
\[
\begin{aligned}
\Bigl(e^{-\lambda_z^\ell\,\tau^{(J)}}+R_{1,0}\frac{1-e^{-\lambda_z^\ell\,\tau^{(J)}}}{\lambda_z^\ell}+\bar{\mu}\Bigr)\, a^\ell_k
+\bar{m}_z\frac{1-e^{-\lambda_z^\ell\,\tau^{(J)}}}{\lambda_z^\ell}\, b^\ell_k=0 \quad J\in\{II,III\}.
\end{aligned}
\]
or equivalently 
\[
\begin{aligned}
(e^{-\lambda_z^\ell\,\tau^{(J)}}+\bar{\mu})\, a^\ell_k
+\frac{1-e^{-\lambda_z^\ell\,\tau^{(J)}}}{\lambda_z^\ell}\, (R_{1,0}a^\ell_k+\bar{m}_zb^\ell_k)=0 \quad J\in\{II,III\},
\end{aligned}
\]
with determinant
\[
\begin{aligned}
&\frac{1}{\lambda_z^\ell}\Bigl((e^{-\lambda_z^\ell\,\tau^{(II)}}+\bar{\mu})(1-e^{-\lambda_z^\ell\,\tau^{(III)}})
-(e^{-\lambda_z^\ell\,\tau^{(III)}}+\bar{\mu})(1-e^{-\lambda_z^\ell\,\tau^{(II)}})\Bigr)\\
&=\frac{1+\bar{\mu}}{\lambda_z^\ell}\Bigl(e^{-\lambda_z^\ell\,\tau^{(II)}}-e^{-\lambda_z^\ell\,\tau^{(III)}}\Bigr)\not=0.
\end{aligned}
\]
\end{comment}
This implies $a^\ell_k=b^\ell_k=0$, for all $\ell\in\mathbb{N}$, $k\in\{1,\ldots,d_\ell\}$, hence $\uld{M}^{eq}=0$, $\uld{R}_1\Phi_z=0$.
Inserting this into the first equation of \eqref{Pellcjperp} and assuming
%\TODO{Maybe better take the residue at this point in the proof?}
\begin{equation}\label{muIell}
\mu_\ell^{(I)}(\vec r)\not=0\text{ for a.e. }{\vec r}\in\Omega
\end{equation}
implies that  
%$\uld{R}_2^*\Phi_\perp$ 
$\uld{R}_2^*$ 
vanishes, provided it is contained in $\Xpar$. 
Thus, the initial and right hand side data in \eqref{FImodprime0}, \eqref{FJmodprime0} vanish and so, as a consequence, do the states ${\uldMpp}^{(J)}$, $\uld{M}_z^{(J)}$.

With this and the time shifted solution spaces defined by 
\[
\begin{aligned}
&\mathbb{V}_{\perp\,\tau}=L^\infty(\tau,T;L^2(\Omega;\mathbb{C}))\cap L^2(\tau,T;H_D^1(\Omega;\mathbb{C}))\\
&\mathbb{V}_{z\,\tau}=L^\infty(\tau,T;L^2(\Omega;\mathbb{R}))\cap L^2(\tau,T;H_D^1(\Omega;\mathbb{R})),
\end{aligned}
\]
we have proven
\begin{theorem}[\revisionown{linearized uniqueness of $(M^{eq},R_1,R_2,\delta B^0)$ in Bloch-Torrey PDE}]\label{thm:uniqueness_diff}
Let the assumptions of Theorem~\ref{thm:uniqueness_diff_Meq} hold.
Moreover, assume that \eqref{spacetime-sep}, \eqref{assmperp}, 
%\eqref{assmz},
\eqref{AellJBellJ}, \eqref{det_ell}, 
\eqref{muIell} hold with $R_{1,\text{ref}}(\vec r)\equiv R_{1,0}$.

Then $\mathbb{F}$ as defined by \eqref{Faao} has an injective linearization $\mathbb{F}'(x_{\text{ref}})$ on 
\[
\widetilde{X}=
\mathbb{V}_{\perp\,0}\times \mathbb{V}_{\perp\,\tau^{(II)}}\times\mathbb{V}_{z\,0}\times \mathbb{V}_{\perp\,\tau^{(III)}}\times\mathbb{V}_{z\,0}
\times \mathbf{E}^{-1}(\Xpar\times\Xpar)\times 
%\tfrac{1}{\Phi_\perp}
\Xpar
\]
with $\mathbf{E}$ as in \eqref{mathbfE}.
\end{theorem}

\begin{remark}\label{rem:refstates} 
We now verify the assumptions \eqref{spacetime-sep}, \eqref{assmperp}, 
%\eqref{assmz} 
\eqref{AellJBellJ}, \eqref{det_ell}, \eqref{muIell}
of Theorem~\ref{thm:uniqueness_diff} for the evident choice of the reference states provided by the (approximately) explicit formulas in the Bloch ODE case with $R_{1\,\text{ref}}(\vec r)\equiv\tilde{R}_{1,0}$, $R_{2\,\text{ref}}^*(\vec r)\equiv\tilde{R}_{2,0}^*$, cf. \eqref{Mperp_ex_90_rect_B}, \eqref{init180tau90eps}, \eqref{Mperp_ex_180-tau-90_rect_B}. 
Indeed, \eqref{spacetime-sep} with $m_{z,\text{ref}}^{(J)}(t)=-2e^{-\tilde{R}_{1,0}t}$, $\Phi_z(\vec r)=M^{eq}_{\text{ref}}(\vec r)$ is immediate from \eqref{init180tau90eps} %\eqref{B_uncoupled}
and a simple computation yields
\[
B_\ell^{(J)}=
\int_0^{\tau^{(J)}} m_{z,\text{ref}}^{(J)}(\sigma)\,e^{-\lambda_z^\ell\,(\tau^{(J)}-\sigma)}\, d\sigma
= - \frac{2}{\lambda_z^\ell-\tilde{R}_{1,0}}\, 
\Bigl(e^{-\tilde{R}_{1,0}\,\tau^{(J)}}-e^{-\lambda_z^\ell\,\tau^{(J)}}\Bigr).
\]
Also \eqref{assmperp} is naturally satisfied with 
\[
\bar{\mu}^{(J)}=1-2e^{-\tilde{R}_{1,0}\tau^{(J)}}, 
\]
cf., \eqref{Mperp_ex_90_rect_B}, \eqref{Mperp_ex_180-tau-90_rect_B} with ${\vec G}(t)\equiv{\vec G}_0$, hence 
${\vec k}(t)={\vec G}_0t$ in \eqref{Mperp_ex_90_rect_B} 
${\vec k}(t)={\vec G}_0(t-\tau^{(J)})$ in \eqref{Mperp_ex_180-tau-90_rect_B}.
To verify \eqref{AellJBellJ}, \eqref{det_ell}, we observe that 
\skipsimple{
\[
\begin{aligned}
A_\ell^{(J)}
&=\frac{\lambda_z^\ell-R_{1,0}}{\lambda_z^\ell}
+\frac{\lambda_z^\ell+R_{1,0}}{\lambda_z^\ell} e^{-\lambda_z^\ell\,\tau^{(J)}} 
-2 e^{-\tilde{R}_{1,0}\,\tau^{(J)}}
\\
&=(\lambda_z^\ell-\tilde{R}_{1,0})B_\ell^{(J)}
+\frac{\lambda_z^\ell-R_{1,0}}{\lambda_z^\ell}
\bigl(1-e^{-\lambda_z^\ell\,\tau^{(J)}}\bigr)
\\[1ex]
B_\ell^{(J)}
&=
\frac{2}{\lambda_z^\ell-\tilde{R}_{1,0}}
\Bigl(\bigl(1-e^{-\tilde{R}_{1,0}\,\tau^{(J)}}\bigr)
-\bigl(1-e^{-\lambda_z^\ell\,\tau^{(J)}}\bigr)\Bigr)
\end{aligned}
\]
hence by linearity of the determinant
}
\[
\begin{aligned}
&A_\ell^{(II)}\, B_\ell^{(III)}-A_\ell^{(III)}\, B_\ell^{(II)}\\
&\skipsimple{
=\frac{\lambda_z^\ell-R_{1,0}}{\lambda_z^\ell}
\Bigl(
\bigl(1-e^{-\lambda_z^\ell\,\tau^{(II)}}\bigr)\, B_\ell^{(III)}-\bigl(1-e^{-\lambda_z^\ell\,\tau^{(III)}}\bigr)\, B_\ell^{(II)}
\Bigr)
}
%\\&
=\frac{2(\lambda_z^\ell-R_{1,0})}{\lambda_z^\ell(\lambda_z^\ell-\tilde{R}_{1,0})}
\Bigl(
\bigl(1-e^{-\lambda_z^\ell\,\tau^{(II)}}\bigr)\, \bigl(1-e^{-\tilde{R}_{1,0}\,\tau^{(III)}}\bigr)
-\bigl(1-e^{-\lambda_z^\ell\,\tau^{(III)}}\bigr)\, \bigl(1-e^{-\tilde{R}_{1,0}\,\tau^{(II)}}\bigr)
\Bigr)\\
&\not=0, \quad \ell\in\mathbb{N}.
\end{aligned}
\]

Also \eqref{muIell} can be easily verified with a choice of $M^{eq}_{\text{ref}}(\vec r)\not=0$ almost everywhere, by using cf. \eqref{Mperp_ex_90_rect_B}
\[
\begin{aligned}
\mu_\ell^{(I)}(\vec r)
=\int_0^\infty{\Mpp}_{\text{ref}}^{(I)}(t,\vec r)e^{\lambda_\perp^\ell\,t}\,dt
=-i\int_0^\infty e^{-(\tilde{R}_{2,0}^*+2\pi i {\vec G}_0\cdot{\vec r}+\lambda_\perp^\ell)\,t)}\,dt
\,M^{eq}(\vec r).
\end{aligned}
\]

\end{remark}

\appendix
\section{}
\subsection{Relaxation time $R_1$ and $R_2^*$ reconstruction formulas from elementary $90{}^{\circ}$ pulse and $180{}^{\circ}$ pulse -- $\tau$ intermission -- $90{}^{\circ}$ pulse data}
\label{sec:appendix_reconR1R2} 
%\textit{$R_2^*(\vec r)$ recovery via differentiaton with respect to time:}\\[1ex]
\subsubsection*{$R_2^*(\vec r)$ recovery from $90{}^{\circ}$ pulse data via differentiaton with respect to time}
Differentiating the identity cf. \eqref{MR-Fourier_rect_B} 
\[
\tilde{y}_j(t):=i\,e^{i\omega_0t} \, y_j(t) 
\dot{=} \int_{\mathbb{R}^3} e^{-R_2^*(\vec r)(t-\taup)} (c^-_j\,M^{eq})(\vec r) \, e^{-2\pi i {\vec k}(t)\cdot{\vec r}}  \,\textsf{d}\vec r
,\qquad j=1,\ldots N,
\]
with respect to time and then multiplying with $e^{R_{2,0}^*(t-\taup)}$ yields
\begin{equation}\label{pre-reconR2_rect}
\begin{aligned}
&e^{R_{2,0}^*(t-\taup)}\, \frac{d}{dt}\tilde{y}_j(t)
\dot{=} -\int_{\mathbb{R}^3} (R_2^*(\vec r)+i\gamma {\vec G}(t)\cdot{\vec r})\,e^{-\delta R_2^*(\vec r)(t-\taup)} (c^-_j\,M^{eq})(\vec r) \, e^{-2\pi i {\vec k}(t)\cdot{\vec r}}  \,\textsf{d}\vec r\\
&\approx
 -\int_{\mathbb{R}^3} (R_2^*(\vec r)+i\gamma {\vec G}(t)\cdot{\vec r}) (c^-_j\,M^{eq})(\vec r) \, e^{-2\pi i {\vec k}(t)\cdot{\vec r}}  \,\textsf{d}\vec r\\
&=-\mathcal{F}[R_2^*c^-_j\,M^{eq}]({\vec k}(t))
+\frac{\gamma}{2\pi} {\vec G}(t)\cdot\nabla_\xi \mathcal{F}[c^-_j\,M^{eq}]({\vec k}(t))
,\qquad j=1,\ldots N,
\end{aligned}
\end{equation}
where again $\approx$ is an approximation for vanishing $\delta R_2^*$, cf. \eqref{R2effdecomp},
thus
\begin{equation}\label{reconR2_rect}
\mathcal{F}[R_2^*c^-_j\,M^{eq}]({\vec k}(t))
\approx \frac{\gamma}{2\pi} {\vec G}(t)\cdot\nabla_\xi \mathcal{F}[c^-_j\,M^{eq}]({\vec k}(t))
-i\,e^{R_{2,0}^*(t-\taup)}\, \frac{d}{dt}\Bigl(e^{i\omega_0t} \, y_j(t)\Bigr)
\end{equation}
The estimate
\[
|1-e^{-\delta R_2^*(\vec r)(t-\taup)}|=
|e^{-\theta}\,\delta R_2^*(\vec r)(t-\taup)|<<|\delta R_2^*(\vec r)|
\]
for some intermediate value $\theta\in[0,\delta R_2^*(\vec r)(t-\taup)]$, justifies neglecting $\delta R_2^*(\vec r)$ in the exponential term of \eqref{pre-reconR2_rect}, while keeping it in the linear term. 

One can use formula \eqref{reconR2_rect} by first of all determining $R_{2,0}^*$ as the unique solution to the scalar nonlinear regression problem 
\[
a_j(t)\, R_{2,0}^* + b_j(t) e^{R_{2,0}^*(t-\taup)}= \frac{\gamma}{2\pi} {\vec G}(t)\cdot\nabla_\xi \mathcal{F}[c^-_j\,M^{eq}]({\vec k}(t))\qquad j=1,\ldots N, \ t>\taup
\] 
with 
\[
a_j(t)=\mathcal{F}[c^-_j\,M^{eq}]({\vec k}(t))\,dt, \quad
b_j(t)= i \, \frac{d}{dt}\Bigl(e^{i\omega_0t} \, y_j(t)\Bigr)
\]
and then using this constant $R_{2,0}^*\in\mathbb{C}$ for reconstructing the spatially variable part from \eqref{reconR2_rect}.

Differentiation with respect to time and (k-)space in \eqref{reconR2_rect} makes this reconstruction formula mildly ill-posed. We thus follow an approach avoiding differentiation in Section~\ref{sec:uniqueness_perturbation}.

\subsubsection*{$R_2^*(\vec r)$ recovery from $180{}^{\circ}$ pulse -- $\tau$ intermission -- $90{}^{\circ}$ pulse data}

From \eqref{obsPhi1_rect_B}, $R_1(\vec r)\in\mathbb{R}$ can be recovered by either using the already reconstructed $M^{eq}(\vec r)$ 
\begin{equation}\label{reconR1_Meq_rect}
R_1(\vec r)=\frac{1}{\tau}\left(\ln 2 -\ln\left(1-\left|\frac{\Phi_\tau(\vec r)}{M^{eq}(\vec r)}\right|\right)\right)
\end{equation}
or two different intermission times $\tau\in\{\tau_1,\tau_2\}$
\begin{equation}\label{reconR1_twotaus_rect}
\left|\frac{\Phi^{(1)}_{\tau_1}(\vec r)}{\Phi^{(1)}_{\tau_2}(\vec r)}\right|
=\psi(-R_1(\vec r))
\text{ with }\psi(x)=\frac{1-2e^{\tau_1 x}}{1-2e^{\tau_2 x}}
\end{equation}
where by elementary computations 
%\arxivonly{
\begin{equation}\label{psiprime}
\begin{aligned}
\psi'(x)&=2(1-2e^{\tau_2 x})^{-2} e^{\tau_1 x}\, \tilde{\psi}(x)\text{ with }
\tilde{\psi}(x)=(\tau_2-\tau_1) (1-2e^{\tau_2 x})+\tau_2(e^{(\tau_2-\tau_1) x}-1)\\
\tilde{\psi}'(x)&=\tau_2(\tau_2-\tau_1)e^{(\tau_2-\tau_1) x}(1-2e^{\tau_2 x})\text{ hence }
\tilde{\psi}'(\tilde{x})=0 \quad \Leftrightarrow \tilde{x}=-\frac{\ln2}{\tau_1}<0\\
\tilde{\psi}(\tilde{x})&=\tau_1(e^{(\tau_2-\tau_1) \tilde{x}}-1)<0
, \quad \tilde{\psi}(0)=-(\tau_2-\tau_1), \quad \lim_{x\to-\infty}\tilde{\psi}(x)= -\tau_1\\
\tilde{\psi}(x)&\leq\max\{\tilde{\psi}(\tilde{x}),\,\tilde{\psi}(0),\,\lim_{x\to-\infty}\tilde{\psi}(x)\}<0
\end{aligned}
\end{equation}
and $2(1-2e^{\tau_2 x})^{-2} e^{\tau_1 x}\geq 2e^{\tau_1 x}\geq 2e^{-\tau_1 \overline{R}_1}>0$ for $x\in[-\overline{R}_1,0]$, 
%}
the derivative $\psi'$ can be bounded away from zero,
which implies bijectivity of $\psi$ and thus unique (and stable) reconstructability of $R_1(\vec r)$ from \eqref{reconR1_Meq_rect}.

\medskip

%\begin{remark}\label{kspace_R1R2}
%Also here, analogous computations could be carried out in the constant coefficient diffusive setting \eqref{eqn:bloch-torrey-complex_rot_FT}, but this would only yield single constants $R_{1,0}$, $R_{2,0}^*$ and be unable to capture space-dependency of the relaxation times.
%\end{remark}

\subsection{Proofs of Lemma ~\ref{lem_rect} and Proposition~\ref{prop:perturb_expl-sol_lin}} 
\subsubsection*{Proof of Lemma~\ref{lem_rect}}
The eigenvalues and -vectors of the matrix $A$ governing the system \eqref{3by3ODEsys} of first order ODEs are
\[
\begin{aligned}
&\lambda_0=\alpha_2, \quad 
&&\lambda_\pm = \frac{\alpha_1+\alpha_2}{2}\pm i\sqrt{|\bb|^2-\left(\frac{\alpha_2-\alpha_1}{2}\right)^2}, \\
&v_0=\left(\begin{array}{c}\bb_x\\ \bb_y\\0\end{array}\right), \quad 
&&v_\pm=\left(\begin{array}{c}
%-(\alpha_1-\lambda_\pm)\bb_y\\ (\alpha_1-\lambda_\pm)\bb_x\\ 
\left(\frac{\alpha_2-\alpha_1}{2}\pm i\sqrt{|\bb|^2-\left(\frac{\alpha_2-\alpha_1}{2}\right)^2}\right)\bb_y\\ 
-\left(\frac{\alpha_2-\alpha_1}{2}\pm i\sqrt{|\bb|^2-\left(\frac{\alpha_2-\alpha_1}{2}\right)^2}\right)\bb_x\\ 
|\bb|^2\end{array}\right)
\end{aligned}
\]
Thus, with the matrices 
\[
\begin{aligned}
&\mathbb{T}=\left(\begin{array}{ccc}
\frac{\bb_x}{|\bb|}&
(d+i\sqrt{1-d^2})\, \frac{\bb_y}{|\bb|}&
(d-i\sqrt{1-d^2})\, \frac{\bb_y}{|\bb|}
\\ 
\frac{B^1_y}{|B^1|}&
-(d+i\sqrt{1-d^2})\, \frac{\bb_x}{|\bb|}&
-(d-i\sqrt{1-d^2})\, \frac{\bb_x}{|\bb|}
\\ 
0&1&1
\end{array}\right), \\
&\mathbb{T}^{-1}=\left(\begin{array}{ccc}
\frac{\bb_x}{|\bb|}&\frac{\bb_y}{|\bb|}&0\\
\frac{1}{2i\sqrt{1-d^2}}\frac{\bb_y}{|\bb|}&-\frac{1}{2i\sqrt{1-d^2}}\frac{\bb_x}{|\bb|}&
%-\frac{d-i\sqrt{1-d^2}}{2i\sqrt{1-d^2}}
\frac12-\frac{d}{2i\sqrt{1-d^2}}
\\
-\frac{1}{2i\sqrt{1-d^2}}\frac{\bb_y}{|\bb|}&\frac{1}{2i\sqrt{1-d^2}}\frac{\bb_x}{|\bb|}&
%\frac{d+i\sqrt{1-d^2}}{2i\sqrt{1-d^2}}
\frac12+\frac{d}{2i\sqrt{1-d^2}}
\end{array}\right)
\end{aligned}
\]
with $d=\frac{\alpha_2-\alpha_1}{2|\bb|}$, we get $\mathbb{T}^{-1}A\mathbb{T}=\text{diag}_\lambda:=\text{diag}(\lambda_0,\lambda_+,\lambda_-)$ and can compute the solution as 
\begin{equation}\label{mtexplicit0}
\begin{aligned}
&m(t)=\tilde{f}+\mathbb{T} e^{-(t-t_0)\text{diag}_\lambda}\mathbb{T}^{-1}\left(m(t_0)-\tilde{f}\right)
\\
&\text{where }\tilde{f}=\mathbb{T}\text{diag}_\lambda^{-1}\mathbb{T}^{-1}\left(\begin{array}{c}0\\0\\f\end{array}\right)=
\frac{f}{|\bb|^2+\alpha_1\alpha_2}\left(\begin{array}{c}\bb_y\\-\bb_x\\ \alpha_2\end{array}\right)
%  (f*y)/(x^2 + y^2 + al1*al2)
% -(f*x)/(x^2 + y^2 + al1*al2)
%(al2*f)/(x^2 + y^2 + al1*al2)
\end{aligned}
\end{equation}
Without loss of generality we take $\bb_x=|\bb|$, $\bb_y=0$ to obtain
\begin{equation}\label{TTex}
\begin{aligned}
&\mathbb{T}=
%\left(\begin{array}{ccc}
%1&0&0\\ 
%0&-(d+i\sqrt{1-d^2})&-(d-i\sqrt{1-d^2})\\ 
%0&1&1
%\end{array}\right)
\mathbb{T}_{ex}+O(d), \qquad 
\mathbb{T}^{-1}=\mathbb{T}^{-1}_{ex}+O(d)\\
&\text{with }
\mathbb{T}_{ex}
=\left(\begin{array}{ccc}
1&0&0\\ 
0&-i&i\\ 
0&1&1
\end{array}\right), \qquad
\mathbb{T}^{-1}_{ex}=\left(\begin{array}{ccc}
1&0&0\\
0&\frac{i}{2}&\frac12
\\
0&-\frac{i}{2}&\frac12
\end{array}\right)\\
&\lambda_0=\alpha_2, \quad \lambda_\pm=|\bb|(\pm i+O(|\bb|^{-1})),
\end{aligned}
\end{equation}
where $O(d)=O(|\bb|^{-1})$ and moreover, $\tilde{f}=O(|\bb|^{-1})f$.
Therefore 
\begin{equation}\label{mtexplicit}
%\hspace*{-1cm}
\begin{aligned}
&m(t)=m_{ex}(t)\,+O(|\bb|^{-1})m(t_0)\, +O(|\bb|^{-1})f\\
&\text{with }m_{ex}(t)=\left(\begin{array}{ccc}
e^{-\alpha_2(t-t_0)}&0&0\\ 
0&\cos(|\bb|(t-t_0))&-\sin(|\bb|(t-t_0))\\ 
0&\sin(|\bb|(t-t_0))&\cos(|\bb|(t-t_0))
\end{array}\right)m(t_0)
\end{aligned}
\end{equation}
%This proves Lemma~\ref{lem_rect} in case of constant real valued $\pp c^+$.
%For a general time dependent Lebesgue integrable function $\pp(t)$, the assertion follows via appoximation by simple functions (that is, step functions taking finitely many values) and superposition.

\subsubsection*{Proof of Proposition~\ref{prop:perturb_expl-sol_lin}}
Analogously to \eqref{mtexplicit0}, \eqref{mtexplicit},
for the solution $m$ to the the more general ODE system
\begin{equation*}%\label{3by3ODEsys_gen}
\frac{d}{dt}\left(\begin{array}{c}m_x(t)\\m_y(t)\\m_z(t)\end{array}\right)
+
\left(\begin{array}{ccc}
\alpha_2&0&0\\ 
0&\alpha_2&|\bb|\\
0&-|\bb|&\alpha_1\end{array}\right)
\left(\begin{array}{c}m_x(t)\\m_y(t)\\m_z(t)\end{array}\right)
={\vec f}(t)\quad t\in[t_0,t_0+\taup]
\end{equation*}
in place of \eqref{3by3ODEsys} we have 
\begin{equation}\label{mtexplicit_gen}
\begin{aligned}
&m(t)=\mathbb{T}\Bigl(e^{-(t-t_0)\text{diag}_\lambda}\mathbb{T}^{-1}m(t_0)
+\int_{t_0}^te^{-(t-s)\text{diag}_\lambda}\mathbb{T}^{-1}{\vec f}(s)\, ds\Bigr)\\
&=m_{ex}(t)\,+O(|\bb|^{-1})|m(t_0)|\, +O(|\bb|^{-1})|{\vec f}|\\
&\text{with }m_{ex}(t)=\mathbb{T}_{ex}\Bigl(e^{-(t-t_0)\text{diag}_\lambda}\mathbb{T}^{-1}_{ex}m(t_0)
+\int_{t_0}^te^{-(t-s)\text{diag}_{\lambda,ex}}\mathbb{T}^{-1}_{ex}{\vec f}(s)\, ds\Bigr)
\end{aligned}
\end{equation}
with $\mathbb{T}$, $\mathbb{T}_{ex}$
as in \eqref{TTex} and $\lambda_{0,ex}=\alpha_2$, $\lambda_{\pm,ex}=\pm i|\bb|$.
This allows us to assess the error in the linearization of \eqref{B_rect} with respect to $(M^{eq},R_1,R_2^*)$, which is given by 
\begin{equation}\label{B_rect_lin}
\begin{aligned}
&\frac{d}{dt}\uldMpp(t, \vec r) 
+ R_2^*(\vec r) \uldMpp(t, \vec r)-i\gamma \taup^{-1}\,  \uld{M}_z(t, \vec r)c^+(\vec r) \\
&\qquad\qquad=-\uld{R}_2^*(\vec r) \Mpp(t, \vec r)\\    
&\frac{d}{dt}\uld{M}_z(t, \vec r)  
+R_1(\vec r) \uld{M}_z(t, \vec r)+\gamma \taup^{-1}\, \Re\bigl(i\uldMpp(t, \vec r)c^+(\vec r)\bigr)\\
&\qquad\qquad=-\uld{R}_1(\vec r) (M_z(t, \vec r)-M^{eq}(\vec r))+R_1(\vec r)\uld{M}^{eq}(\vec r))\\
&\vec r\in\Omega, \quad t\in[t_0,t_0+\taup]
\end{aligned}
\end{equation}
The estimate therewith relies on \eqref{mtexplicit_gen} and a propagation of this error through \eqref{B_uncoupled}.

\subsection{Derivation of \eqref{FIprimeD0}, \eqref{FIIprimeD0}}
During the $90^\circ$ and $180^\circ$ pulses, the inhomogeneous right hand sides lead to $O(\taup)$ terms and so as in \eqref{init90eps}, \eqref{init180tau90eps}, \eqref{init180tau90tauplus2eps} we obtain (skipping the subscript ``ref'' in $\Mphp$ and $M_z$)
\begin{itemize}
\item with $\pp={\pp}_{90^\circ}$:
\begin{equation*}%\label{init90eps_appendix}
\hspace*{-1cm}\begin{aligned}
&\Mpp(\taup,\cdot)\dot{=}-iM_z(0,\cdot)=-iM^{eq}
&&\uldMpp(\taup,\cdot)\dot{=}-iM_z(0,\cdot)=-i\uld{M}^{eq}\\
&M_z(\taup,\cdot)\dot{=}\Im(\Mpp(0,\cdot))=0
&&\uld{M}_z(\taup,\cdot)\dot{=}\Im(\uldMpp(0,\cdot))=0
\end{aligned}
\end{equation*}
\item with $\pp={\pp}_{180^\circ-\tau-90^\circ}$:
\begin{equation*}%\label{init180tau90eps_appendix}
\hspace*{-1cm}\begin{aligned}
&\Mpp(\taup,\cdot)\dot{=}-\Mpp(0,\cdot)=0
&&\uldMpp(\taup,\cdot)\dot{=}-\uldMpp(0,\cdot)=0\\
&M_z(\taup,\cdot)\dot{=}-M_z(0,\cdot)=-M^{eq}
&&\uld{M}_z(\taup,\cdot)\dot{=}-\uld{M}_z(0,\cdot)=-\uld{M}^{eq}\\
&\Mpp(\taup+\tau,\cdot)\dot{=}0
&&\uldMpp(\taup+\tau,\cdot)\dot{=}0\\
&M_z(\taup+\tau,\cdot)\dot{=}\mathbb{L}_\tau M^{eq}
&&\uld{M}_z(\taup+\tau,\cdot)\dot{=}\mathbb{L}_\tau \uld{M}^{eq}\\
&&&
-\int_\taup^{\tau+\taup}e^{-\mathcal{A}_1(\tau+\taup-s)}\uld{R}_1(\mathbb{L}_{s-\taup}-1)M^{eq}\\
&\Mpp(\tau+2\taup,\cdot)\dot{=}-iM_z(\taup+\tau,\cdot)
&&\uldMpp(\tau+2\taup,\cdot)\dot{=}-i\uld{M}_z(\taup+\tau,\cdot)\\
&M_z(\tau+2\taup,\cdot)\dot{=}\Im(\Mpp(\taup+\tau,\cdot))\dot{=}0
&&\uld{M}_z(\tau+2\taup,\cdot)\dot{=}\Im(\uldMpp(\taup+\tau,\cdot))\dot{=}0,
\end{aligned}
\end{equation*}
\end{itemize}
By the method of characteristics, we obtain 
\begin{comment}
\begin{equation*}
\begin{aligned}
&\Mphp(a,\vec\zeta-{\vec k}(a))\dot{=}
e^{-\int_{t_0}^a((\vec\zeta-\vec k(\sigma))\cdot D_0(\vec\zeta-\vec k(\sigma))+R_{2,0}^*)\, d\sigma}\,\Mphp(t_0,\vec\zeta)\\[1ex]
&\uld{\Mphp}(t,\vec\zeta-{\vec k}(t))\dot{=}
e^{-\int_{t_0}^t((\vec\zeta-\vec k(\sigma))\cdot D_0(\vec\zeta-\vec k(\sigma))+R_{2,0}^*)\, d\sigma}\,\uld{\Mphp}(t_0,\vec\zeta)
\\&-\int_{t_0}^t e^{-\int_a^t((\vec\zeta-\vec k(\sigma))\cdot D_0(\vec\zeta-\vec k(\sigma))+R_{2,0}^*)\, d\sigma}
[\widehat{\uld{R_2^*}}*{\Mphp}(a)](\vec\zeta-\vec k(a))\, da\\[1ex]
&\text{ for }\vec\zeta\in\mathbb{R}^3, \quad t>t_0,
\end{aligned}
\end{equation*}
\begin{equation*}
\begin{aligned}
&\Mphp(a,\vec\zeta+{\vec k}(a)-{\vec k}(t_0))\dot{=}
e^{-\int_{t_0}^a((\vec\zeta+\vec k(\sigma)-{\vec k}(t_0))\cdot D_0(\vec\zeta+\vec k(\sigma)-{\vec k}(t_0))+R_{2,0}^*)\, d\sigma}\,\Mphp(t_0,\vec\zeta)\\[1ex]
&\uld{\Mphp}(t,\vec\zeta+{\vec k}(t)-{\vec k}(t_0))\dot{=}
e^{-\int_{t_0}^t((\vec\zeta+\vec k(\sigma)-{\vec k}(t_0))\cdot D_0(\vec\zeta+\vec k(\sigma)-{\vec k}(t_0))+R_{2,0}^*)\, d\sigma}\,\uld{\Mphp}(t_0,\vec\zeta)
\\&-\int_{t_0}^t e^{-\int_a^t((\vec\zeta+\vec k(\sigma)-{\vec k}(t_0))\cdot D_0(\vec\zeta+\vec k(\sigma)-{\vec k}(t_0))+R_{2,0}^*)\, d\sigma}
[\widehat{\uld{R_2^*}}*{\Mphp}(a)](\vec\zeta+\vec k(a)-{\vec k}(t_0))\, da\\[1ex]
&\text{ for }\vec\zeta\in\mathbb{R}^3, \quad t>t_0,
\end{aligned}
\end{equation*}
\end{comment}
\begin{equation*}
\begin{aligned}
&\Mphp(a,\vec\zeta)\dot{=}
e^{-\int_{t_0}^a(\Dterm{(\vec\zeta+{\vec k}(a)-\vec k(\sigma))}+R_{2,0}^*)\, d\sigma}\,\Mphp(t_0,\vec\zeta+{\vec k}(a)-\vec k(t_0))\\[1ex]
&\uld{\Mphp}(t,\vec\xi)\dot{=}
e^{-\int_{t_0}^t(\Dterm{(\vec\xi+{\vec k}(t)-\vec k(\sigma))}+R_{2,0}^*)\, d\sigma}\,\uld{\Mphp}(t_0,\vec\xi+{\vec k}(t)-{\vec k}(t_0))
\\&-\int_{t_0}^t e^{-\int_a^t(\Dterm{(\vec\xi+{\vec k}(t)-\vec k(\sigma))}+R_{2,0}^*)\, d\sigma}
[\widehat{\uld{R_2^*}}*{\Mphp}(a)](\vec\xi+{\vec k}(t)-{\vec k}(a))\, da\\[1ex]
&\text{ for }\vec\xi\in\mathbb{R}^3, \quad t>t_0,
\end{aligned}
\end{equation*}
where
\begin{itemize}
\item with $\pp={\pp}_{90^\circ}$:
\begin{equation*}
\begin{aligned}
&t_0=\taup , \quad
\Mphp(t_0,\cdot)\dot{=}-i\widehat{M^{eq}}, \quad
\uld{\Mphp}(t_0,\cdot)\dot{=}-i\widehat{\uld{M}^{eq}}
\end{aligned}
\end{equation*}
\item with $\pp={\pp}_{180^\circ-\tau-90^\circ}$:
\begin{equation*}
\begin{aligned}
&t_0=\tau+2\taup, \quad
\Mphp(t_0,\cdot)\dot{=}
%-i\mathcal{F}\left[\mathbb{L}_\tau M^{eq}\right]
-i\widehat{\mathbb{L}_\tau M^{eq}}, \\
&\uld{\Mphp}(t_0,\cdot)\dot{=}
-i\Bigl(\widehat{\mathbb{L}_\tau \uld{M}^{eq}}
-\mathcal{F}\left[\int_\taup^{\tau+\taup}e^{-\mathcal{A}_1(\tau+\taup-s)}\uld{R}_1(\mathbb{L}_{s-\taup}-1)M^{eq}
\right]\Bigr)
\end{aligned}
\end{equation*}
\end{itemize}
Putting everything together yields 
\begin{equation*}%\label{FIprimeD0_appendix}
\begin{aligned}
&\frac{1}{c_{j,0}^-}\Bigl(F_{(I)j}'(x^{img}_{\text{ref}})\uld{x}^{img}\Bigr)(t)= 
\uld{\Mphp}(t,0) \text{ with }\pp={\pp}_{90^\circ}\\
&\dot{=}
-i\,e^{-\int_{\taup}^t(\Dterm{({\vec k}(t)-\vec k(\sigma))}+R_{2,0}^*)\, d\sigma}\,
\widehat{\uld{M}^{eq}}({\vec k}(t)-{\vec k}(\taup))
\\&+i\,\int_{\taup}^t e^{-\int_a^t(\Dterm{({\vec k}(t)-\vec k(\sigma))}+R_{2,0}^*)\, d\sigma}
\int_{\mathbb{R}^3}\widehat{\uld{R_2^*}}({\vec k}(t)-{\vec k}(a)-\zeta)\\
&\qquad\qquad e^{-\int_{\taup}^a(\Dterm{(\vec\zeta+{\vec k}(a)-\vec k(\sigma))} +R_{2,0}^*)\, d\sigma}\,\widehat{M^{eq}}(\vec\zeta+{\vec k}(a)-\vec k(\taup))
\,d\zeta\, da
\end{aligned}
\end{equation*}
\begin{equation*}%\label{FIIprimeD0_appendix}
\begin{aligned}
&\frac{1}{c_{j,0}^-}\Bigl(F_{(II)j}'(x^{img}_{\text{ref}})\uld{x}^{img}\Bigr)(t)= 
\uld{\Mphp}(t,0) \text{ with }\pp={\pp}_{180^\circ-\tau-90^\circ}\\
&\dot{=}
-i\,e^{-\int_{\tau+2\taup}^t(\Dterm{({\vec k}(t)-\vec k(\sigma))}+R_{2,0}^*)\, d\sigma}\\
&\qquad\qquad\Bigl(\widehat{\mathbb{L}_\tau \uld{M}^{eq}}
-\mathcal{F}\left[\int_\taup^{\tau+\taup}e^{-\mathcal{A}_1(\tau+\taup-s)}\uld{R}_1(\mathbb{L}_{s-\taup}-1)M^{eq}\right]\Bigr)({\vec k}(t)-{\vec k}(\tau+2\taup))
\\&+i\,\int_{\tau+2\taup}^t e^{-\int_a^t(\Dterm{({\vec k}(t)-\vec k(\sigma))}+R_{2,0}^*)\, d\sigma}
\int_{\mathbb{R}^3}\widehat{\uld{R_2^*}}({\vec k}(t)-{\vec k}(a)-\zeta)\\
&\qquad\qquad e^{-\int_{\tau+2\taup}^a(\Dterm{(\vec\zeta+{\vec k}(a)-\vec k(\sigma))} +R_{2,0}^*)\, d\sigma}\,\widehat{\mathbb{L}_\tau M^{eq}}(\vec\zeta+{\vec k}(a)-\vec k(\tau+2\taup))
\,d\zeta\, da
\end{aligned}
\end{equation*}
and thus, with $\zeta'=\zeta+{\vec k}(a)$ and $\vec k(\taup)=\vec k(\tau+2\taup)=0$, leads to  \eqref{FIprimeD0}, \eqref{FIIprimeD0}.
 %for arxiv version; take this from arxiv_2025-06-16
\section*{Acknowledgment}
This research was funded in part by the Austrian Science Fund (FWF) 
[10.55776/F100800]. 
%For open access purposes, the author has applied a CC BY public copyright license to any author accepted manuscript version arising from this submission

%\bibliographystyle{plain}
%\bibliography{lit}
\end{document}